\documentclass[11pt,reqno]{amsart}

\usepackage{amsmath}
\usepackage{amssymb}
\usepackage{amsthm}

\newtheorem{theorem}{Theorem}[section]
\newtheorem{prop}[theorem]{Proposition}
\newtheorem{lem}[theorem]{Lemma}
\newtheorem*{cor}{Corollary}
\theoremstyle{definition}
\newtheorem{defn}[theorem]{Definition}

\theoremstyle{remark}
\newtheorem*{rem}{Remark}
\numberwithin{equation}{section}

\begin{document}

\title[Locally compact quantum groupoids]
{A class of {$C^*$-algebraic} locally compact quantum groupoids \\
{Part II}. {M}ain theory}

\author{Byung-Jay Kahng}
\date{}
\address{Department of Mathematics and Statistics\\ Canisius College\\
Buffalo, NY 14208, USA}
\email{kahngb@canisius.edu}

\author{Alfons Van Daele}
\date{}
\address{Department of Mathematics\\ University of Leuven\\ Celestijnenlaan 200B\\ 
B-3001 Heverlee, BELGIUM}
\email{Alfons.VanDaele@wis.kuleuven.be}


\begin{abstract}
This is Part~II in our multi-part series of papers developing the theory of a subclass 
of locally compact quantum groupoids ({\em quantum groupoids of separable type\/}), 
based on the purely algebraic notion of weak multiplier Hopf algebras.  The definition 
was given in Part~I.  The existence of a certain canonical idempotent element $E$ plays 
a central role.  In this Part~II, we develop the main theory, discussing the structure 
of our quantum groupoids.  We will construct from the defining axioms the right/left 
regular representations and the antipode map. 
\end{abstract}
\maketitle

{\sc Introduction.}

In Part~I of this series \cite{BJKVD_qgroupoid1}, we proposed a definition for a class 
of locally compact quantum groupoids.  This definition (see Definition~\ref{definitionlcqgroupoid} 
below) is motivated by the purely algebraic notion of {\em weak multiplier Hopf algebras\/}, 
introduced by one of us (Alfons Van Daele) and Shuanhong Wang \cite{VDWangwha0}, 
\cite{VDWangwha1}. 

A fundamental example of a weak multiplier Hopf algebra is the algebra $A=K(G)$, where 
$G$ is a groupoid and $K(G)$ is the set of all complex-valued functions on $G$ having 
finite support.  Here, the comultiplication map is not necessarily non-degenerate, while 
the existence of a certain {\em canonical idempotent\/} element $E\in M(A\otimes A)$ is assumed, 
which coincides with $\Delta(1)$ in the unital case. See Introduction of the previous paper 
(Part~I) for further motivation.

In our $C^*$-algebraic framework, building on the authors' previous work at the purely algebraic 
level \cite{BJKVD_LSthm}, we consider a subclass of {\em locally compact quantum groupoids\/} 
determined by the data $(A,\Delta,E,B,\nu,\varphi,\psi)$, where (i) $A$ is a $C^*$-algebra, 
taking the role of the groupoid $G$; (ii) $\Delta:A\to M(A\otimes A)$ is the comultiplication, 
corresponding to the multiplication on $G$; (iii) $B$ is a sub-$C^*$-algebra of $M(A)$, 
taking the role of the unit space $G^{(0)}$; (iv) $\nu$ is a weight on $B$; (v) $E$ is 
the canonical idempotent element motivated by the theory of weak multiplier Hopf algebras; 
and finally, (vi) $\varphi$ and $\psi$ are the left and right invariant Haar weights, 
respectively. 

This framework is closely related to the notion of {\em measured quantum groupoids\/}, 
developed in the von Neumann algebra setting \cite{LesSMF}, \cite{EnSMF}, though not 
quite as general.  It is because requiring the existence of a canonical idempotent 
element restricts the choice of the subalgebra $B$ and the weight $\nu$ on it. 
As our framework is a subclass, we plan to use the terminology {\em locally compact 
quantum groupoids of separable type\/} from now on (The authors thank the referee 
for this suggestion.).

On the other hand, our framework is less burdened with the technical difficulties that 
accompany the measured quantum groupoids, which is based on the notion of a ``fiber product'' 
of von Neumann algebras over a ``relative tensor product'' of Hilbert spaces.  There is 
really no corresponding notion for a fiber product in the $C^*$-algebra theory, and 
therefore, there has been no $C^*$-algebraic theory of quantum groupoids developed for 
non-unital (non-compact) cases.  Our framework, while restrictive, achieves this. 
In addition, as can be seen in what follows, the theory is rich in nature.  The authors 
believe that the current framework will provide a nice bridge, until a more general theory 
of $C^*$-algebraic locally compact quantum groupoids can be developed in the future 
(at the purely algebraic level, some progress has been made in this direction: \cite{TTVD}, 
\cite{TTVD2}).

Here is how this paper is organized.  In Section~\ref{sec1}, we review the notations and state 
our main definition, revisiting Definition~4.8 of Part~I \cite{BJKVD_qgroupoid1}.  As this 
paper is a continuation of the previous paper (Part~I), we refer to that paper for details 
and proofs. 

In Section~2, we further investigate the consequences of the right/left invariance conditions 
of our Haar weights $\psi$, $\varphi$.  We construct four maps, $Q_R$, $Q_{\rho}$, 
$Q_L$, $Q_{\lambda}$, which will play useful roles throughout the paper.

In Section~3, we construct the partial isometries $V$ and $W$, using the invariance properties 
developed in Section~2.  These operators are essentially the right and the left regular representations, 
and they play similar roles as the multiplicative unitary operators in the quantum group case. 

In Section~4, we carry out the construction of the antipode map $S$.  We first construct a Hilbert 
space operator $K$ implementing the antipode.  Both the right and the left Haar weights play 
significant roles.  Similar to the quantum group case, our antipode map $S$ would be defined 
in terms of its polar decomposition.  As the operators $V$ and $W$ are no longer unitaries, 
however, the arguments need to be modified and generalized accordingly.   

The discussion on the antipode map is continued in Section~5, where we collect some 
useful formulas and properties involving $S$ (antipode), $\sigma$ (modular group for 
$\varphi$), $\tau$ (``scaling group'').  Along the way, we obtain some alternative 
characterizations of the antipode map, which do not explicitly involve the weights $\psi$ 
or $\varphi$. We also explore the restrictions of $S$, $\sigma$, $\sigma^{\psi}$, $\tau$ 
to the base algebra level, gathering some useful results.  As a consequence, we establish 
that the weight $\nu$ on $B$ is quasi-invariant.

With the construction of the regular representations and the antipode, we can say that 
what we have is indeed a valid (though restrictive) framework for locally compact quantum groupoids. 
This framework contains all {\em locally compact quantum groups\/} (\cite{KuVa}, \cite{KuVavN}, 
\cite{MNW}, \cite{VDvN}), while extending the notions like {\em weak Hopf $C^*$-algebras\/} 
(\cite{BNSwha1}, \cite{BSwha2}), {\em generalized Kac algebras\/} (\cite{Yagroupoid}), 
and {\em finite quantum groupoids\/} (\cite{NVfqg}, \cite{Valfqg}).  Some other examples include 
the {\em face algebras\/} (\cite{Hsh_fa}), the {\em linking quantum groupoids\/} (\cite{DeComm_galois}, 
\cite{DeComm_LinkingQG}), or the {\em partial compact quantum groups\/} (\cite{DCTT_partialCQG}, 
\cite{DC_C*partialCQG}).  See section~5 of Part~I for some discussion on these examples.

The series will continue in Part~III \cite{BJKVD_qgroupoid3}.  In that paper, we plan to construct 
the dual object, which is also a locally compact quantum groupoid of our type.  We aim to show 
there that our class of quantum groupoids is self-dual, by obtaining a Pontryagin-type duality result.

\bigskip

{\sc Acknowledgments.}

This project began during the first named author (Byung-Jay Kahng)'s visit to University 
of Leuven during 2012/2013.  He is very much grateful to his coauthor (Alfons Van~Daele) 
and the mathematics department at KU Leuven, for their warm support 
and hospitality during his stay.   He also wishes to thank Michel Enock for the hospitality 
he received while he was visiting Jussieu.  The discussions on measured quantum groupoids, 
together with his encouraging words, were all very inspiring.  Discussions with Erik Koelink 
(Radboud) was helpful as he suggested a possible future application toward dynamical 
quantum groupoids. Two occasions to visit Albert Sheu (Kansas) were also helpful, who 
was very  supportive from the beginning of the project. Finally, he also acknowledges 
Thomas Timmermann (M{\"u}nster), as an on-going partner in pursuing the generalizations 
and applications of quantum groupoids. The numerous discussions helped polishing the current 
project as well as shaping the future ones.  Kahng's sabbatical was supported by the 
Fulbright U.S. Scholar program.  This work was partially supported by the Simons 
Fundation grant 346300 and the Polish Government MNiSW 2015-2019 matching fund.

\bigskip

\section{Definition of a locally compact quantum groupoid}\label{sec1}

\subsection{Separability idempotent}\label{sub1.1}

We gather here some results concerning the notion of a {\em separability idempotent\/}, 
which plays a central role in our theory.  We skip the proofs.  Instead, we refer the reader to sections~2 
and 3 of Part~I \cite{BJKVD_qgroupoid1}, as well as the authors' earlier paper on the subject 
\cite{BJKVD_SepId}.

Let $B$ be a $C^*$-algebra and $\nu$ a KMS weight on $B$.  This means that $\nu$ is faithful, 
lower semi-continuous, and semi-finite, together with a certain one-parameter group of automorphisms 
$(\sigma^{\nu}_t)_{t\in\mathbb{R}}$, the ``modular automorphism group'', satisfying the KMS property. 
See \cite{Tk2}.  See also section~1 of Part~I.  Consider another $C^*$-algebra $C\cong B^{\operatorname{op}}$, 
together with a ${}^*$-anti-isomorphism $R=R_{BC}:B\to C$.  

A self-adjoint idempotent element $E\in M(B\otimes C)$ 
is referred to as a {\em separability idempotent\/}, if 
$$(\nu\otimes\operatorname{id})(E)=1 \quad{\text { and }}\quad
(\nu\otimes\operatorname{id})\bigl(E(b\otimes1)\bigr)=R\bigl(\sigma^{\nu}_{i/2}(b)\bigr),
\forall b\in{\mathcal D}(\sigma^{\nu}_{i/2}).$$
See Definition~2.1 in Part~I.  Using the ${}^*$-anti-isomorphism $R$, we can also define a KMS weight 
$\mu$ on $C$, by $\mu=\nu\circ R^{-1}$.  We have $\sigma^{\mu}_t=R\circ\sigma^{\nu}_{-t}\circ R^{-1}$, 
and $(\operatorname{id}\otimes\mu)(E)=1$.

Consider a densely-defined map $\gamma_B:B\to C$, such that ${\mathcal D}(\gamma_B)={\mathcal D}
(\sigma^{\nu}_{i/2})$ and defined by $\gamma_B(b):=(R\circ\sigma^{\nu}_{i/2})(b)=(\sigma^{\mu}_{-i/2}\circ R)(b)$. 
It is a closed injective map, has a dense range, and is anti-multiplicative: $\gamma_B(bb')=\gamma_B(b')
\gamma_B(b)$.  Its inverse map is $\gamma_B^{-1}:C\to B$, such that ${\mathcal D}(\gamma_B^{-1})
=\operatorname{Ran}(\gamma_B)={\mathcal D}(\sigma^{\mu}_{i/2})$ and given by $\gamma_B^{-1}
=\sigma^{\nu}_{-i/2}\circ R^{-1}=R^{-1}\circ\sigma^{\mu}_{i/2}$.  It is also a closed, densely-defined, injective 
map, has a dense range, and is anti-multiplicative.

Similarly, there exists a closed, densely-defined, injective map $\gamma_C$ from $C$ into $B$, 
such that ${\mathcal D}(\gamma_C)={\mathcal D}(\sigma^{\mu}_{-i/2})$ and given by $\gamma_C
=R^{-1}\circ\sigma^{\mu}_{-i/2}=\sigma^{\nu}_{i/2}\circ R^{-1}$.  It also has a dense range and 
is anti-multiplicative.  Its inverse map is $\gamma_C^{-1}=\sigma^{\mu}_{i/2}\circ R=R\circ\sigma^{\nu}_{-i/2}$, 
again closed, densely-defined, injective, having a dense range, and anti-multiplicative.

\begin{lem}\label{gammagamma'}
The maps $\gamma_B:B\to C$ and $\gamma_C:C\to B$ defined above are closed, densely-defined, 
injective anti-homomorphisms, having dense ranges.  Moreover, 
\begin{enumerate}
 \item The $\gamma_B$, $\gamma_C$ maps satisfy the following relations:
\begin{align}
&E(b\otimes1)=E\bigl(1\otimes\gamma_B(b)\bigr),\forall b\in{\mathcal D}(\gamma_B),  \notag \\
&E(1\otimes c)=E\bigl(\gamma_B^{-1}(c)\otimes1\bigr),\forall c\in{\mathcal D}(\gamma_B^{-1}),  
\notag \\
&(1\otimes c)E=\bigl(\gamma_C(c)\otimes1\bigr)E,\forall c\in{\mathcal D}(\gamma_C),  \notag \\
&(b\otimes1)E=\bigl(1\otimes\gamma_C^{-1}(b)\bigr)E,\forall b\in{\mathcal D}(\gamma_C^{-1}).
\notag
\end{align}
 \item We have: $\gamma_C\bigl(\gamma_B(b)^*\bigr)^*=b$, for $b\in{\mathcal D}(\gamma_B)$, 
and $\gamma_B\bigl(\gamma_C(c)^*\bigr)^*=c$, for $c\in{\mathcal D}(\gamma_C)$.
\end{enumerate}
\end{lem}

\begin{proof} 
See Propositions~2.4, 2.5, 2.7 in Part~I.  For more details, see \cite{BJKVD_SepId}.
\end{proof}

It can be shown that $E\in M(B\otimes C)$ is ``full'' (see Propositions~4.3, 4.4, 4.5 
in \cite{BJKVD_SepId}), satisfying certain density conditions.  See also Proposition~2.8 
of Part~I.  This means that the left leg of $E$ is $B$ and the right leg of $E$ is $C$.  

We have: $(\sigma^{\nu}_t\otimes\sigma^{\mu}_{-t})(E)=E$, $\forall t$ (see Proposition~2.10 
in Part~I). Meanwhile, it turns out that $(\gamma_B\otimes\gamma_C)(E)$ is bounded, 
and that $(\gamma_B\otimes\gamma_C)(E)=\varsigma E$ and $(\gamma_C\otimes\gamma_B)
(\varsigma E)=E$, where $\varsigma$ denotes the flip map on $M(B\otimes C)$.  As a consequence, 
we also have: $(R^{-1}\otimes R)(\varsigma E)=E$ and $(R\otimes R^{-1})(E)=\varsigma E$.  
See Proposition~2.12 in Part~I.

The separability idempotent condition for $E$ is actually a condition on the pair $(B,\nu)$, 
as $E$ is uniquely determined by $(B,\nu)$.  See Proposition~2.2 in Part~I. It can be 
shown that $B$ must be postliminal, and the choice of $\nu$ cannot be arbitrary.  
For more details on all these, please refer to \cite{BJKVD_SepId}.

\subsection{Definition: Locally compact quantum groupoid of separable type}

In what follows, we re-state Definition~4.8 from Part~I \cite{BJKVD_qgroupoid1}. 
See paragraphs following the definition for some additional comments.

\begin{defn}\label{definitionlcqgroupoid}
The data $(A,\Delta,E,B,\nu,\varphi,\psi)$ defines a {\em locally compact quantum 
groupoid of separable type\/}, if
\begin{itemize}
  \item $A$ is a $C^*$-algebra.
  \item $\Delta:A\to M(A\otimes A)$ is a comultiplication on $A$.
  \item $B$ is a non-degenerate $C^*$-subalgebra of $M(A)$.
  \item $\nu$ is a KMS weight on $B$.
  \item $E$ is the canonical idempotent of $(A,\Delta)$ as in 
Definition~3.7 of Part~I.  That is,
  \begin{enumerate}
    \item $\Delta(A)(A\otimes A)$ is dense in $E(A\otimes A)$ and $(A\otimes A)\Delta(A)$ 
is dense in $(A\otimes A)E$;
    \item there exists a $C^*$-subalgebra $C\cong B^{\operatorname{op}}$ contained in $M(A)$, 
with a ${}^*$-anti-isomorphism $R=R_{BC}:B\to C$, so that $E\in M(B\otimes C)$ and the triple 
$(E,B,\nu)$ forms a separability triple (see subsection~\S1.1);
    \item $E\otimes1$ and $1\otimes E$ commute, and we have:
$$
(\operatorname{id}\otimes\Delta)(E)=(E\otimes1)(1\otimes E)=(1\otimes E)(E\otimes1)
=(\Delta\otimes\operatorname{id})(E).
$$
  \end{enumerate}
  \item $\varphi$ is a KMS weight, and is left invariant.
  \item $\psi$ is a KMS weight, and is right invariant.
  \item There exists a (unique) one-parameter group of automorphisms 
$(\theta_t)_{t\in\mathbb{R}}$ of $B$ such that $\nu\circ\theta_t=\nu$ and that 
$\sigma^{\varphi}_t|_B=\theta_t$, $\forall t\in\mathbb{R}$.
\end{itemize}
\end{defn}

In the above, the comultiplication $\Delta$ is a ${}^*$-homomorphism from $A$ into $M(A\otimes A)$, 
not necessarily non-degenerate.  It is assumed to satisfy a ``weak coassociativity'' condition, and 
there is also a certain density condition ($\Delta$ is ``full'').  See Definition~3.1 and Lemma~3.2 
in Part~I.

Without the non-degeneracy, there is no straightforward way of extending $\Delta$ to $M(A)$. 
However, it turns out that Condition\,(1) for the idempotent element $E$ can be used to make 
the extension possible, and the following coassociativity condition is shown to hold: 
\begin{equation}\label{(coassociativity)}
(\Delta\otimes\operatorname{id})(\Delta a)=(\operatorname{id}\otimes\Delta)(\Delta a),\quad
\forall a\in A.
\end{equation}
For proof, see Theorem~3.5 in Part~I.  

Among other consequences of Condition\,(1) on $E$, we note that 
\begin{equation}\label{(EDelta)}
E(\Delta a)=\Delta a=(\Delta a)E,\quad\forall a\in A.
\end{equation}
See Proposition~3.3 in Part~I.  In fact, $E=\Delta(1_{M(A)})$, the image of $1\in M(A)$ under 
the extended comultiplication map.  However, without the non-degeneracy of the comultiplication 
map, we have $E\ne1\otimes1$ in general. 

As for Condition\,(2), the ``separability idempotent condition'', refer to the brief discussion 
given in the previous subsection.

Condition\,(3) for $E$ is referred to as the ``weak comultiplicativity of the unit''.
Parts of it can be proved (for instance, Proposition~3.6 in Part~I) and it is automatically true 
in the ordinary groupoid case with $E=\Delta(1)$, but the condition as a whole does not follow 
from other axioms.

As a consequence of the existence of the canonical idempotent $E$, we can show that 
the $C^*$-algebras $B$ and $C$ commute, and so do $M(B)$ and $M(C)$.  See 
Proposition~3.8 in Part~I.  We also have:
\begin{equation}\label{(DeltaonB)}
\Delta y=E(1\otimes y)=(1\otimes y)E,\quad\forall y\in M(B),
\end{equation}
\begin{equation}\label{(DeltaonC)}
\Delta x=(x\otimes1)E=E(x\otimes1),\quad\forall x\in M(C).
\end{equation}
See Proposition~3.9 and its Corollary in Part~I.  The canonical idempotent is uniquely 
determined by the Conditions\,(1), (2), (3) above.

Meanwhile, the left/right invariance of $\varphi$ and $\psi$ mean the following:
\begin{itemize}
  \item For any $a\in{\mathfrak M}_{\varphi}$, we have 
$\Delta a\in\overline{\mathfrak M}_{\operatorname{id}\otimes\varphi}$ and
$(\operatorname{id}\otimes\varphi)(\Delta a)\in M(C)$.
  \item For any $a\in{\mathfrak M}_{\psi}$, we have 
$\Delta a\in\overline{\mathfrak M}_{\psi\otimes\operatorname{id}}$ and
$(\psi\otimes\operatorname{id})(\Delta a)\in M(B)$.
\end{itemize}
See section~4 of Part~I for more discussion on the left/right invariance, and some other 
properties of $\varphi$ and $\psi$.  In particular, we note the following proposition. 
It shows the relationships between the weights $\varphi$, $\psi$ with the expressions 
$(\operatorname{id}\otimes\varphi)(\Delta x)$ and $(\psi\otimes\operatorname{id})(\Delta x)$, 
which are in fact operator-valued weights:

\begin{prop}\label{nuphipsi}
We have:
\begin{itemize}
  \item $\nu\bigl((\psi\otimes\operatorname{id})(\Delta x)\bigr)=\psi(x)$, 
for $x\in{\mathfrak M}_{\psi}$.
  \item $\mu\bigl((\operatorname{id}\otimes\varphi)(\Delta x)\bigr)=\varphi(x)$, 
for $x\in{\mathfrak M}_{\varphi}$.
\end{itemize}
\end{prop}

\begin{proof}
See Proposition~4.9 in Part~I.
\end{proof}

Finally, note the last condition of Definition~\ref{definitionlcqgroupoid}, concerning 
$\nu$.  In fact, it is possible to prove using the other axioms that each $\sigma^{\varphi}_t$ 
leaves $B$ invariant and that $(\sigma^{\varphi}_t|_B)_{t\in\mathbb{R}}$ is an automorphism 
group of $B$.  In the definition, we are further requiring that $\nu$ is invariant under 
$\sigma^{\varphi}_t|_B$ (that is, $\nu\circ\sigma^{\varphi}_t|_B=\nu$).  This additional 
condition is necessary later, when we show that $\nu$ is quasi-invariant.  More discussion 
on this aspect will be given in Section~\ref{sec5} below.

\section{Some alternative formulations of the left/right invariance}\label{sec2}

Fix a quantum groupoid of separable type $(A,\Delta,E,B,\nu,\varphi,\psi)$, 
as given in Definition~\ref{definitionlcqgroupoid}.  In what follows, the left/right invariance 
properties of $\varphi$, $\psi$ are needed in a fundamental way.  In preparation, 
it is quite helpful to introduce certain densely-defined maps on $A\otimes A$.

\subsection{Definition of the maps $Q_R$, $Q_{\rho}$, $Q_L$, $Q_{\lambda}$}

Let us begin with a lemma, which is a consequence of the weights $\varphi$, 
$\psi$ being faithful.

\begin{lem}\label{densesubsetA}
The following subspaces are norm-dense in $A$:

$\operatorname{span}\bigl\{(\operatorname{id}\otimes\varphi)(\Delta(a^*)(1\otimes x)):
a,x\in{\mathfrak N}_{\varphi}\bigr\}$,

$\operatorname{span}\bigl\{(\operatorname{id}\otimes\varphi)((1\otimes x^*)(\Delta a)):
a,x\in{\mathfrak N}_{\varphi}\bigr\}$,

$\operatorname{span}\bigl\{(\psi\otimes\operatorname{id})(\Delta(a^*)(y\otimes1)):
a,y\in{\mathfrak N}_{\psi}\bigr\}$,

$\operatorname{span}\bigl\{(\psi\otimes\operatorname{id})((y^*\otimes1)(\Delta a)):
a,y\in{\mathfrak N}_{\psi}\bigr\}$.
\end{lem}

\begin{proof}
For $a\in{\mathfrak N}_{\varphi}$, we know 
$\Delta a\in\overline{\mathfrak N}_{\operatorname{id}\otimes\varphi}$ by the left invariance 
of $\varphi$ (See Proposition~4.2 in Part~I.).  For any $x\in{\mathfrak N}_{\varphi}$, 
we have $\Delta(a^*)(1\otimes x)\in{\mathfrak M}_{\operatorname{id}\otimes\varphi}$, 
so  we can see that $(\operatorname{id}\otimes\varphi)\bigl(\Delta(a^*)(1\otimes x)\bigr)\in A$.

Let $\theta\in A^*$ be such that $\theta\bigl((\operatorname{id}\otimes\varphi)
(\Delta(a^*)(1\otimes x))\bigr)=0$, $\forall a,x\in{\mathfrak N}_{\varphi}$. 
But, $\theta\bigl((\operatorname{id}\otimes\varphi)(\Delta(a^*)(1\otimes x))\bigr)
=\varphi\bigl((\theta\otimes\operatorname{id})(\Delta(a^*)(1\otimes x))\bigr)
=\varphi\bigl((\theta\otimes\operatorname{id})(\Delta(a^*))x\bigr)$.  Since 
$\varphi$ is faithful, and since $x$ is arbitrary in ${\mathfrak N}_{\varphi}$, 
a dense subspace of $A$, it follows that
$(\theta\otimes\operatorname{id})\bigl(\Delta(a^*)\bigr)x=0$, or 
$(\theta\otimes\operatorname{id})\bigl(\Delta(a^*)(1\otimes x)\bigr)=0$,
for any $a,x\in{\mathfrak N}_{\varphi}$.
So, for any $\omega\in A^*$ and $a,x\in{\mathfrak N}_{\varphi}$, we will have:
\begin{equation}\label{(densesubsetA_eqn1)}
\theta\bigl((\operatorname{id}\otimes\omega)(\Delta(a^*)(1\otimes x))\bigr)
=\omega\bigl((\theta\otimes\operatorname{id})(\Delta(a^*)(1\otimes x))\bigr)
=0.
\end{equation}

By definition, the comultiplication $\Delta$ is ``full'' (see Definition~3.1 and Lemma~3.2 
in Part~I).  Together with the fact that ${\mathfrak N}_{\varphi}$ is dense in $A$, 
this means that the elements $(\operatorname{id}\otimes\omega)(\Delta(a^*)(1\otimes x))$, 
for $a,x\in{\mathfrak N}_{\varphi}$, span a dense subspace in $A$.  Therefore, 
Equation~\eqref{(densesubsetA_eqn1)} shows that actually $\theta\equiv0$. 

By the Hahn--Banach argument, we can thus conclude that the elements 
$(\operatorname{id}\otimes\varphi)\bigl(\Delta(a^*)(1\otimes x)\bigr)$, 
$a,x\in{\mathfrak N}_{\varphi}$, span a dense subspace in $A$.  The other 
three density results can be proved in the same way.
\end{proof}

Now let $b\in A$ be arbitrary, and consider 
$p=(\operatorname{id}\otimes\varphi)\bigl(\Delta(a^*)(1\otimes x)\bigr)$, 
for $a,x\in{\mathfrak N}_{\varphi}$.  We just saw from Lemma~\ref{densesubsetA} 
that such elements span a dense subset in $A$.  For $p\otimes b
=(\operatorname{id}\otimes\operatorname{id}\otimes\varphi)\bigl(\Delta_{13}(a^*)
(1\otimes b\otimes x)\bigr)$, define:
\begin{align}\label{(defQ_R)}
Q_R(p\otimes b)&=Q_R\bigl((\operatorname{id}\otimes\operatorname{id}\otimes\varphi)
(\Delta_{13}(a^*)(1\otimes b\otimes x))\bigr)  \notag \\
&:=(\operatorname{id}\otimes\operatorname{id}\otimes\varphi)\bigl(\Delta_{13}(a^*)
(1\otimes E)(1\otimes b\otimes x)\bigr).
\end{align}

By the left invariance of $\varphi$, we have $\Delta_{13}(a)
\in{\mathfrak N}_{\operatorname{id}\otimes\operatorname{id}\otimes\varphi}$.  
Since $x\in{\mathfrak N}_{\varphi}$, the expression in the right side is valid. 
Nevertheless, at this stage we do not know whether $Q_R$ determines a well-defined map. 
So, to motivate the choice of $Q_R$ as well as to show that it is indeed a well-defined map, 
let us, for the time being, assume that $b\in{\mathcal T}_{\nu}$. Here, ${\mathcal T}_{\nu}$ 
is the Tomita subalgebra of $B$, which is dense in $B$.

Consider the following lemma, where we are using the $\gamma_B$, $\gamma_C$ maps 
we reviewed in subsection~\ref{sub1.1}.

\begin{lem}
Suppose $b\in{\mathcal T}_{\nu}$.  Then $b\in{\mathcal D}(\gamma_B)\cap{\mathcal D}(\gamma_C^{-1})$, 
and we have: 
$$
E_{13}(1\otimes E)(1\otimes b\otimes1)
=E_{13}\,(\operatorname{id}\otimes\gamma_C\otimes\operatorname{id})
\bigl((1\otimes\gamma_C^{-1}(b)\otimes1)(E\otimes1)\bigr).
$$
\end{lem}

\begin{proof}
By the property of $E$, it is known that $(\gamma_C\otimes\gamma_B)(\varsigma E)=E$, 
where $\varsigma$ denotes the flip map.  Since $b\in D(\gamma_B)$, and since the maps 
$\gamma_B$, $\gamma_C$ are anti-homomorphisms, we thus have: 
$$
(\gamma_C\otimes\gamma_B)\bigl((1\otimes b)(\varsigma E)\bigr)=
E\bigl(1\otimes\gamma_B(b)\bigr)=E(b\otimes1).
$$
The last equality is using Lemma~\ref{gammagamma'}  It follows that 
$$
E_{13}(1\otimes E)(1\otimes b\otimes1)
=E_{13}\,\bigl(1\otimes(\gamma_C\otimes\gamma_B)((1\otimes b)(\varsigma E))\bigr).
$$
Using again $E(b\otimes1)=E\bigl(1\otimes\gamma_B(b)\bigr)$, but this time 
for $E_{13}$, the right side of the equation becomes: 
$E_{13}\,(\operatorname{id}\otimes\gamma_C\otimes\operatorname{id})
\bigl((b\otimes1\otimes1)(E\otimes1)\bigr)$.  Since $b\in{\mathcal D}(\gamma_C^{-1})$ 
as well, we can now use $(b\otimes1)E=\bigl(1\otimes\gamma_C^{-1}(b)\bigr)E$. 
We thus have:
$$
E_{13}(1\otimes E)(1\otimes b\otimes1)
=E_{13}\,(\operatorname{id}\otimes\gamma_C\otimes\operatorname{id})
\bigl((1\otimes\gamma_C^{-1}(b)\otimes1)(E\otimes1)\bigr).
$$
\end{proof}

\begin{rem}
Formally, the right hand side of the equation in the Lemma may be written 
as $E_{13}\,\bigl((\operatorname{id}\otimes\gamma_C)(E)\otimes1\bigr)
(1\otimes b\otimes1)$.  Since $b$ is arbitrary, the result of the lemma 
would mean that 
$$
E_{13}(1\otimes E)=E_{13}(F_1\otimes1),
$$
where $F_1:=(\operatorname{id}\otimes\gamma_C)(E)$.  This exact formula  
appears in the weak multiplier Hopf algebra theory.  See Proposition~3.17 
in \cite{VDWangwha0} and Proposition~4.6 in \cite{VDWangwha1}.  In our case, 
unlike in the purely algebraic framework, the $F_1$ map would be unbounded 
(because $\gamma_C$ is), and so we needed to treat this with care.
\end{rem}

If we assume that $p\in A$ is as above and $b\in{\mathcal T}_{\nu}$, then 
Equation~\eqref{(defQ_R)} for $Q_R(p\otimes b)$ becomes:
\begin{align}
Q_R(p\otimes b)&=(\operatorname{id}\otimes\operatorname{id}\otimes\varphi)
\bigl(\Delta_{13}(a^*)(1\otimes E)(1\otimes b\otimes x)\bigr)  \notag \\
&=(\operatorname{id}\otimes\operatorname{id}\otimes\varphi)\bigl(\Delta_{13}(a^*)
E_{13}(1\otimes E)(1\otimes b\otimes1)(1\otimes1\otimes x)\bigr) 
\notag \\
&=(\operatorname{id}\otimes\gamma_C\otimes\varphi)\bigl(\Delta_{13}(a^*)
(1\otimes\gamma_C^{-1}(b)\otimes1)(E\otimes1)(1\otimes1\otimes x)\bigr)
\notag \\
&=(\operatorname{id}\otimes\gamma_C)\bigl((\operatorname{id}\otimes\operatorname{id}
\otimes\varphi)(\Delta_{13}(a^*)(1\otimes1\otimes x))(1\otimes\gamma_C^{-1}(b))E\bigr)
\notag \\
&=(\operatorname{id}\otimes\gamma_C)\bigl((p\otimes1)(1\otimes\gamma_C^{-1}(b))E\bigr).
\notag
\end{align}
The third equality is using the previous Lemma.  We also used the fact that $\Delta a
=(\Delta a)E$, as given in Equation~\eqref{(EDelta)}.  

Note that formally, with the notation for $F_1$ as in Remark above, the result 
can be written as: $Q_R(p\otimes b)=(p\otimes1)F_1(1\otimes b)$.  This observation 
explains where the definition of $Q_R$ comes from.  In addition, it is now easy to see 
that if $p\otimes b=0$, then $Q_R(p\otimes b)=0$.  So, for $p$ of the type given 
in Equation~\eqref{(defQ_R)} and $b\in{\mathcal T}_{\nu}$, we see that 
$Q_R:p\otimes b\mapsto Q_R(p\otimes b)$ is a well-defined linear map. 

We can actually go even further.  See next proposition:

\begin{prop}\label{Q_Rproposition}
Consider ${\mathcal D}(Q_R)$, which is dense in $A\otimes A$, as follows:
$$
{\mathcal D}(Q_R):=\operatorname{span}\bigl\{sp\otimes b:p=(\operatorname{id}\otimes\varphi)
(\Delta(a^*)(1\otimes x)),\,a,x\in{\mathfrak N}_{\varphi};\ s\in M(A);\ b\in A\bigr\}.
$$
Then:
\begin{enumerate}
  \item We can extend Equation~\eqref{(defQ_R)} to a well-defined map 
$Q_R:{\mathcal D}(Q_R)\to A\otimes A$, such that
\begin{align}
Q_R(sp\otimes b)&=(s\otimes1)(\operatorname{id}\otimes\operatorname{id}\otimes\varphi)
\bigl(\Delta_{13}(a^*)(1\otimes E)(1\otimes1\otimes x)\bigr)(1\otimes b) \notag \\
&=(\operatorname{id}\otimes\operatorname{id}\otimes\varphi)\bigl((s\otimes1\otimes1)
\Delta_{13}(a^*)(1\otimes E)(1\otimes b\otimes x)\bigr).
\notag
\end{align}
  \item The following properties hold:

$Q_R(sr\otimes b)=(s\otimes 1)\,Q_R(r\otimes b)$, for $s\in M(A)$ and 
$r\otimes b\in{\mathcal D}(Q_R)$.

$Q_R(r\otimes bd)=Q_R(r\otimes b)\,(1\otimes d)$, for $d\in M(A)$ and 
$r\otimes b\in{\mathcal D}(Q_R)$.
\end{enumerate}
\end{prop}

\begin{proof}
From the discussion given in the previous paragraph, we can see quickly that for 
$s\in M(A)$ and $d\in A$, the definition of $Q_R$ can be extended to elements 
of the form $sp\otimes bd\in A\otimes A$, by letting
$$
Q_R(sp\otimes bd)=(s\otimes1)Q_R(p\otimes b)(1\otimes d),
$$
where $p$ is of the given type and $b\in{\mathcal T}_{\nu}$.

Meanwhile, since ${\mathcal T}_{\nu}$ is dense in $B$ and since $B$ is a non-degenerate 
subalgebra of $M(A)$, any element of $A$ can be approximated by the elements of the form 
$bd$, for $b\in{\mathcal T}_{\nu}$, $d\in A$.  In fact, looking at the way the $Q_R$ map is 
defined, it is evident that it can be naturally defined for elements of the form $sp\otimes b
\in A\otimes A$, for any $b\in A$. Such elements form the dense domain ${\mathcal D}(Q_R)$ 
in $A\otimes A$.

Definition of $Q_R$ is now an immediate consequence of Equation~\eqref{(defQ_R)}. 
It is also clear from these discussions that the two properties in (2) hold.
\end{proof}

Observe now that $Q_R$ is an idempotent map:

\begin{prop}\label{Q_Ridempotent}
We have: $Q_RQ_R=Q_R$.
\end{prop}

\begin{proof}
Let $p=(\operatorname{id}\otimes\varphi)(\Delta(a^*)(1\otimes x))$, for 
$a,x\in{\mathfrak N}_{\varphi}$, and let $b\in A$.  We can define $Q_R(p\otimes b)$ 
as above.

Note next that $E(b\otimes x)\in A\otimes{\mathfrak N}_{\varphi}$, because 
${\mathfrak N}_{\varphi}$ is a left ideal.  From this, we can see that 
$Q_R(p\otimes b)\in{\mathcal D}(Q_R)$. Moreover, we have:
\begin{align}
Q_RQ_R(p\otimes b)&=Q_R\bigl((\operatorname{id}\otimes\operatorname{id}\otimes\varphi)
(\Delta_{13}(a^*)(1\otimes E)(1\otimes b\otimes x))\bigr)   \notag \\
&=Q_R\bigl((\operatorname{id}\otimes\operatorname{id}\otimes\varphi)
(\Delta_{13}(a^*)(1\otimes E)(1\otimes E)(1\otimes b\otimes x))\bigr)  \notag \\
&=Q_R(p\otimes b),   \notag
\end{align}
because $E^2=E$.  It follows that $Q_RQ_R=Q_R$.
\end{proof} 

In the next proposition, we write down a characterization of the map $Q_R$, which 
will be useful later:

\begin{prop}\label{Q_Rcharacterization}
If $a,x\in{\mathfrak N}_{\varphi}$ and $b\in A$, then
\begin{enumerate}
  \item $\Delta_{13}(a^*)(1\otimes b\otimes x)\in{\mathcal D}(Q_R\otimes\operatorname{id})$;
  \item $(Q_R\otimes\operatorname{id})\bigl(\Delta_{13}(a^*)(1\otimes b\otimes x)\bigr)
=\Delta_{13}(a^*)(1\otimes E)(1\otimes b\otimes x)$.
\end{enumerate}
\end{prop}

\begin{proof}
By Proposition~\ref{Q_Rproposition}, we know that $(\operatorname{id}\otimes\operatorname{id}
\otimes\varphi)\bigl(\Delta_{13}(a^*)(1\otimes b\otimes x)\bigr)$ is contained in ${\mathcal D}(Q_R)$. 
But, this immediately means that $\Delta_{13}(a^*)(1\otimes b\otimes x)
\in{\mathcal D}(Q_R\otimes\operatorname{id})$ and that $(Q_R\otimes\operatorname{id})
\bigl(\Delta_{13}(a^*)(1\otimes b\otimes x)\bigr)\in{\mathfrak M}_{\operatorname{id}\otimes
\operatorname{id}\otimes\varphi}$.  We have 
\begin{align}
&(\operatorname{id}\otimes\operatorname{id}\otimes\varphi)
\bigl((Q_R\otimes\operatorname{id})(\Delta_{13}(a^*)(1\otimes b\otimes x))\bigr)
\notag \\
&=Q_R\bigl((\operatorname{id}\otimes\operatorname{id}\otimes\varphi)
(\Delta_{13}(a^*)(1\otimes b\otimes x))\bigr)  \notag \\
&=(\operatorname{id}\otimes\operatorname{id}\otimes\varphi)
\bigl(\Delta_{13}(a^*)(1\otimes E)(1\otimes b\otimes x)\bigr).
\notag
\end{align}
This is true for any $x\in{\mathfrak N}_{\varphi}$, while $\varphi$ is a faithful weight. 
It follows that 
$$
(Q_R\otimes\operatorname{id})\bigl(\Delta_{13}(a^*)(1\otimes b\otimes x)\bigr)
=\Delta_{13}(a^*)(1\otimes E)(1\otimes b\otimes x).
$$
\end{proof}

We may modify Equation~\eqref{(defQ_R)} a little to define three other densely-defined 
idempotent maps $Q_{\rho}$, $Q_L$, $Q_{\lambda}$, which satisfy results analogous to 
Propositions~\ref{Q_Rproposition}, \ref{Q_Ridempotent}, \ref{Q_Rcharacterization} above.  
Since the method is essentially no different, we will skip the details and collect the results 
in three propositions below:

\begin{prop}\label{Q_R'}
\begin{enumerate}
  \item Let $p=(\operatorname{id}\otimes\varphi)\bigl((1\otimes x^*)(\Delta a)\bigr)$, 
for $a,x\in{\mathfrak N}_{\varphi}$, and let $b\in A$, $s\in M(A)$.  Define:
$$
Q_{\rho}(ps\otimes b):=(\operatorname{id}\otimes\operatorname{id}\otimes\varphi)
\bigl((1\otimes b\otimes x^*)(1\otimes E)\Delta_{13}(a)\bigr)\,(s\otimes1).
$$
This determines a densely-defined map on $A\otimes A$ into itself.
  \item The following properties hold:

$Q_{\rho}(rs\otimes b)=Q_{\rho}(r\otimes b)\,(s\otimes 1)$, for $s\in M(A)$ and 
$r\otimes b\in{\mathcal D}(Q_{\rho})$.

$Q_{\rho}(r\otimes db)=(1\otimes d)\,Q_{\rho}(r\otimes b)$, for $d\in M(A)$ and 
$r\otimes b\in{\mathcal D}(Q_{\rho})$.
  \item $Q_{\rho}$ is an idempotent map.  That is, $Q_{\rho}Q_{\rho}=Q_{\rho}$.
  \item If $a,x\in{\mathfrak N}_{\varphi}$ and $b\in A$, then $(1\otimes b\otimes x^*)
\Delta_{13}(a)\in{\mathcal D}(Q_{\rho}\otimes\operatorname{id})$, and we have:
$$
(Q_{\rho}\otimes\operatorname{id})\bigl((1\otimes b\otimes x^*)\Delta_{13}(a)\bigr)
=(1\otimes b\otimes x^*)(1\otimes E)\Delta_{13}(a).
$$
\end{enumerate}
\end{prop}

\begin{prop}\label{Q_L}
\begin{enumerate}
  \item Let $q=(\psi\otimes\operatorname{id})\bigl(\Delta(a^*)(y\otimes1)\bigr)$, 
for $a,y\in{\mathfrak N}_{\psi}$, and let $c\in A$, $s\in M(A)$.  Define:
$$
Q_L(c\otimes sq):=(1\otimes s)\,(\psi\otimes\operatorname{id}\otimes\operatorname{id})
\bigl(\Delta_{13}(a^*)(E\otimes1)(y\otimes c\otimes1)\bigr).
$$
This determines a densely-defined map on $A\otimes A$ into itself.
  \item The following properties hold:

$Q_L(c\otimes sr)=(1\otimes s)\,Q_L(c\otimes r)$, for $s\in M(A)$ and 
$c\otimes r\in{\mathcal D}(Q_L)$.

$Q_L(cd\otimes r)=Q_L(c\otimes r)\,(d\otimes1)$, for $d\in M(A)$ and 
$c\otimes r\in{\mathcal D}(Q_L)$.
  \item $Q_L$ is an idempotent map: $Q_LQ_L=Q_L$.
  \item If $a,y\in{\mathfrak N}_{\psi}$ and $c\in A$, then $\Delta_{13}(a^*)
(y\otimes c\otimes1)\in{\mathcal D}(\operatorname{id}\otimes Q_L)$, and we have:
$$
(\operatorname{id}\otimes Q_L)\bigl(\Delta_{13}(a^*)(y\otimes c\otimes1)\bigr)
=\Delta_{13}(a^*)(E\otimes1)(y\otimes c\otimes1).
$$
\end{enumerate}
\end{prop}

\begin{prop}\label{Q_L'}
\begin{enumerate}
  \item Let $q=(\psi\otimes\operatorname{id})\bigl((y^*\otimes1)(\Delta a)\bigr)$, 
for $a,y\in{\mathfrak N}_{\psi}$, and let $c\in A$, $s\in M(A)$.  Define:
$$
Q_{\lambda}(c\otimes qs):=(\psi\otimes\operatorname{id}\otimes\operatorname{id})
\bigl((y^*\otimes c\otimes1)(E\otimes1)\Delta_{13}(a)\bigr)\,(1\otimes s).
$$
This determines a densely-defined map on $A\otimes A$ into itself.
  \item The following properties hold:

$Q_{\lambda}(c\otimes rs)=Q_{\lambda}(c\otimes r)\,(1\otimes s)$, for $s\in M(A)$ and 
$c\otimes r\in{\mathcal D}(Q_{\lambda})$.

$Q_{\lambda}(dc\otimes r)=(d\otimes1)\,Q_{\lambda}(c\otimes r)$, for $d\in M(A)$ and 
$c\otimes r\in{\mathcal D}(Q_{\lambda})$.
  \item $Q_{\lambda}$ is an idempotent map: $Q_{\lambda}Q_{\lambda}=Q_{\lambda}$.
  \item If $a,y\in{\mathfrak N}_{\psi}$ and $c\in A$, then $(y^*\otimes c\otimes1)
\Delta_{13}(a)\in{\mathcal D}(\operatorname{id}\otimes Q_{\lambda})$, and we have:
$$
(\operatorname{id}\otimes Q_{\lambda})\bigl((y^*\otimes c\otimes1)\Delta_{13}(a)\bigr)
=(y^*\otimes c\otimes1)(E\otimes1)\Delta_{13}(a).
$$
\end{enumerate}
\end{prop}

These maps are very closely related, as one can imagine.  The following result will be 
useful later.

\begin{prop}\label{RR'LL'}
\begin{enumerate}
 \item Let $p=(\operatorname{id}\otimes\varphi)\bigl((1\otimes x^*)(\Delta a)\bigr)$, 
for $a,x\in{\mathfrak N}_{\varphi}$, and let $b\in A$.  Then we have:
$Q_{\rho}(p\otimes b)=Q_R(p^*\otimes b^*)^*$.
 \item Let $q=(\psi\otimes\operatorname{id})\bigl((y^*\otimes1)(\Delta a)\bigr)$, 
for $a,y\in{\mathfrak N}_{\psi}$, and let $c\in A$.  Then we have:
$Q_{\lambda}(c\otimes q)=Q_L(c^*\otimes q^*)^*$.
\end{enumerate}
\end{prop}

\begin{proof}
Straightforward.
\end{proof}

\begin{rem}
(1). A different way to look at all these results is that the case of $Q_{\rho}$ is 
when we instead work with the opposite $C^*$-algebra $A^{\operatorname{op}}$; 
$Q_L$ is when we work with $(A,\Delta^{\operatorname{cop}})$ so that the 
roles of the left/right Haar weights are reversed; $Q_{\lambda}$ is when we work 
with $(A^{\operatorname{op}},\Delta^{\operatorname{cop}})$.

(2). The four maps $Q_R$, $Q_{\rho}$, $Q_L$, $Q_{\lambda}$ are only densely-defined 
on $A\otimes A$.  For those readers who are familiar with the theory of weak multiplier 
Hopf algebras \cite{VDWangwha0}, \cite{VDWangwha1}, we pointed out earlier that 
the $Q_R$ map plays the role of the map $p\otimes b\mapsto (p\otimes1)F_1(1\otimes b)$. 
Similarly, the $Q_{\rho}$ map corresponds to $p\otimes b\mapsto (1\otimes b)F_3(p\otimes1)$; 
the $Q_L$ map corresponds to $c\otimes q\mapsto (1\otimes q)F_4(c\otimes1)$; and 
the $Q_{\lambda}$ map  corresponds to $c\otimes q\mapsto (c\otimes1)F_2(1\otimes q)$. 
However, in our case, these expressions are not very useful and we will not mention them 
except in formal settings, because $F_1$, $F_2$, $F_3$, $F_4$ are unbounded elements.
\end{rem}

\subsection{Left/Right invariance of $\varphi$, $\psi$ and the maps 
$Q_R$, $Q_{\rho}$, $Q_L$, $Q_{\lambda}$}

The four maps $Q_R$, $Q_{\rho}$, $Q_L$, $Q_{\lambda}$ satisfy some 
very useful properties, as a consequence of the invariance properties of the Haar weights 
$\varphi$ and $\psi$.

Let $a\in{\mathfrak M}_{\varphi}$.  Then by the left invariance of $\varphi$, 
we know $\Delta a\in\overline{\mathfrak M}_{\operatorname{id}\otimes\varphi}$ and 
$(\operatorname{id}\otimes\varphi)(\Delta a)\in M(C)$.  Apply $\Delta$ here. 
By Equation~\eqref{(DeltaonC)}, we have: 
$$
\Delta\bigl((\operatorname{id}\otimes\varphi)(\Delta a)\bigr)
=E\bigl((\operatorname{id}\otimes\varphi)(\Delta a)\otimes1\bigr)
=(\operatorname{id}\otimes\operatorname{id}\otimes\varphi)\bigl((E\otimes1)
\Delta_{13}(a)\bigr).
$$
On the other hand, by the coassociativity of $\Delta$, we also have:
$$
\Delta\bigl((\operatorname{id}\otimes\varphi)(\Delta a)\bigr)
=(\operatorname{id}\otimes\operatorname{id}\otimes\varphi)
\bigl((\Delta\otimes\operatorname{id})(\Delta a)\bigr)
=(\operatorname{id}\otimes\operatorname{id}\otimes\varphi)
\bigl((\operatorname{id}\otimes\Delta)(\Delta a)\bigr).
$$
Combining the two results, we see that 
$$
(\operatorname{id}\otimes\operatorname{id}\otimes\varphi)
\bigl((\operatorname{id}\otimes\Delta)(\Delta a)\bigr)
=(\operatorname{id}\otimes\operatorname{id}\otimes\varphi)\bigl((E\otimes1)
\Delta_{13}(a)\bigr),\quad\forall a\in{\mathfrak M}_{\varphi}.
$$
Let $c,y\in A$ and multiply $y\otimes c$ to the equation above, from left.  Then, 
the two sides become: 
\begin{align}
&{\text {(LHS)}}\,=\,(\operatorname{id}\otimes\operatorname{id}\otimes\varphi)
\bigl((1\otimes c\otimes1)(\operatorname{id}\otimes\Delta)((y\otimes1)(\Delta a))\bigr),
\notag \\
&{\text {(RHS)}}\,=\,(\operatorname{id}\otimes\operatorname{id}\otimes\varphi)
\bigl((y\otimes c\otimes1)(E\otimes1)\Delta_{13}(a)\bigr).
\notag
\end{align}

Now, let $\theta\in A^*$ be arbitrary and apply $\theta\otimes\operatorname{id}$ to 
both sides.  Then 
\begin{align}\label{(varphiQ_L')}
&(\operatorname{id}\otimes\varphi)\bigl((c\otimes1)\Delta((\theta\otimes\operatorname{id})
[(y\otimes1)(\Delta a)])\bigr)   \notag \\ 
&=(\operatorname{id}\otimes\varphi)\bigl((\theta\otimes\operatorname{id}\otimes\operatorname{id})
[(y\otimes c\otimes1)(E\otimes1)\Delta_{13}(a)]\bigr).
\end{align}
In the right hand side of Equation~\eqref{(varphiQ_L')}, recall the characterization 
of the $Q_{\lambda}$ map given in Proposition~\ref{Q_L'}\,(4).  Assuming the choice of $a,c,y$ are 
such that $(y\otimes c\otimes1)\Delta_{13}(a)\in{\mathcal D}(\operatorname{id}\otimes Q_{\lambda})$, 
this means that the right side of Equation~\eqref{(varphiQ_L')} is same as 
$$
=(\operatorname{id}\otimes\varphi)\bigl((\theta\otimes Q_{\lambda})
[(y\otimes c\otimes1)\Delta_{13}(a)]\bigr).
$$
So Equation~\eqref{(varphiQ_L')} now becomes:
\begin{align}
&(\operatorname{id}\otimes\varphi)\bigl((c\otimes1)\Delta((\theta\otimes\operatorname{id})
[(y\otimes1)(\Delta a)])\bigr)   \notag \\ 
&=(\operatorname{id}\otimes\varphi)\bigl(Q_{\lambda}(c\otimes(\theta\otimes\operatorname{id})
[(y\otimes1)(\Delta a)])\bigr).
\notag
\end{align}
For convenience, write $q:=(\theta\otimes\operatorname{id})[(y\otimes1)(\Delta a)]$, an 
element in $A$.  We can then write the result above as:
\begin{equation}\label{(Q_L'leftinvariance)}
(\operatorname{id}\otimes\varphi)\bigl((c\otimes1)(\Delta q)\bigr)
=(\operatorname{id}\otimes\varphi)\bigl(Q_{\lambda}(c\otimes q)\bigr).
\end{equation}
The elements of the form $q=(\theta\otimes\operatorname{id})[(y\otimes1)(\Delta a)]$ span 
a dense subset in $A$ (``$\Delta$ is full'').  So Equation~\eqref{(Q_L'leftinvariance)} 
will remain valid as long as the expression makes sense.  This would happen when 
$(c\otimes1)(\Delta q)\in{\mathfrak M}_{\operatorname{id}\otimes\varphi}$; $c\otimes q
\in{\mathcal D}(Q_{\lambda})$; and $Q_{\lambda}(c\otimes q)\in{\mathfrak M}_{\operatorname{id}
\otimes\varphi}$. See Proposition~\ref{Q_L'leftinvariance} below:

\begin{prop}\label{Q_L'leftinvariance}
As a consequence of the left invariance of $\varphi$, we have:
\begin{enumerate}
  \item Let $c,q\in A$ be such that $(c\otimes1)(\Delta q)\in{\mathfrak M}_{\operatorname{id}
\otimes\varphi}$; $c\otimes q\in{\mathcal D}(Q_{\lambda})$; and $Q_{\lambda}(c\otimes q)
\in{\mathfrak M}_{\operatorname{id}\otimes\varphi}$.  Then Equation~\eqref{(Q_L'leftinvariance)} 
holds true.  Namely,
$$
(\operatorname{id}\otimes\varphi)\bigl((c\otimes1)(\Delta q)\bigr)
=(\operatorname{id}\otimes\varphi)\bigl(Q_{\lambda}(c\otimes q)\bigr).
$$
  \item In particular, let $p=(\psi\otimes\operatorname{id})\bigl((x^*\otimes1)
\Delta(r^*w)\bigr)$, for $x,w\in{\mathfrak N}_{\psi}$, $r\in{\mathfrak N}_{\varphi}$, 
and let $s\in{\mathfrak N}_{\varphi}$.  Let $c\in A$ be arbitrary. 
Then we have:
\begin{equation}\label{(Q_L'leftinvariancenew)}
(\operatorname{id}\otimes\varphi)\bigl((c\otimes1)\Delta(ps)\bigr)
=(\operatorname{id}\otimes\varphi)\bigl(Q_{\lambda}(c\otimes ps)\bigr).
\end{equation}
\end{enumerate}
\end{prop}

\begin{proof}
We already discussed (1) in previous paragraphs.  

As for (2), note first that since ${\mathfrak N}_{\psi}$ is a left ideal in $A$, so $r^*w
\in{\mathfrak N}_{\psi}$, from which it follows that $c\otimes ps\in{\mathcal D}(Q_{\lambda})$ 
by Proposition~\ref{Q_L'}.
Meanwhile, since ${\mathfrak N}_{\varphi}$ is a left ideal in $A$, we have $r^*w
\in{\mathfrak N}_{\varphi}^*$.  So $\Delta(r^*w)\in{\mathfrak N}_{\operatorname{id}
\otimes\varphi}^*$ (by the left invariance of $\varphi$), which in turn means that 
$p\in{\mathfrak N}_{\varphi}^*$.  Since $s\in{\mathfrak N}_{\varphi}$, we have 
$ps\in{\mathfrak M}_{\varphi}$.  From this observation, we can conclude quickly 
that $(c\otimes1)\Delta(ps)\in{\mathfrak M}_{\operatorname{id}\otimes\varphi}$ 
and also that $Q_{\lambda}(c\otimes ps)\in{\mathfrak M}_{\operatorname{id}\otimes\varphi}$. 
By (1), or by Equation~\eqref{(Q_L'leftinvariance)}, the result follows.
\end{proof}

The next three propositions are similar in nature, corresponding to maps 
$Q_L$, $Q_{\rho}$, $Q_R$, respectively.  The proofs are skipped, as they are 
essentially no different from the above discussion.  

\begin{prop}\label{Q_Lleftinvariance}
As a consequence of the left invariance of $\varphi$, we have:
\begin{enumerate}
  \item Let $c,q\in A$ be such that $(\Delta q)(c\otimes1)\in{\mathfrak M}_{\operatorname{id}
\otimes\varphi}$; $c\otimes q\in{\mathcal D}(Q_L)$; and $Q_L(c\otimes q)
\in{\mathfrak M}_{\operatorname{id}\otimes\varphi}$.  Then: 
\begin{equation}\label{(Q_Lleftinvariance)} 
(\operatorname{id}\otimes\varphi)\bigl((\Delta q)(c\otimes1)\bigr)
=(\operatorname{id}\otimes\varphi)\bigl(Q_L(c\otimes q)\bigr).
\end{equation}
  \item In particular, let $p=(\psi\otimes\operatorname{id})\bigl(\Delta(w^*r)(x\otimes1)\bigr)$, 
for $x,w\in{\mathfrak N}_{\psi}$, $r\in{\mathfrak N}_{\varphi}$, and let 
$s\in{\mathfrak N}_{\varphi}$ (so $s^*\in{\mathfrak N}_{\varphi}^*$).  Let 
$c\in A$ be arbitrary.  Then:
\begin{equation}\label{(Q_Lleftinvariancenew)}
(\operatorname{id}\otimes\varphi)\bigl(\Delta(s^*p)(c\otimes1)\bigr)
=(\operatorname{id}\otimes\varphi)\bigl(Q_L(c\otimes s^*p)\bigr).
\end{equation}
\end{enumerate}
\end{prop}

\begin{prop}\label{Q_R'rightinvariance}
As a consequence of the right invariance of $\psi$, we have:
\begin{enumerate}
  \item Let $b,p\in A$ be such that $(1\otimes b)(\Delta p)\in{\mathfrak M}_{\psi\otimes
\operatorname{id}}$; $p\otimes b\in{\mathcal D}(Q_{\rho})$; and $Q_{\rho}(p\otimes b)
\in{\mathfrak M}_{\psi\otimes\operatorname{id}}$.  Then:
\begin{equation}\label{(Q_R'rightinvariance)}
(\psi\otimes\operatorname{id})\bigl((1\otimes b)(\Delta p)\bigr)
=(\psi\otimes\operatorname{id})\bigl(Q_{\rho}(p\otimes b)\bigr).
\end{equation}
  \item In particular, let $q=(\operatorname{id}\otimes\varphi)\bigl((1\otimes s^*)
\Delta(y^*r)\bigr)$, for $r,s\in{\mathfrak N}_{\varphi}$, $y\in{\mathfrak N}_{\psi}$, 
and let $w\in{\mathfrak N}_{\psi}$.  Let $b\in A$ be arbitrary. 
Then we have:
\begin{equation}\label{(Q_R'rightinvariancenew)}
(\psi\otimes\operatorname{id})\bigl((1\otimes b)\Delta(qw)\bigr)
=(\psi\otimes\operatorname{id})\bigl(Q_{\rho}(qw\otimes b)\bigr).
\end{equation}
\end{enumerate}
\end{prop}

\begin{prop}\label{Q_Rrightinvariance}
As a consequence of the right invariance of $\psi$, we have:
\begin{enumerate}
  \item Let $b,p\in A$ be such that $(\Delta p)(1\otimes b)\in{\mathfrak M}_{\psi\otimes
\operatorname{id}}$; $p\otimes b\in{\mathcal D}(Q_R)$; and $Q_R(p\otimes b)
\in{\mathfrak M}_{\psi\otimes\operatorname{id}}$.  Then: 
\begin{equation}\label{(Q_Rrightinvariance)} 
(\psi\otimes\operatorname{id})\bigl((\Delta p)(1\otimes b)\bigr)
=(\psi\otimes\operatorname{id})\bigl(Q_R(p\otimes b)\bigr).
\end{equation}
  \item In particular, let $q=(\operatorname{id}\otimes\varphi)\bigl(\Delta(r^*y)(1\otimes s)\bigr)$, 
for $r,s\in{\mathfrak N}_{\varphi}$, $y\in{\mathfrak N}_{\psi}$, and let 
$w\in{\mathfrak N}_{\psi}$ (so $w^*\in{\mathfrak N}_{\psi}^*$).  Let 
$b\in A$ be arbitrary.  Then:
\begin{equation}\label{(Q_Rrightinvariancenew)}
(\psi\otimes\operatorname{id})\bigl(\Delta(w^*q)(1\otimes b)\bigr)
=(\psi\otimes\operatorname{id})\bigl(Q_R(w^*q\otimes b)\bigr).
\end{equation}
\end{enumerate}
\end{prop}

The results collected in Propositions~\ref{Q_L'leftinvariance}, \ref{Q_Lleftinvariance}, 
\ref{Q_R'rightinvariance}, \ref{Q_Rrightinvariance} are all consequences of the 
left/right invariance properties of $\varphi$ and $\psi$.  They will play useful roles 
in what follows.  Before we wrap up this subsection, we prove the following result, 
which sharpen the right/left invariance conditions given in Section~\ref{sec1}.

\begin{prop}\label{idphiBpsiidC}
We have:
\begin{enumerate}
  \item $B=\overline{\operatorname{span}\bigl\{(\psi\otimes\operatorname{id})
(\Delta k):k\in{\mathfrak M}_{\psi}\bigr\}}^{\|\ \|}$
  \item $C=\overline{\operatorname{span}\bigl\{(\operatorname{id}\otimes\varphi)
(\Delta k):k\in{\mathfrak M}_{\varphi}\bigr\}}^{\|\ \|}$
\end{enumerate}
\end{prop}

\begin{proof}
(1). Consider $q=(\operatorname{id}\otimes\varphi)\bigl(\Delta(r^*y)(1\otimes s)\bigr)$, 
where $r,s\in{\mathfrak N}_{\varphi}$, $y\in{\mathfrak M}_{\psi}$.  Since $C$ is 
a non-degenerate subalgebra in $M(A)$, we may, without loss of generality, 
assume that $s=ck$, for $c\in C$ and $k\in{\mathfrak N}_{\varphi}$.  In addition, 
let $(b_{\lambda})$ be an approximate unit for $A$.  Then we have:
$$
(\psi\otimes\operatorname{id})(\Delta q)=\lim_{\lambda}(\psi\otimes\operatorname{id})
\bigl((\Delta q)(1\otimes b_{\lambda})\bigr)
=\lim_{\lambda}(\psi\otimes\operatorname{id})\bigl(Q_R(q\otimes b_{\lambda})\bigr),
$$
by Proposition~\ref{Q_Rrightinvariance}.  Computing further, now using the definition 
of the $Q_R$ map given in Proposition~\ref{Q_Rproposition}, this becomes 
\begin{align}
(\psi\otimes\operatorname{id})(\Delta q)
&=\lim_{\lambda}(\psi\otimes\operatorname{id})\bigl((\operatorname{id}\otimes
\operatorname{id}\otimes\varphi)[\Delta_{13}(r^*y)(1\otimes E)(1\otimes b_{\lambda}
\otimes s)]\bigr)   \notag \\
&=(\psi\otimes\operatorname{id})\bigl((\operatorname{id}\otimes\operatorname{id}
\otimes\varphi)[\Delta_{13}(r^*y)(1\otimes E)(1\otimes1\otimes ck)]\bigr)  \notag \\
&=(\operatorname{id}\otimes\varphi)\bigl((\psi\otimes\operatorname{id}\otimes
\operatorname{id})(\Delta_{13}(r^*y))E(1\otimes ck)\bigr)   \notag \\
&=(\operatorname{id}\otimes\varphi)\bigl((1\otimes x)E(1\otimes ck)\bigr),
\notag
\end{align}
where in the last line, we wrote $x=(\psi\otimes\operatorname{id})\bigl(\Delta(r^*y)\bigr)$, 
which is a valid element contained in ${\mathfrak N}_{\varphi}^*$, because 
$r\in{\mathfrak N}_{\varphi},y\in{\mathfrak M}_{\psi}$.  Next, write $\omega:=\varphi(x\,\cdot\,k)$. 
This is a linear functional in $M(A)^*\subseteq C^*$.  In this way,  we see that
$$
(\psi\otimes\operatorname{id})(\Delta q)=(\operatorname{id}\otimes\omega)
\bigl(E(1\otimes c)\bigr).
$$
Since $c\in C$ and $\omega\in C^*$, this is an element contained in $B$ (Here, we 
are using the fact that the left leg of $E$ is $B$.  See Proposition~2.10 in Part~I.). 
By Lemma~\ref{densesubsetA}, we know the elements of the form $q$ above are 
dense in ${\mathfrak M}_{\psi}$. It follows that $(\psi\otimes\operatorname{id})(\Delta k)\in B$, 
for any $k\in{\mathfrak M}_{\psi}$.

For the opposite inclusion, suppose $\theta\in B^*$ is such that 
\begin{equation}\label{(idphiBpsiidCeqn1)}
\theta\bigl((\psi\otimes\operatorname{id})(\Delta k)\bigr)=0,\forall 
k\in{\mathfrak M}_{\psi}.
\end{equation}
Without loss of generality, we may assume that $\theta=\nu(\,\cdot\,b)$, $b\in{\mathfrak M}_{\nu}$, 
where $\nu$ is the KMS weight on $B$.  Note that $\bigl\{\nu(\,\cdot\,b):b\in{\mathfrak M}_{\nu}\bigr\}$ 
is dense in $B^*$.  Meanwhile, since $b\in B$, 
we know from Equation~\eqref{(DeltaonB)} that $\Delta b=E(1\otimes b)$.
Putting these together, the left side of Equation~\eqref{(idphiBpsiidCeqn1)} becomes:
$$
{\text {(LHS)}}=\nu\bigl((\psi\otimes\operatorname{id})((\Delta k)(1\otimes b))\bigr) 
=\nu\bigl((\psi\otimes\operatorname{id})(\Delta(kb))\bigr)  
=\psi(kb).
$$
In the last equality, we used the result of Proposition~\ref{nuphipsi}.  In other words, 
Equation~\eqref{(idphiBpsiidCeqn1)} is none other than saying that $\psi(kb)=0$, $\forall 
k\in{\mathfrak M}_{\psi}$.  By the faithfulness of $\psi$, this would mean that $b=0$. 
So $\theta\equiv0$.  This shows that the elements of the form $(\psi\otimes\operatorname{id})
(\Delta k)$ are dense in $B$.

(2). The result about $C$ is similarly proved.
\end{proof}

\section{The operators $V$ and $W$: The right and left regular representations}\label{sec3}

In the theory of locally compact quantum groups \cite{KuVa}, \cite{KuVavN}, \cite{MNW}, 
\cite{VDvN}, a fundamental role is played by the ``multiplicative unitary operators'' 
(in the sense of Baaj and Skandalis \cite{BS}; see also \cite{Wr7}), as they are 
essentially the right/left regular representations.  In the setting of measured quantum 
groupoids \cite{LesSMF}, \cite{EnSMF}, a similar role is played by the ``pseudo-multiplicative 
unitary operators'' (see also \cite{Timmbook}).

In our setting, an analogous role will be played by certain partial isometries 
(the {\em multiplicative partial isometries\/}).  However, unlike in the cases of quantum groups 
or measured quantum groupoids, these operators are in general not unitaries. 
This causes some subtle issues that are not present in the quantum group theory. 
The tools developed in Section~\ref{sec2} will be useful.  

\subsection{Defining the operators  $V$ and $W$}

For our discussion, we will assume the existence of a proper weight $\eta$ on our $C^*$-algebra 
$A$ and fix a GNS-construction $({\mathcal H},\pi,\Lambda)$ corresponding to $\eta$.  
We will identify $A=\pi(A)\,\subseteq{\mathcal B}({\mathcal H})$.  Depending on the context, 
we may later let $\eta=\psi$, $\eta=\varphi$, or something else.  Let $(e_j)_{j\in J}$ be 
an orthonormal basis for ${\mathcal H}$.  We will also use the standard notation 
$\omega_{\xi,\zeta}$, for $\xi,\zeta\in{\mathcal H}$, to denote the linear form defined 
by $\omega_{\xi,\zeta}(T)=\langle T\xi,\zeta\rangle$, $\forall T\in{\mathcal B}({\mathcal H})$. 
Note that any $\omega\in{\mathcal B}({\mathcal H})_*$ can be approximated by the 
$\omega_{\xi,\zeta}$.  See Lemma~\ref{omega_xizetaLem} below for some useful results: 

\begin{lem}\label{omega_xizetaLem}
Let $\xi,\zeta\in{\mathcal H}$.  Then: 
\begin{enumerate}
  \item $\overline{\omega_{\xi,\zeta}}=\omega_{\zeta,\xi}$;
  \item $(\operatorname{id}\otimes\omega_{\xi,\zeta})(T)^*
=(\operatorname{id}\otimes\omega_{\zeta,\xi})(T^*)$, for $T\in{\mathcal B}
({\mathcal H}\otimes{\mathcal H})$;
  \item For $S,T\in{\mathcal B}({\mathcal H})$, we have:
$$
\sum_{j\in J}\omega_{e_j,\zeta}(S)\omega_{\xi,e_j}(T)=\omega_{\xi,\zeta}(ST);
$$
  \item For $S,T\in{\mathcal B}({\mathcal H}\otimes{\mathcal H})$, we have:
$$
\sum_{j\in J}(\operatorname{id}\otimes\omega_{e_j,\zeta})(S)(\operatorname{id}
\otimes\omega_{\xi,e_j})(T)=(\operatorname{id}\otimes\omega_{\xi,\zeta})(ST).
$$
\end{enumerate}
\end{lem}

\begin{rem}
We skip the proof of the lemma, as these are essentially basic linear algebra results. 
Here, the complex conjugate $\overline{\omega}$ for $\omega\in{\mathcal B}({\mathcal H})_*$ 
is given by $\overline{\omega}(T):=\overline{\omega(T^*)}$, for $T\in{\mathcal B}({\mathcal H})$. 
Also, the result in (4) is to be understood that the net of finite sums 
$\left(\sum_{j\in I}(\operatorname{id}\otimes\omega_{e_j,\zeta})(S)
(\operatorname{id}\otimes\omega_{\xi,e_j})(T)\right)_{I\in F(J)}$ converges, in operator norm, 
to $(\operatorname{id}\otimes\omega_{\xi,\zeta})(ST)$. 
\end{rem}

Let us first construct the right regular representation, in terms of a certain operator $V$.  For this, 
consider the right Haar weight $\psi$ and let $({\mathcal H}_{\psi},\pi_{\psi},\Lambda_{\psi})$ 
be its GNS-triple.   See below, where we can recognize the resemblance to the corresponding 
definition in the case of locally compact quantum groups \cite{KuVa}, \cite{KuVavN}, \cite{VDvN}.

\begin{prop}\label{Vdefn}
\begin{enumerate}
  \item There exists a bounded operator $V\in{\mathcal B}({\mathcal H}_{\psi}\otimes{\mathcal H})$ 
satisfying
$$
\bigl((\operatorname{id}\otimes\omega)(V)\bigr)\Lambda_{\psi}(p)
=\Lambda_{\psi}\bigl((\operatorname{id}\otimes\omega)(\Delta p)\bigr),
$$
for $p\in{\mathfrak N}_{\psi}$ and $\omega\in{\mathcal B}({\mathcal H})_*$.
  \item If $p\in{\mathfrak N}_{\psi}$ and $a\in{\mathfrak N}_{\eta}$, then we have:
\begin{equation}\label{(VT_1)}
V\bigl(\Lambda_{\psi}(p)\otimes\Lambda(a)\bigr)=(\Lambda_{\psi}\otimes\Lambda)
\bigl((\Delta p)(1\otimes a)\bigr).
\end{equation}
\end{enumerate}
\end{prop}

\begin{proof}
(1). By the right invariance of $\psi$, we know $(\operatorname{id}\otimes\omega)(\Delta p)
\in\overline{\mathfrak N}_{\psi}$.  So the expression makes sense.  Now, let $\xi$ be arbitrary, 
and consider $\omega_{\xi,e_j}\in{\mathcal B}({\mathcal H})_*$, where $(e_j)_{j\in J}$ is 
an orthonormal basis for ${\mathcal H}$.  Then:
\begin{align}
\sum_{j\in J}\bigl\|\Lambda_{\psi}((\operatorname{id}\otimes\omega_{\xi,e_j})(\Delta p))\bigr\|^2
&=\sum_{j\in J}\psi\bigl((\operatorname{id}\otimes\omega_{\xi,e_j})(\Delta p)^*
(\operatorname{id}\otimes\omega_{\xi,e_j})(\Delta p)\bigr)
\notag \\
&=\sum_{j\in J}\psi\bigl((\operatorname{id}\otimes\omega_{e_j,\xi})(\Delta (p^*))
(\operatorname{id}\otimes\omega_{\xi,e_j})(\Delta p)\bigr)   \notag \\
&\le\psi\bigl((\operatorname{id}\otimes\omega_{\xi,\xi})(\Delta(p^*p))\bigr),
\notag
\end{align}
where we are using (2), (4) of Lemma~\ref{omega_xizetaLem} and the lower semi-continuity 
of the weight $\psi$.  As $\psi\bigl((\operatorname{id}\otimes\omega_{\xi,\xi})(\Delta(p^*p))\bigr)
=\omega_{\xi,\xi}\bigl((\psi\otimes\operatorname{id})(\Delta(p^*p))\bigr)$, this then becomes: 
$$
\sum_{j\in J}\bigl\|\Lambda_{\psi}((\operatorname{id}\otimes\omega_{\xi,e_j})(\Delta p))\bigr\|^2
\le\bigl\langle(\psi\otimes\operatorname{id})(\Delta(p^*p))\,\xi,\xi\bigr\rangle
\le\bigl\|(\psi\otimes\operatorname{id})(\Delta(p^*p))\bigr\|\|\xi\|^2.
$$
Because of the way $V$ was defined, the left side of this equation is actually 
$\sum_{j\in J}\bigl\|\langle V(\Lambda_{\psi}(p)\otimes\xi),\cdot\otimes e_j\rangle\bigr\|^2$, 
while $(\psi\otimes\operatorname{id})(\Delta(p^*p))$ in the right side is a bounded element 
in $M(B)$, due to the right invariance of $\psi$.  We can thus see that $V:\Lambda_{\psi}(p)
\otimes\xi\mapsto V(\Lambda_{\psi}(p)\otimes\xi)$ is a bounded operator.

(2). Let $a,b\in{\mathfrak N}_{\eta}$ be arbitrary and consider $\omega=\omega_{\Lambda(a),\Lambda(b)}$. 
Then, for any $p,q\in{\mathfrak N}_{\psi}$, we have:
\begin{align}
\bigl\langle V(\Lambda_{\psi}(p)\otimes\Lambda(a)),\Lambda_{\psi}(q)\otimes\Lambda(b)\bigr\rangle
&=\bigl\langle(\operatorname{id}\otimes\omega_{\Lambda(a),\Lambda(b)})(V)\Lambda_{\psi}(p),
\Lambda_{\psi}(q)\bigr\rangle   \notag \\
&=\bigl\langle\Lambda_{\psi}((\operatorname{id}\otimes\omega_{\Lambda(a),\Lambda(b)})(\Delta p)),
\Lambda_{\psi}(q)\bigr\rangle   \notag \\
&=(\psi\otimes\eta)\bigl((q^*\otimes b^*)(\Delta p)(1\otimes a)\bigr)  \notag \\
&=\bigl\langle(\Lambda_{\psi}\otimes\Lambda)((\Delta p)(1\otimes a)),\Lambda_{\psi}(q)\otimes
\Lambda(b)\bigr\rangle.
\notag
\end{align}
This is true for arbitrary $q\in{\mathfrak N}_{\psi}$ and $b\in{\mathfrak N}_{\eta}$, so we have: 
$$
V\bigl(\Lambda_{\psi}(p)\otimes\Lambda(a)\bigr)=(\Lambda_{\psi}\otimes\Lambda)
\bigl((\Delta p)(1\otimes a)\bigr).
$$
\end{proof}

In the below are some immediate properties for $V$.  Here, for convenience of the notation, 
we just wrote $x\in A$ to represent $\pi_{\psi}(x)\in{\mathcal B}({\mathcal H}_{\psi})$.  As it will 
eventually turn out that we can regard ${\mathcal H}_{\psi}={\mathcal H}$ (see last Remark 
in Section~\ref{sec5}), this casual bookkeeping is not too harmful.  Similarly, we regard 
$\Delta x=(\pi_{\psi}\otimes\pi)(\Delta x)$ and $E=(\pi_{\psi}\otimes\pi)(E)$, as contained 
in ${\mathcal B}({\mathcal H}_{\psi}\otimes{\mathcal H})$.

\begin{prop}\label{Vprop}
\begin{enumerate}
  \item $V(x\otimes 1)=(\Delta x)V$, for any $x\in A$;
  \item $EV=V$;
  \item $\overline{\operatorname{Ran}(V)}\subseteq\operatorname{Ran}(E)$, 
in ${\mathcal H}_{\psi}\otimes{\mathcal H}$.
\end{enumerate}
\end{prop}

\begin{proof}
(1). For any $p\in{\mathfrak N}_{\psi}$ and any $a\in{\mathfrak N}_{\eta}$, we have: 
\begin{align}
V(x\otimes1)\bigl(\Lambda_{\psi}(p)\otimes\Lambda(a)\bigr)
&=V\bigl(\Lambda_{\psi}(xp)\otimes\Lambda(a)\bigr) 
=(\Lambda_{\psi}\otimes\Lambda)\bigl(\Delta (xp)(1\otimes a)\bigr)  \notag \\
&=(\Delta x)(\Lambda_{\psi}\otimes\Lambda)\bigl((\Delta p)(1\otimes a)\bigr) \notag \\
&=(\Delta x)V\bigl(\Lambda_{\psi}(p)\otimes\Lambda(a)\bigr),
\notag
\end{align}
where we used the property of the GNS representation and Equation~\eqref{(VT_1)}.

(2). As $E(\Delta p)=\Delta p$, we have for any $p\in{\mathfrak N}_{\psi}$ and any 
$a\in{\mathfrak N}_{\eta}$, 
\begin{align}
EV\bigl(\Lambda_{\psi}(p)\otimes\Lambda(a)\bigr)
&=E(\Lambda_{\psi}\otimes\Lambda)\bigl((\Delta p)(1\otimes a)\bigr)
=(\Lambda_{\psi}\otimes\Lambda)\bigl(E(\Delta p)(1\otimes a)\bigr)  \notag \\
&=(\Lambda_{\psi}\otimes\Lambda)\bigl((\Delta p)(1\otimes a)\bigr)
=V\bigl(\Lambda_{\psi}(p)\otimes\Lambda(a)\bigr).
\notag
\end{align}

(3). Since $EV=V$, the result is immediate.  Note here that the space 
$\operatorname{Ran}(E)$ is already closed in ${\mathcal H}_{\psi}\otimes
{\mathcal H}$ because $E$ is a projection.
\end{proof}

In an analogous way, we can also construct the left regular representation, in terms 
of an operator $W$.  For this, consider now the left Haar weight $\varphi$ and let 
$({\mathcal H}_{\varphi},\pi_{\varphi},\Lambda_{\varphi})$ be its GNS-triple.   We 
actually define $W^*$ first, as is typically done in the quantum group case: 

\begin{prop}\label{Wdefn}
There exists a bounded operator $W\in{\mathcal B}({\mathcal H}\otimes{\mathcal H}_{\varphi})$ 
characterized and defined by
$$
\bigl((\theta\otimes\operatorname{id})(W^*)\bigr)\Lambda_{\varphi}(a)
=\Lambda_{\varphi}\bigl((\theta\otimes\operatorname{id})(\Delta a)\bigr), 
\quad {\text { for }}a\in{\mathfrak N}_{\varphi},\theta\in{\mathcal B}({\mathcal H})_*.
$$

If $a\in{\mathfrak N}_{\varphi}$ and $p\in{\mathfrak N}_{\eta}$, then we have:
\begin{equation}\label{(WT_4)}
W^*\bigl(\Lambda(p)\otimes\Lambda_{\varphi}(a)\bigr)=(\Lambda\otimes\Lambda_{\varphi})
\bigl((\Delta a)(p\otimes 1)\bigr).
\end{equation}
\end{prop}

Proof is essentially no different than in the case of Proposition~\ref{Vdefn}.  We also have 
the following result, analogous to  Proposition~\ref{Vprop}.  With a similar comment as before, 
we are regarding $x=\pi_{\varphi}(x)\in{\mathcal B}({\mathcal H}_{\varphi})$, while 
$\Delta x=(\pi\otimes\pi_{\varphi})(\Delta x)$ and $E=(\pi\otimes\pi_{\varphi})(E)$, 
as elements contained in ${\mathcal B}({\mathcal H}\otimes{\mathcal H}_{\varphi})$.

\begin{prop}\label{Wprop}
\begin{enumerate}
  \item $W^*(1\otimes x)=(\Delta x)W^*$, for any $x\in A$;
  \item $EW^*=W^*$;
  \item $\overline{\operatorname{Ran}(W^*)}\subseteq\operatorname{Ran}(E)$, 
in ${\mathcal H}\otimes{\mathcal H}_{\varphi}$.
\end{enumerate}
\end{prop}

In the quantum group case \cite{KuVa}, \cite{KuVavN}, \cite{VDvN}, we have $E=1\otimes1$ 
and it is known that $V$ and $W$ are unitaries.  In our case ($E\ne1\otimes1$), this is 
no longer expected.  It turns out that $V$ and $W$ are partial isometries such that 
$\operatorname{Ran}(V)=\operatorname{Ran}(E)$ for $V$, and $\operatorname{Ran}(W^*)
=\operatorname{Ran}(E)$ for $W$.  However, the proof of ``$\supseteq$'' is rather complicated 
and needs more work.  This is what we aim to establish in the following two subsections.

\subsection{$V$ is a partial isometry}

Let us begin with a result that will later help us understand how the operator $V^*$ 
behaves:

\begin{prop}\label{V^*prep}
Consider $a=(\operatorname{id}\otimes\varphi)\bigl(\Delta(r^*y)(1\otimes s)\bigr)$, where 
$r,s\in{\mathfrak N}_{\varphi}$ and $y\in{\mathfrak N}_{\psi}$.  Then $a\in{\mathfrak N}_{\psi}$. 
Consider also $b\in{\mathfrak N}_{\eta}$.

Let $z=(\operatorname{id}\otimes\operatorname{id}\otimes\varphi)
\bigl(\Delta_{13}(r^*y)\Delta_{23}(s)\bigr)\in\overline{\mathfrak N}_{\psi\otimes\operatorname{id}}$. 
Then we have:
$$
z(1\otimes b)=(\operatorname{id}\otimes\operatorname{id}\otimes\varphi)
\bigl(\Delta_{13}(r^*y)\Delta_{23}(s)(1\otimes b\otimes1)\bigr).
$$
With the notation as above, we have: 
$$
V\bigl((\Lambda_{\psi}\otimes\Lambda)(z(1\otimes b))\bigr)=E\bigl(\Lambda_{\psi}(a)
\otimes\Lambda(b)\bigr).
$$
\end{prop}

\begin{proof}
As $y\in{\mathfrak N}_{\psi}$, it is easy to observe, by the right invariance of $\psi$, that 
$a\in{\mathfrak N}_{\psi}$ and $z\in\overline{\mathfrak N}_{\psi\otimes\operatorname{id}}$. 
We also have $z(1\otimes b)\in\overline{\mathfrak N}_{\psi\otimes\eta}$.

Suppose $w\in{\mathfrak N}_{\psi}$ (so $w^*\in{\mathfrak N}_{\psi}^*$) and $c\in{\mathfrak N}_{\eta}$ 
(so $c\in{\mathfrak N}_{\eta}^*$).  Then by the definition of the operator $V$, as in Equation~\eqref{(VT_1)}, 
we have:
\begin{align}
&\bigl\langle V((\Lambda_{\psi}\otimes\Lambda)[z(1\otimes b)]),\Lambda_{\psi}(w)\otimes
\Lambda(c)\bigr\rangle   \notag \\
&=(\psi\otimes\eta)\bigl((w^*\otimes c^*)(\operatorname{id}\otimes\operatorname{id}\otimes\varphi)
((\Delta\otimes\operatorname{id})[\Delta(r^*y)]\Delta_{23}(s)(1\otimes b\otimes1))\bigr).
\notag
\end{align}
By the coassociativity of $\Delta$, the right side becomes:
$$
=(\psi\otimes\eta)\bigl((w^*\otimes c^*)(\operatorname{id}\otimes\operatorname{id}\otimes\varphi)
((\operatorname{id}\otimes\Delta)[\Delta(r^*y)]\Delta_{23}(s)(1\otimes b\otimes1))\bigr).
$$
This is equal to 
\begin{align}
&=(\psi\otimes\eta\otimes\varphi)\bigl((w^*\otimes c^*\otimes1)(\operatorname{id}\otimes\Delta)
[\Delta(r^*y)(1\otimes s)](1\otimes b\otimes1)\bigr)   \notag \\
&=(\eta\otimes\varphi)\bigl((c^*\otimes1)\Delta((\psi\otimes\operatorname{id})[(w^*\otimes1)
\Delta(r^*y)]s)(b\otimes1)\bigr)   \notag \\
&=(\eta\otimes\varphi)\bigl((c^*\otimes1)\Delta(ps)(b\otimes1)\bigr)
=\eta\bigl((\operatorname{id}\otimes\varphi)((c^*\otimes1)\Delta(ps))\,b\bigr),
\notag
\end{align}
where we wrote $p=(\psi\otimes\operatorname{id})[(w^*\otimes1)\Delta(r^*y)]$.  

Apply here the result of Proposition~\ref{Q_L'leftinvariance}, a consequence of the left invariance 
of $\varphi$.  All the conditions for Equation~\eqref{(Q_L'leftinvariancenew)} hold, 
so we have:
$$
(\operatorname{id}\otimes\varphi)\bigl((c^*\otimes1)\Delta(ps)\bigr)
=(\operatorname{id}\otimes\varphi)\bigl(Q_{\lambda}(c^*\otimes ps)\bigr).
$$
But by the definition of the $Q_{\lambda}$ map as given in Proposition~\ref{Q_L'}, 
and remembering the definition of $p$ above, we have:
$$
Q_{\lambda}(c^*\otimes ps)
=(\psi\otimes\operatorname{id}\otimes\operatorname{id})\bigl((w^*\otimes c^*\otimes1)
(E\otimes1)\Delta_{13}(r^*y)\bigr)(1\otimes s).
$$

Therefore, combining all these together, we have:
\begin{align}
&\bigl\langle V((\Lambda_{\psi}\otimes\Lambda)[z(1\otimes b)]),\Lambda_{\psi}(w)\otimes
\Lambda(c)\bigr\rangle   \notag \\
&=\eta\bigl((\psi\otimes\operatorname{id}\otimes\varphi)[(w^*\otimes c^*\otimes1)(E\otimes1)
\Delta_{13}(r^*y)(1\otimes1\otimes s)]\,b\bigr)  \notag \\
&=(\psi\otimes\eta)\bigl((w^*\otimes c^*)E\,(\operatorname{id}\otimes\operatorname{id}\otimes
\varphi)[\Delta_{13}(r^*y)(1\otimes1\otimes s)]\,(1\otimes b)\bigr)  \notag \\
&=(\psi\otimes\eta)\bigl((w^*\otimes c^*)E(a\otimes b)\bigr) 
=\bigl\langle(\Lambda_{\psi}\otimes\Lambda)(E(a\otimes b)),
\Lambda_{\psi}(w)\otimes\Lambda(c)\bigr\rangle     \notag \\ 
&=\bigl\langle E(\Lambda_{\psi}(a)\otimes\Lambda(b)),
\Lambda_{\psi}(w)\otimes\Lambda(c)\bigr\rangle.
\notag
\end{align}
Third equality is just remembering that $a=(\operatorname{id}\otimes\varphi)
\bigl(\Delta(r^*y)(1\otimes s)\bigr)$, and the last equality is the property of the GNS-representation. 
Since this is true for arbitrary $w\in{\mathfrak N}_{\psi}$ and $c\in{\mathfrak N}_{\eta}$, we have 
shown the desired result.
\end{proof}

\begin{rem}
In the quantum group theory, a result analogous to the above proposition helps us  see that the map 
$\Lambda_{\psi}(a)\otimes\Lambda(b)\mapsto(\Lambda_{\psi}\otimes\Lambda)\bigl(z(1\otimes b)\bigr)$ 
determines the (adjoint) operator $V^*$.  It will turn out later that this fact is still true even in our more 
general setting, but at this stage (with $E\ne1\otimes1$), this is not so obvious.
\end{rem}

Define now a subspace ${\mathcal K}_{\psi}$ of ${\mathcal H}_{\psi}$, as follows:
\begin{equation}\label{(K_psi)}
{\mathcal K}_{\psi}:=\overline{\operatorname{span}}^{\|\ \|}\bigl\{\Lambda_{\psi}((\operatorname{id}
\otimes\varphi)[\Delta(r^*y)(1\otimes s)]):y\in{\mathfrak N}_{\psi},\,r,s\in{\mathfrak N}_{\varphi}\bigr\}.
\end{equation}
Then the result of Proposition~\ref{V^*prep} implies the following:

\begin{prop}\label{RanEsubsetRanV}
$E({\mathcal K}_{\psi}\otimes{\mathcal H})\subseteq\overline{V({\mathcal K}_{\psi}\otimes{\mathcal H})}$.
\end{prop}

\begin{proof}
With the notation as in Proposition~\ref{V^*prep}, we see that $\Lambda_{\psi}(a)
\in{\mathcal K}_{\psi}$, and such elements span a dense subspace of ${\mathcal K}_{\psi}$. 
Write $z\in\overline{\mathfrak N}_{\psi\otimes\operatorname{id}}$ as in that proposition, 
and let $b\in{\mathfrak N}_{\eta}$.

Note that using basically the same (tedious but straightforward) method as in 
Proposition~A.9 in \cite{KuVa}, we can show:
$$
(\Lambda_{\psi}\otimes\Lambda)\bigl(z(1\otimes b)\bigr)
=\sum_{j\in J}\Lambda_{\psi}\bigl((\operatorname{id}\otimes\omega_{\Lambda(b),e_j})(z)\bigr)
\otimes e_j,
$$
where $(e_j)_{j\in J}$ is an orthonormal basis for ${\mathcal H}$.  With the definition of $z$, 
this becomes:
$$
=\sum_{j\in J}\Lambda_{\psi}\bigl((\operatorname{id}\otimes\varphi)[\Delta(r^*y)
(1\otimes(\omega_{\Lambda(b),e_j}\otimes\operatorname{id})(\Delta s))]\bigr)
\otimes e_j.
$$
Expressed in this way, we can see that $(\Lambda_{\psi}\otimes\Lambda)\bigl(z(1\otimes b)\bigr)
\in{\mathcal K}_{\psi}\otimes{\mathcal H}$.

By the result $E\bigl(\Lambda_{\psi}(a)\otimes\Lambda(b)\bigr)
=V\bigl((\Lambda_{\psi}\otimes\Lambda)(z(1\otimes b))\bigr)$ from Proposition~\ref{V^*prep}, 
we see that $E({\mathcal K}_{\psi}\otimes{\mathcal H})
\subseteq\overline{V({\mathcal K}_{\psi}\otimes{\mathcal H})}$.
\end{proof}

\begin{rem}
While we know from Lemma~\ref{densesubsetA} that the elements of the form 
$(\operatorname{id}\otimes\varphi)[\Delta(r^*y)(1\otimes s)]$ span a dense subspace 
in $A$, this does not necessarily mean that the $\Lambda_{\psi}((\operatorname{id}\otimes
\varphi)[\Delta(r^*y)(1\otimes s)])$ span a dense subspace in ${\mathcal H}_{\psi}$. 
The norms involved (operator norm versus Hilbert space norm) are different.  So, 
at this stage, the result of the proposition saying that $E({\mathcal K}_{\psi}
\otimes{\mathcal H})\subseteq\overline{V({\mathcal K}_{\psi}\otimes{\mathcal H})}$ 
is not yet enough to argue that $\overline{\operatorname{Ran}(V)}\supseteq
\operatorname{Ran}(E)$.  We will have to give a separate proof that we indeed have 
${\mathcal K}_{\psi}={\mathcal H}_{\psi}$.  This is what we will achieve in this subsection.
\end{rem}

\begin{prop}\label{K_psinew}
Let ${\mathcal K}_{\psi}\subseteq{\mathcal H}_{\psi}$ be as defined in Equation~\eqref{(K_psi)}. 
Then
$$
{\mathcal K}_{\psi}=\overline{\operatorname{span}}^{\|\ \|}\bigl\{\Lambda_{\psi}((\operatorname{id}
\otimes\omega)(\Delta x)):x\in{\mathfrak N}_{\psi},\omega\in A^*\bigr\}.
$$
\end{prop}

\begin{proof}
Let $x\in{\mathfrak N}_{\psi}$.  Then, for $\omega\in A^*$, we know from the definition of $V$ 
(see Proposition~\ref{Vdefn}) that $\Lambda_{\psi}\bigl((\operatorname{id}\otimes\omega)
(\Delta x)\bigr)=(\operatorname{id}\otimes\omega)(V)\Lambda_{\psi}(x)$. 
As $(\operatorname{id}\otimes\omega)(V)$ is a bounded operator, the map 
$\omega\mapsto\Lambda_{\psi}\bigl((\operatorname{id}\otimes\omega)(\Delta x)\bigr)$ 
is a bounded map.  It can be extended, as long as the expression remains valid.

Note that any $\omega\in A^*$ can be approximated by the $\omega_{\Lambda(a),
\Lambda(b)}$, for $a,b\in{\mathcal T}_{\varphi}$ (the Tomita subalgebra). 
It follows that any $\Lambda_{\psi}((\operatorname{id}\otimes\omega)(\Delta x))$, for 
$x\in{\mathfrak N}_{\psi}$, $\omega\in A^*$, can be approximated by the 
$\Lambda_{\psi}((\operatorname{id}\otimes\omega_{\Lambda(a),\Lambda(b)})(\Delta(r^*y))$, 
for $y\in{\mathfrak N}_{\psi}$, $r\in{\mathfrak N}_{\varphi}$.  

Meanwhile, it is easy to see that $\omega_{\Lambda(a),\Lambda(b)}=\varphi(b^*\,\cdot\,a)$, 
because for any $x\in A$ we have: $\omega_{\Lambda(a),\Lambda(b)}(x)=\bigl\langle x\Lambda(a),
\Lambda(b)\bigr\rangle=\varphi(b^*xa)$.  So we have 
\begin{align}
\Lambda_{\psi}((\operatorname{id}\otimes\omega_{\Lambda(a),\Lambda(b)})(\Delta(r^*y))
&=\Lambda_{\psi}\bigl((\operatorname{id}\otimes\varphi)[(1\otimes b^*)\Delta(r^*y)(1\otimes a)]\bigr)
\notag \\
&=\Lambda_{\psi}\bigl((\operatorname{id}\otimes\varphi)[\Delta(r^*y)(1\otimes s)]\bigr).
\notag
\end{align}

In this way, we just showed that 
\begin{align}
{\mathcal K}_{\psi}&=\overline{\operatorname{span}}^{\|\ \|}\bigl\{\Lambda_{\psi}((\operatorname{id}
\otimes\varphi)[\Delta(r^*y)(1\otimes s)]):y\in{\mathfrak N}_{\psi},\,r,s\in{\mathfrak N}_{\varphi}\bigr\}
\notag \\
&=\overline{\operatorname{span}}^{\|\ \|}\bigl\{\Lambda_{\psi}((\operatorname{id}\otimes
\omega_{\Lambda(a),\Lambda(b)})(\Delta x)):x\in{\mathfrak N}_{\psi},a,b\in{\mathcal T}_{\varphi}\bigr\}
\notag  \\
&=\overline{\operatorname{span}}^{\|\ \|}\bigl\{\Lambda_{\psi}((\operatorname{id}\otimes\omega)
(\Delta x)):x\in{\mathfrak N}_{\psi},\omega\in A^*\bigr\}.
\notag 
\end{align}
\end{proof}

\begin{cor}
By the new characterization of the subspace ${\mathcal K}_{\psi}$, we have 
$$
\overline{\operatorname{Ran}(V)}=\overline{V({\mathcal H}_{\psi}\otimes{\mathcal H})}
\subseteq{\mathcal K}_{\psi}\otimes{\mathcal H}. 
$$
\end{cor}

\begin{proof}
For any $x\in{\mathfrak N}_{\psi}$, recall that 
$(\operatorname{id}\otimes\omega)(V)\Lambda_{\psi}(x)
=\Lambda_{\psi}\bigl((\operatorname{id}\otimes\omega)(\Delta x)\bigr)$, which is contained in 
${\mathcal K}_{\psi}$ by Proposition~\ref{K_psinew}.  This is true for any $\omega\in{\mathcal B}
({\mathcal H})_*$.  As the elements $\Lambda_{\psi}(x)$, $x\in{\mathfrak N}_{\psi}$, are dense 
in ${\mathcal H}_{\psi}$, it follows that $\overline{V({\mathcal H}_{\psi}\otimes{\mathcal H})}
\subseteq{\mathcal K}_{\psi}\otimes{\mathcal H}$.
\end{proof}

Combine the result of the Corollary, together with our earlier result in Proposition~\ref{RanEsubsetRanV} 
that $E({\mathcal K}_{\psi}\otimes{\mathcal H})\subseteq\overline{V({\mathcal K}_{\psi}\otimes
{\mathcal H})}$.  We then have:
\begin{equation}\label{(RangeVinRangeE)}
E({\mathcal K}_{\psi}\otimes{\mathcal H})\subseteq\overline{V({\mathcal K}_{\psi}\otimes{\mathcal H})}
\subseteq\overline{V({\mathcal H}_{\psi}\otimes{\mathcal H})}
\subseteq{\mathcal K}_{\psi}\otimes{\mathcal H}.
\end{equation}
Next, apply $E$ from the left, and use the fact that $EV=V$ (from Proposition~\ref{Vprop}) and that 
$E^2=E$.  Then: 
$$
E({\mathcal K}_{\psi}\otimes{\mathcal H})\subseteq\overline{V({\mathcal K}_{\psi}\otimes{\mathcal H})}
\subseteq\overline{V({\mathcal H}_{\psi}\otimes{\mathcal H})}
\subseteq E({\mathcal K}_{\psi}\otimes{\mathcal H}).
$$
From this, we observe the following: 
\begin{equation}\label{(RangeV)}
\overline{V({\mathcal K}_{\psi}\otimes{\mathcal H})}
=\overline{V({\mathcal H}_{\psi}\otimes{\mathcal H})}
=E({\mathcal K}_{\psi}\otimes{\mathcal H}).
\end{equation}

Motivated by Equations~\eqref{(RangeVinRangeE)} and \eqref{(RangeV)}, let us consider 
$U:=V|_{{\mathcal K}_{\psi}\otimes{\mathcal H}}$, which is an operator contained in ${\mathcal B}
({\mathcal K}_{\psi}\otimes{\mathcal H})$.  Equation~\eqref{(RangeV)} means that 
$\overline{\operatorname{Ran}(U)}=E({\mathcal K}_{\psi}\otimes{\mathcal H})$. 
See also the next result: 

\begin{prop}\label{U^*}
Let $U:=V|_{{\mathcal K}_{\psi}\otimes{\mathcal H}}$.  Then $U^*\in{\mathcal B}({\mathcal K}_{\psi}
\otimes{\mathcal H})$ is characterized by
$$
\Lambda_{\psi}(a)\otimes\Lambda(b)\mapsto(\Lambda_{\psi}\otimes\Lambda)\bigl(z(1\otimes b)\bigr),
$$
where $a$, $z$, $b$ are as in Proposition~\ref{V^*prep}.  That is, 
$a=(\operatorname{id}\otimes\varphi)\bigl(\Delta(r^*y)(1\otimes s)\bigr)$, for 
$y\in{\mathfrak N}_{\psi}$, $r,s\in{\mathfrak N}_{\varphi}$ and $z=(\operatorname{id}\otimes
\operatorname{id}\otimes\varphi)\bigl(\Delta_{13}(r^*y)\Delta_{23}(s)\bigr)$.  Also 
$b\in{\mathfrak N}_{\eta}$.
\end{prop}

\begin{proof}
By the definition of the subspace ${\mathcal K}_{\psi}$ given in Equation~\eqref{(K_psi)}, we know 
that the elements $\Lambda_{\psi}(a)\otimes\Lambda(b)$ span a dense subset in 
${\mathcal K}_{\psi}\otimes{\mathcal H}$.  We would prove the result if we can show that
\begin{equation}\label{(eqn_U^*)}
\bigl\langle\Lambda_{\psi}(a)\otimes\Lambda(b),V(\Lambda_{\psi}(c)\otimes\Lambda(d))\bigr\rangle
=\bigl\langle(\Lambda_{\psi}\otimes\Lambda)(z(1\otimes b)),\Lambda_{\psi}(c)\otimes\Lambda(d)
\bigr\rangle,
\end{equation}
for any $c\in{\mathfrak N}_{\psi}$ and any $d\in{\mathfrak N}_{\eta}$.  This is actually a stronger 
result than is necessary because $\Lambda_{\psi}(c)\otimes\Lambda(d)\in{\mathcal H}_{\psi}
\otimes{\mathcal H}$, but convenient because $V(\zeta\otimes\xi)=U(\zeta\otimes\xi)$, 
if $\zeta\in{\mathcal K}_{\psi}$, $\xi\in{\mathcal H}$.  

For convenience, denote the left side of Equation~\eqref{(eqn_U^*)} as (LHS) and the right side 
as (RHS).  We have:
\begin{align}
{\text {(RHS)}}\,&=(\psi\otimes\eta)\bigl((c^*\otimes d^*)z(1\otimes b)\bigr)  \notag \\
&=(\psi\otimes\eta\otimes\varphi)\bigl((c^*\otimes d^*\otimes1)\Delta_{13}(r^*y)\Delta_{23}(s)
(1\otimes b\otimes1)\bigr)   \notag \\
&=(\psi\otimes\eta\otimes\varphi)\bigl((c^*\otimes d^*\otimes1)\Delta_{13}(r^*y)(1\otimes E)\Delta_{23}(s)
(1\otimes b\otimes1)\bigr),
\notag
\end{align}
because $\Delta s=E(\Delta s)$.  Meanwhile, $(\operatorname{id}\otimes\operatorname{id}
\otimes\varphi)[\Delta_{13}(r^*y)\Delta_{23}(s)(1\otimes b\otimes 1)]\in{\mathcal D}(Q_R)$, because 
$(\Delta s)(b\otimes 1)\in{\mathfrak N}_{\operatorname{id}\otimes\varphi}$ and 
$r^*y\in{\mathfrak N}_{\varphi}^*$.  By the definition of the $Q_R$ map, we can write:
$$
{\text {(RHS)}}\,=(\psi\otimes\eta)\bigl((c^*\otimes d^*)\,Q_R((\operatorname{id}\otimes\operatorname{id}
\otimes\varphi)[\Delta_{13}(r^*y)\Delta_{23}(s)(1\otimes b\otimes 1)])\bigr).
$$

Recall Proposition~\ref{Q_Rproposition} and Proposition~\ref{Q_Rrightinvariance}, where we have:
$$(\psi\otimes\operatorname{id})\bigl((c^*\otimes1)\,Q_R(p\otimes\tilde{b})\bigr)
=(\psi\otimes\operatorname{id})\bigl(Q_R(c^*p\otimes\tilde{b})\bigr)=(\psi\otimes\operatorname{id})
\bigl(\Delta(c^*p)(1\otimes\tilde{b})\bigr),$$
for an appropriate $p\otimes\tilde{b}$.  Applying this result to our case, we obtain:
$$
{\text {(RHS)}}=\eta\bigl(d^*(\psi\otimes\operatorname{id}\otimes\varphi)[(\Delta\otimes\operatorname{id})
((c^*\otimes1)\Delta(r^*y))\Delta_{23}(s)(1\otimes b\otimes1)]\bigr).
$$
By the coassociativity of $\Delta$ and the fact that $\Delta$ is a homomorphism, this becomes:
\begin{align}\label{(RHSnew)}
{\text {(RHS)}}\,&=\eta\bigl(d^*(\psi\otimes\operatorname{id}\otimes\varphi)[(\Delta(c^*)\otimes1)
(\operatorname{id}\otimes\Delta)(\Delta(r^*y))\Delta_{23}(s)(1\otimes b\otimes1)]\bigr)
\notag \\
&=\eta\bigl(d^*(\psi\otimes\operatorname{id}\otimes\varphi)((\Delta(c^*)\otimes1)(\operatorname{id}\otimes\Delta)
[\Delta(r^*y)(1\otimes s)])b\bigr).
\end{align}

Let us now turn our attention to computing the (LHS).  Note first that by Proposition~\ref{Vdefn}, we have: 
$$
{\text {(LHS)}}=\bigl\langle\Lambda_{\psi}(a)\otimes\Lambda(b),(\Lambda_{\psi}\otimes\Lambda)
((\Delta c)(1\otimes d))\bigr\rangle
=(\psi\otimes\eta)\bigl((1\otimes d^*)\Delta(c^*)(a\otimes b)\bigr).
$$
Since $a=(\operatorname{id}\otimes\varphi)\bigl(\Delta(r^*y)(1\otimes s)\bigr)$, 
this becomes:
$$
{\text {(LHS)}}\,=(\psi\otimes\eta\otimes\varphi)\bigl((1\otimes d^*\otimes1)(\Delta(c^*)\otimes1)
\Delta_{13}(r^*y)(1\otimes1\otimes s)(1\otimes b\otimes 1)\bigr).
$$

Note here that $(\psi\otimes\operatorname{id}\otimes\operatorname{id})
\bigl((1\otimes d^*\otimes 1)(\Delta(c^*)\otimes1)\Delta_{13}(r^*y)\bigr)\in{\mathcal D}(Q_{\lambda})$, 
because $(1\otimes d^*)\Delta(c^*)\in{\mathfrak N}_{\psi\otimes\operatorname{id}}^*$ 
and $r^*y\in{\mathfrak N}_{\psi}$. Applying the $Q_{\lambda}$ map and using $\Delta(c^*)E
=\Delta(c^*)$, we can see easily that (LHS) now becomes:
$$
=\eta\bigl((\operatorname{id}\otimes\varphi)[Q_{\lambda}((\psi\otimes\operatorname{id}
\otimes\operatorname{id})[(1\otimes d^*\otimes 1)(\Delta(c^*)\otimes1)\Delta_{13}(r^*y)])
(1\otimes s)]b\bigr).
$$
Next, use the fact (from Proposition~\ref{Q_L'}) that $Q_{\lambda}(\tilde{c}\otimes q)(1\otimes s)
=Q_{\lambda}(\tilde{c}\otimes qs)$,  and also that $(\operatorname{id}\otimes\varphi)
\bigl(Q_{\lambda}(\tilde{c}\otimes qs)\bigr)=(\operatorname{id}\otimes\varphi)
\bigl((\tilde{c}\otimes1)\Delta(qs)\bigr)$, from Proposition~\ref{Q_L'leftinvariance}.  Then 
we have:
\begin{equation}\label{(LHSnew)}
{\text {(LHS)}}
=\eta\bigl((\psi\otimes\operatorname{id}\otimes\varphi)((1\otimes d^*\otimes 1)
(\Delta(c^*)\otimes1)(\operatorname{id}\otimes\Delta)[\Delta(r^*y)(1\otimes s)])b\bigr).
\end{equation}

Comparing Equation~\eqref{(RHSnew)} with Equation~\eqref{(LHSnew)}, we can 
finally see that (RHS)=(LHS), thereby finishing the proof.
\end{proof}

The operators $U$ and $U^*$ are adjoints of each other, and we now know how they 
behave.  In addition, with the notation as above, we have:
\begin{align}\label{(UU*=E)}
UU^*\bigl(\Lambda_{\psi}(a)\otimes\Lambda(b)\bigr)
&=U\bigl((\Lambda_{\psi}\otimes\Lambda)[z(1\otimes b)]\bigr)  \notag \\
&=V\bigl((\Lambda_{\psi}\otimes\Lambda)[z(1\otimes b)]\bigr)
=E\bigl(\Lambda_{\psi}(a)\otimes\Lambda(b)\bigr).
\end{align}
The first equality is the result of Proposition~\ref{U^*} and the last equality is using
Proposition~\ref{V^*prep}.  Since the $\Lambda_{\psi}(a)\otimes\Lambda(b)$ span 
a dense subset in ${\mathcal K}_{\psi}\otimes{\mathcal H}$, this shows that we have: 
$UU^*=E|_{{\mathcal K}_{\psi}\otimes{\mathcal H}}$.  See below:

\begin{prop}\label{UU^*}
Recall that $U=V|_{{\mathcal K}_{\psi}\otimes{\mathcal H}}$.  We have: 
\begin{enumerate}
  \item $UU^*=E|_{{\mathcal K}_{\psi}\otimes{\mathcal H}}$;
  \item $UU^*U=U$.
\end{enumerate}
This shows that $U$ is a partial isometry in ${\mathcal B}({\mathcal K}_{\psi}\otimes{\mathcal H})$.
\end{prop}

\begin{proof}
(1). Already shown above.  See Equation~\eqref{(UU*=E)}.

(2). We know from Proposition~\ref{Vprop} that $EV=V$.  So $EU=U$, because we saw 
from Equation~\eqref{(RangeVinRangeE)} that $E({\mathcal K}_{\psi}\otimes{\mathcal H})
\subseteq{\mathcal K}_{\psi}\otimes{\mathcal H}$.  It follows that $UU^*U=EU=U$.
\end{proof}

It is clear that $U^*$ is also a partial isometry (as $U^*UU^*=U^*$).  By the general theory on 
partial isometries, $UU^*$ is a projection onto $\operatorname{Ran}(UU^*)=\operatorname{Ran}(U)
=\operatorname{Ker}(U^*)^\perp$, while $U^*U$ is a projection onto $\operatorname{Ran}(U^*U)
=\operatorname{Ran}(U^*)=\operatorname{Ker}(U)^\perp$.  These spaces are necessarily closed 
in ${\mathcal K}_{\psi}\otimes{\mathcal H}$.  This means that we can write: ${\mathcal K}_{\psi}
\otimes{\mathcal H}=\operatorname{Ker}(U)\oplus\operatorname{Ran}(U^*U)$, as well as 
${\mathcal K}_{\psi}\otimes{\mathcal H}=\operatorname{Ker}(U^*)\oplus\operatorname{Ran}(UU^*)$. 
In addition, $U$ is an isometry from $\operatorname{Ran}(U^*U)$ onto $\operatorname{Ran}(UU^*)$, 
and similarly, $U^*$ is an isometry from $\operatorname{Ran}(UU^*)$ onto $\operatorname{Ran}(U^*U)$. 
All these are standard results.

We next turn our attention to showing that ${\mathcal K}_{\psi}={\mathcal H}_{\psi}$.  We need some 
preparation.  Note first that 
\begin{equation}\label{(proofeq1)}
{\mathcal H}_{\psi}\otimes{\mathcal H}=\operatorname{Ker}(V)\oplus\operatorname{Ker}(V)^\perp
=\operatorname{Ker}(V)\oplus\overline{\operatorname{Ran}(V^*)}.
\end{equation}
Similarly, 
\begin{equation}\label{(proofeq2)}
{\mathcal H}_{\psi}\otimes{\mathcal H}=\operatorname{Ker}(V^*)\oplus\overline{\operatorname{Ran}(V)}.
\end{equation}

Let us look more carefully about the subspaces $\overline{\operatorname{Ran}(V)}$ and $\overline
{\operatorname{Ran}(V^*)}$.  First, by Equation~\eqref{(RangeV)} and Proposition~\ref{UU^*}, 
we have: 
$$
\overline{\operatorname{Ran}(V)}=\overline{V({\mathcal H}_{\psi}\otimes{\mathcal H})}
=\overline{V({\mathcal K}_{\psi}\otimes{\mathcal H})}
=E({\mathcal K}_{\psi}\otimes{\mathcal H})
=\operatorname{Ran}(UU^*).
$$
Meanwhile, considering Equation~\eqref{(proofeq2)} and the previous equation, we have:
$$
\overline{\operatorname{Ran}(V^*)}=\overline{V^*({\mathcal H}_{\psi}\otimes{\mathcal H})}
=\overline{V^*(\operatorname{Ran}(V))}
=\overline{V^*(\operatorname{Ran}(UU^*))},
$$
which is actually no different than $\overline{U^*(\operatorname{Ran}(UU^*))}=
\operatorname{Ran}(U^*)$, by the property of the partial isometry $U^*$.  In particular, 
we see that $\overline{\operatorname{Ran}(V^*)}\subseteq{\mathcal K}_{\psi}\otimes{\mathcal H}$.

Here is our main result of this subsection: 

\begin{theorem}\label{K_psi=H_psi}
(1). Let ${\mathcal K}_{\psi}$ be the subspace of ${\mathcal H}_{\psi}$, as in Equation~\eqref{(K_psi)} 
and Proposition~\ref{K_psinew}.  We actually have: ${\mathcal K}_{\psi}={\mathcal H}_{\psi}$.

(2). $V$ is a partial isometry such that $VV^*=E$, where we regard $E=(\pi_{\psi}\otimes\pi)(E)$. 
In particular, we see that $\operatorname{Ran}(V)$ is closed, and that
$$
\operatorname{Ran}(V)=\operatorname{Ran}(E).
$$
\end{theorem}

\begin{proof}
(1). 
Suppose ${\mathcal K}_{\psi}\ne{\mathcal H}_{\psi}$, and let $x\in{\mathfrak N}_{\psi}$ be 
such that $\Lambda_{\psi}(x)\in{\mathcal H}_{\psi}\ominus{\mathcal K}_{\psi}$.  Then for any 
$\xi\in{\mathcal H}$, we have $\Lambda_{\psi}(x)\otimes\xi\in({\mathcal H}_{\psi}\ominus
{\mathcal K}_{\psi})\otimes{\mathcal H}$.  Since ${\mathcal H}_{\psi}\otimes{\mathcal H}
={\mathcal K}_{\psi}\otimes{\mathcal H}\,\oplus\,({\mathcal H}_{\psi}\ominus{\mathcal K}_{\psi})
\otimes{\mathcal H}$, and since we saw above that $\overline{\operatorname{Ran}(V^*)}\subseteq
{\mathcal K}_{\psi}\otimes{\mathcal H}$, the observation in Equation~\eqref{(proofeq1)} means 
that we must have $({\mathcal H}_{\psi}\ominus{\mathcal K}_{\psi})\otimes{\mathcal H}
\subseteq\operatorname{Ker}(V)$. So $V\bigl(\Lambda_{\psi}(x)\otimes\xi)
=0_{{\mathcal H}_{\psi}\otimes{\mathcal H}}$.

It follows that for any $\zeta\in{\mathcal H}$, we will have:
$$
\bigl((\operatorname{id}\otimes\omega_{\xi,\zeta})(V)\bigr)\Lambda_{\psi}(x)
=\bigl\langle V\bigl(\Lambda_{\psi}(x)\otimes\xi),\,\cdot\,\otimes\zeta\bigr\rangle=0.
$$
By Definition~\ref{Vdefn}, this means that
$$
\Lambda_{\psi}\bigl((\operatorname{id}\otimes\omega_{\xi,\zeta})(\Delta x)\bigr)=0.
$$
Since $\xi,\zeta\in{\mathcal H}$ are arbitrary, this means that $x=0$.  This in turn means 
that ${\mathcal H}_{\psi}\ominus{\mathcal K}_{\psi}=\{0\}$, or ${\mathcal K}_{\psi}
={\mathcal H}_{\psi}$.

(2). Since ${\mathcal K}_{\psi}={\mathcal H}_{\psi}$, it follows that $U=V\in{\mathcal B}
({\mathcal H}_{\psi}\otimes{\mathcal H})$.  By Proposition~\ref{UU^*}, we thus conclude 
that $V$ is a partial isometry such that $VV^*=E$.  In particular, 
$\operatorname{Ran}(V)$ is closed, and $\operatorname{Ran}(V)=\operatorname{Ran}(E)$.
\end{proof}

Since $V$ is a partial isometry, it follows that the operator $V^*V$ is also a projection. 
To learn more about this projection, let $r,s\in{\mathfrak N}_{\varphi}$, $y\in{\mathfrak N}_{\psi}$, 
and let $p=(\operatorname{id}\otimes\varphi)\bigl(\Delta(r^*y)(1\otimes s)\bigr)$. Then 
consider $\Lambda_{\psi}(p)\otimes\Lambda(b)\in{\mathcal H}_{\psi}\otimes{\mathcal H}$, 
where $b\in{\mathfrak N}_{\eta}$.  By Theorem~\ref{K_psi=H_psi}, such elements span 
a dense subspace in the Hilbert space ${\mathcal H}_{\psi}\otimes{\mathcal H}$.  Meanwhile, 
recall from Proposition~\ref{Q_Rproposition} that $p\otimes b\in{\mathcal D}(Q_R)$.  Now define:
\begin{equation}\label{(G_R)}
G_R\bigl(\Lambda_{\psi}(p)\otimes\Lambda(b)\bigr):=(\Lambda_{\psi}\otimes\Lambda)
\bigl(Q_R(p\otimes b)\bigr).
\end{equation}
Here are the properties of $G_R$:

\begin{prop}\label{G_R}
Let the notation be as in the previous paragraph.  Then:
\begin{enumerate}
  \item $G_R$ is self-adjoint and idempotent, so it determines a bounded operator 
$G_R\in{\mathcal B}({\mathcal H}_{\psi}\otimes{\mathcal H})$;
  \item The operator $G_R$ is characterized by 
$$
(G_R\otimes\operatorname{id})\bigl((\Lambda_{\psi}\otimes\Lambda\otimes\Lambda)
(\Delta_{13}(a)(1\otimes b\otimes x))\bigr)
=(\Lambda_{\psi}\otimes\Lambda\otimes\Lambda)\bigl(\Delta_{13}(a)(1\otimes E)
(1\otimes b\otimes x)\bigr),
$$  
for $a\in{\mathfrak N}_{\psi}$, $b,x\in{\mathfrak N}_{\eta}$.
  \item $G_R=V^*V$;
  \item $\operatorname{Ran}(V^*)$ is a closed subspace in ${\mathcal H}_{\psi}\otimes
{\mathcal H}$, and we have: $\operatorname{Ran}(V^*)=\operatorname{Ran}(G_R)$.
\end{enumerate}
\end{prop}

\begin{proof}
(1). $G_RG_R=G_R$ is clear, because we know $Q_RQ_R=Q_R$ from Proposition~\ref{Q_Ridempotent}. 
To show that it is self-adjoint, suppose $q\in{\mathfrak N}_{\psi}$ is of a similar type 
as $p$, that is $q=(\operatorname{id}\otimes\varphi)\bigl(\Delta(u^*x)(1\otimes v)\bigr)$, 
for $u,v\in{\mathfrak N}_{\varphi}$, $x\in{\mathfrak N}_{\psi}$, and consider 
$\Lambda_{\psi}(q)\otimes\Lambda(d)\in{\mathcal H}_{\psi}\otimes{\mathcal H}$, where 
$d\in{\mathfrak N}_{\eta}$.  Then: 
\begin{align}
&\bigl\langle G_R(\Lambda_{\psi}(p)\otimes\Lambda(b)),\Lambda_{\psi}(q)\otimes\Lambda(d)\bigr\rangle
\notag \\
&=(\psi\otimes\eta)\bigl((q^*\otimes d^*)Q_R(p\otimes b)\bigr)
=\eta\bigl(d^*\,(\psi\otimes\operatorname{id})[(q^*\otimes1)Q_R(p\otimes b)]\bigr).
\notag
\end{align}
But by Proposition~\ref{Q_Rproposition} and Equation~\eqref{(Q_Rrightinvariancenew)}, we have: 
$$
(\psi\otimes\operatorname{id})\bigl((q^*\otimes1)Q_R(p\otimes b)\bigr)
=(\psi\otimes\operatorname{id})\bigl(Q_R(q^*p\otimes b)\bigr)
=(\psi\otimes\operatorname{id})\bigl(\Delta(q^*p)(1\otimes b)\bigr).
$$
Returning to the earlier equation, we now have:
\begin{align}
&\bigl\langle G_R(\Lambda_{\psi}(p)\otimes\Lambda(b)),\Lambda_{\psi}(q)\otimes\Lambda(d)\bigr\rangle
\notag \\
&=\eta\bigl(d^*\,(\psi\otimes\operatorname{id})[\Delta(q^*p)(1\otimes b)]\bigr)
=(\psi\otimes\eta)\bigl((1\otimes d^*)\Delta(q^*)\Delta(p)(1\otimes b)\bigr).
\notag
\end{align}
Using the same idea, we also have:
\begin{align}
&\bigl\langle\Lambda_{\psi}(p)\otimes\Lambda(b),G_R(\Lambda_{\psi}(q)\otimes\Lambda(d))\bigr\rangle
\notag \\
&=\overline{\bigl\langle G_R(\Lambda_{\psi}(q)\otimes\Lambda(d)),\Lambda_{\psi}(p)\otimes\Lambda(b)\bigr\rangle}
=\overline{(\psi\otimes\eta)\bigl((1\otimes b^*)\Delta(p^*q)(1\otimes d)\bigr)}
\notag \\
&=(\psi\otimes\eta)\bigl((1\otimes d^*)\Delta(q^*)\Delta(p)(1\otimes b)\bigr).
\notag
\end{align}
In this way, we showed that
$$
\bigl\langle G_R(\Lambda_{\psi}(p)\otimes\Lambda(b)),
\Lambda_{\psi}(q)\otimes\Lambda(d)\bigr\rangle=\bigl\langle\Lambda_{\psi}(p)\otimes\Lambda(b),
G_R(\Lambda_{\psi}(q)\otimes\Lambda(d))\bigr\rangle.
$$

We thereby observe that $G_R$ is symmetric and idempotent, defined on a dense subspace of 
the Hilbert space ${\mathcal H}_{\psi}\otimes{\mathcal H}$.  This means $G_R$ extends to an orthogonal 
projection, a bounded self-adjoint operator in ${\mathcal B}({\mathcal H}_{\psi}\otimes{\mathcal H})$.

(2). This follows from the characterization of the $Q_R$ map, given in 
Proposition~\ref{Q_Rcharacterization}.  Unlike the case of $Q_R$, there is no 
issue about the domain in this case, because $G_R$ is a bounded operator in 
${\mathcal B}({\mathcal H}_{\psi}\otimes{\mathcal H})$.

(3). From the computations in (1) above, and using the definition of the operator $V$ earlier 
(Proposition~\ref{Vdefn}), we can see that 
\begin{align}
&\bigl\langle G_R(\Lambda_{\psi}(p)\otimes\Lambda(b)),\Lambda_{\psi}(q)\otimes\Lambda(d)\bigr\rangle
=(\psi\otimes\eta)\bigl((1\otimes d^*)\Delta(q^*)\Delta(p)(1\otimes b)\bigr)
\notag \\
&=\bigl\langle(\Lambda_{\psi}\otimes\Lambda)(\Delta(p)(1\otimes b)),
(\Lambda_{\psi}\otimes\Lambda)(\Delta(q)(1\otimes d))\bigr\rangle
\notag \\
&=\bigl\langle V(\Lambda_{\psi}(p)\otimes\Lambda(b)),V(\Lambda_{\psi}(q)\otimes\Lambda(d))\bigr\rangle
\notag \\
&=\bigl\langle V^*V(\Lambda_{\psi}(p)\otimes\Lambda(b)),\Lambda_{\psi}(q)\otimes\Lambda(d)\bigr\rangle.
\notag
\end{align}
Since the $\Lambda_{\psi}(p)\otimes\Lambda(b)$ span a dense subset in the Hilbert space 
${\mathcal H}_{\psi}\otimes{\mathcal H}$, we conclude that $G_R=V^*V\in{\mathcal B}
({\mathcal H}_{\psi}\otimes{\mathcal H})$.

(4). This follows from the fact that $V^*$ is a partial isometry such that $V^*V=G_R$.  See the discussion 
immediately following Proposition~\ref{UU^*}.
\end{proof}

\begin{cor}
In a similar way, we can define a self-adjoint idempotent operator $G_{\rho}\in{\mathcal B}
({\mathcal H}_{\psi}\otimes{\mathcal H})$, given by the $Q_{\rho}$ map at the algebra level.
\end{cor}

\subsection{$W$ is a partial isometry}

In Proposition~\ref{Wdefn} and Proposition~\ref{Wprop}, we gave the definition and some 
basic properties of $W\in{\mathcal B}({\mathcal H}\otimes{\mathcal H}_{\varphi})$, 
which is the operator corresponding to the left regular representation.  All the earlier 
results concerning the operator $V$ will have analogous results for the operator $W$. 
The argument is essentially no different.  However, for the purpose of clarifying the 
notation for the later sections, we will gather the relevant results here (without proof).

\begin{prop}\label{K_phi}
We have the following characterizations for the Hilbert space ${\mathcal H}_{\varphi}$.
\begin{align}
{\mathcal H}_{\varphi}&=\overline{\operatorname{span}}^{\|\ \|}\bigl\{\Lambda_{\varphi}
((\psi\otimes\operatorname{id})[\Delta(z^*r)(y\otimes1)]):y,z\in{\mathfrak N}_{\psi},
r\in{\mathfrak N}_{\varphi}\bigr\}   \notag \\
&=\overline{\operatorname{span}}^{\|\ \|}\bigl\{\Lambda_{\varphi}((\theta\otimes\operatorname{id})
(\Delta x)):x\in{\mathfrak N}_{\varphi},\theta\in A^*\bigr\}.
\notag
\end{align}
\end{prop}

\begin{prop}\label{Woperator}
Let $p=(\psi\otimes\operatorname{id})\bigl(\Delta(z^*r)(y\otimes1)\bigr)$, where 
$y,z\in{\mathfrak N}_{\psi}$ and $r\in{\mathfrak N}_{\varphi}$.  Also let 
$c\in{\mathfrak N}_{\eta}$.  Then 
$W\in{\mathcal B}({\mathcal H}\otimes{\mathcal H}_{\varphi})$ is characterized by
$$
W:\Lambda(c)\otimes\Lambda_{\varphi}(p)\mapsto(\Lambda\otimes\Lambda_{\varphi})
\bigl(w(c\otimes1)\bigr),
$$
where $w=(\psi\otimes\operatorname{id}\otimes\operatorname{id})\bigl(\Delta_{13}
(z^*r)\Delta_{12}(y)(1\otimes c\otimes 1)\bigr)$.
\end{prop}

\begin{theorem}\label{Wpartialisometry}
$W$ and $W^*$ are partial isometries in ${\mathcal B}({\mathcal H}\otimes
{\mathcal H}_{\varphi})$, with the following properties:
\begin{enumerate}
  \item $W^*W=E$, where we regard $E=(\pi\otimes\pi_{\varphi})(E)$, as an element 
in ${\mathcal B}({\mathcal H}\otimes{\mathcal H}_{\varphi})$.  In particular, 
the subspace $\operatorname{Ran}(W^*)$ is closed and we have:
$$
\operatorname{Ran}(W^*)=\operatorname{Ran}(E)\,(\subseteq{\mathcal H}\otimes
{\mathcal H}_{\varphi}).
$$
  \item $WW^*=G_L$, where $G_L\in{\mathcal B}({\mathcal H}\otimes{\mathcal H}_{\varphi})$ 
is a (bounded) self-adjoint idempotent operator given by
$$
G_L:\Lambda(c)\otimes\Lambda_{\varphi}(p)\mapsto(\Lambda\otimes\Lambda_{\varphi})
\bigl(Q_L(c\otimes p)\bigr).
$$
Here, we are using the same notation as in Proposition~\ref{Woperator}.  By definition of 
the $Q_L$ map (see Proposition~\ref{Q_L}), the operator $G_L$ is thus characterized by 
$$
(\operatorname{id}\otimes G_L)\bigl((\Lambda\otimes\Lambda\otimes\Lambda_{\varphi})
(\Delta_{13}(a)(y\otimes c\otimes 1))\bigr)=(\Lambda\otimes\Lambda\otimes\Lambda_{\varphi})
\bigl(\Delta_{13}(a)(E\otimes1)(y\otimes c\otimes 1)\bigr),
$$
for $a\in{\mathfrak N}_{\varphi}$, $y,c\in{\mathfrak N}_{\eta}$.
  \item We also have: $\operatorname{Ran}(W)=\operatorname{Ran}(G_L)$, as a closed 
  subspace in ${\mathcal H}\otimes{\mathcal H}_{\varphi}$.
  \item In a similar way, we can define a self-adjoint idempotent operator $G_{\lambda}
\in{\mathcal B}({\mathcal H}\otimes{\mathcal H}_{\varphi})$, given by the $Q_{\lambda}$ map 
at the algebra level.
\end{enumerate}
\end{theorem}

\subsection{Multiplicativity properties of $W$ and $V$}

We indicated earlier that the operators $V$ and $W$ behave much like the multiplicative 
unitary operators in the case of locally compact quantum groups.  In this subsection, 
we give results that strengthen this point.   For convenience, we give here the results 
regarding the operator $W$.  Of course, similar results exist for the operator $V$.

Let us begin with a result that will be useful down the road.

\begin{prop}\label{idomegaW}
If $r,s\in{\mathfrak N}_{\varphi}$.  Then we have:
$$
(\operatorname{id}\otimes\omega_{\Lambda_{\varphi}(s),\Lambda_{\varphi}(r)})(W)
=(\operatorname{id}\otimes\varphi)\bigl((\Delta(r^*)(1\otimes s)\bigr).
$$
\end{prop}

\begin{proof}
Let $c,d\in{\mathfrak N}_{\eta}$ be arbitrary.  Then
\begin{align}
\bigl\langle(\operatorname{id}\otimes\omega_{\Lambda_{\varphi}(s),\Lambda_{\varphi}(r)})(W)
\Lambda(c),\Lambda(d)\bigr\rangle
&=\bigl\langle\Lambda(c),(\operatorname{id}\otimes\omega_{\Lambda_{\varphi}(r),
\Lambda_{\varphi}(s)})(W^*)\Lambda(d)\bigr\rangle  \notag \\
&=\bigl\langle\Lambda(c)\otimes\Lambda_{\varphi}(s),W^*(\Lambda(d)\otimes\Lambda_{\varphi}(r))
\bigr\rangle,
\notag
\end{align}
where we used the result of Lemma~\ref{omega_xizetaLem}\,(2).  By using the definition 
of $W^*$ (see Proposition~\ref{Wdefn}), this becomes:
\begin{align}
&=\bigl\langle\Lambda(c)\otimes\Lambda_{\varphi}(s),(\Lambda\otimes\Lambda_{\varphi})
((\Delta r)(d\otimes 1))\bigr\rangle
\notag \\
&=(\eta\otimes\varphi)\bigl((d^*\otimes1)\Delta(r^*)(c\otimes s)\bigr)
=\eta\bigl(d^*(\operatorname{id}\otimes\varphi)[\Delta(r^*)(1\otimes s)]c\bigr)
\notag \\
&=\bigl\langle\Lambda((\operatorname{id}\otimes\varphi)[\Delta(r^*)(1\otimes s)]c,
\Lambda(d)\bigr\rangle   \notag \\
&=\bigl\langle\pi((\operatorname{id}\otimes\varphi)[\Delta(r^*)(1\otimes s)])
\Lambda(c),\Lambda(d)\bigr\rangle.
\notag
\end{align}
This is true for any $c,d\in{\mathfrak N}_{\eta}$, which is dense in ${\mathcal H}$. 
It follows that 
$$
(\operatorname{id}\otimes\omega_{\Lambda_{\varphi}(s),\Lambda_{\varphi}(r)})(W)
=(\operatorname{id}\otimes\varphi)\bigl((\Delta(r^*)(1\otimes s)\bigr),
$$
with the convention that $\pi(A)=A$.
\end{proof}

For any $\omega\in{\mathcal B}({\mathcal H}_{\varphi})_*$, since it can be approximated 
by the $\omega_{\Lambda_{\varphi}(s),\Lambda_{\varphi}(r)}$, we see from 
Proposition~\ref{idomegaW} that $(\operatorname{id}\otimes\omega)(W)\in \pi(A)=A$. 
Furthermore, we have the following useful characterization for our $C^*$-algebra $A$, 
which resembles the result from the quantum group theory.

\begin{prop}\label{idomegaWclosureA}
$$
A=\pi(A)=\overline{\bigl\{(\operatorname{id}\otimes\omega)(W):
\omega\in{\mathcal B}({\mathcal H}_{\varphi})_*\bigr\}}^{\|\ \|}\,\bigl(\subseteq{\mathcal B}
({\mathcal H})\bigr).
$$
\end{prop}

\begin{proof}
We saw that $\operatorname{span}\bigl\{(\operatorname{id}\otimes\varphi)
(\Delta(a^*)(1\otimes x)):a,x\in{\mathfrak N}_{\varphi}\bigr\}$ is norm-dense 
in $A$ (see Lemma~\ref{densesubsetA}).  But by Proposition~\ref{idomegaW} above, 
we know that $(\operatorname{id}\otimes\varphi)(\Delta(a^*)(1\otimes x))
=(\operatorname{id}\otimes\omega_{\Lambda_{\varphi}(x),\Lambda_{\varphi}(a)})(W)$. 
The result follows immediately.
\end{proof}

The next result also looks quite familiar:

\begin{prop}\label{DeltaW}
We have:
$$(\pi\otimes\pi_{\varphi})(\Delta x)=W^*\bigl(1\otimes \pi_{\varphi}(x)\bigr)W,
$$ 
for all $x\in A$.
\end{prop}

\begin{proof}
We saw in Proposition~\ref{Wprop}\,(1) that $W^*\bigl(1\otimes\pi_{\varphi}(x)\bigr)
=(\pi\otimes\pi_{\varphi})(\Delta x)W^*$.  It follows that
\begin{align}
W^*\bigl(1\otimes\pi_{\varphi}(x)\bigr)W&=(\pi\otimes\pi_{\varphi})(\Delta x)W^*W
=(\pi\otimes\pi_{\varphi})(\Delta x)(\pi\otimes\pi_{\varphi})(E) \notag \\
&=(\pi\otimes\pi_{\varphi})\bigl((\Delta x)E\bigr)
=(\pi\otimes\pi_{\varphi})(\Delta x).
\notag
\end{align}
\end{proof}

In what follows, we show some properties that support the idea that $W$ is, in a sense, 
a ``multiplicative partial isometry''.  This notion is not as completely axiomatized as in 
\cite{BS}, \cite{Wr7} for multiplicative unitaries, but some work in this direction has appeared 
in \cite{BSzMultIso} in the finite-dimensional case.  The second-named author (Kahng) is 
currently working on the axiomatic approach \cite{BJK_mpi}.

For the rest of this section, let us take ${\mathcal H}={\mathcal H}_{\varphi}$ (so $\eta=\varphi$), 
so that $W\in{\mathcal B}({\mathcal H}\otimes{\mathcal H})$.  This assumption is not as restrictive 
as it may sound, because later (Section~5 and Part~III), we will be able to identify the 
Hilbert spaces ${\mathcal H}={\mathcal H}_{\psi}={\mathcal H}_{\varphi}$.  

\begin{prop}\label{Wpentagon}
In ${\mathcal B}({\mathcal H}\otimes{\mathcal H}\otimes{\mathcal H})$, we have: 
$$
W_{12}W_{13}W_{23}=W_{23}W_{12}.
$$
\end{prop}

\begin{proof}
For $a,b,c\in{\mathfrak N}_{\varphi}$, by definition of $W$ (Proposition~\ref{Wdefn}), 
we have:
\begin{align}
&W_{12}^*W_{23}^*\bigl(\Lambda(a)\otimes\Lambda(b)\otimes\Lambda(c)\bigr)
=W_{12}^*\bigl((\Lambda\otimes\Lambda\otimes\Lambda)(a\otimes[(\Delta c)(b\otimes1)])\bigr)
\notag \\
&=(\Lambda\otimes\Lambda\otimes\Lambda)\bigl((\Delta\otimes\operatorname{id})
[(\Delta c)(b\otimes1)](a\otimes1\otimes1)\bigr),
\notag
\end{align}
while 
\begin{align}
&W_{23}^*W_{13}^*W_{12}^*\bigl(\Lambda(a)\otimes\Lambda(b)\otimes\Lambda(c)\bigr)
=W_{23}^*W_{13}^*\bigl((\Lambda\otimes\Lambda\otimes\Lambda)([(\Delta b)(a\otimes1)]
\otimes c)\bigr)    \notag \\
&=W_{23}^*\bigl((\Lambda\otimes\Lambda\otimes\Lambda)(\Delta_{13}(c)\Delta_{12}(b)
(a\otimes1\otimes1)\bigr)   \notag \\
&=(\Lambda\otimes\Lambda\otimes\Lambda)\bigl((\operatorname{id}\otimes\Delta)(\Delta c)
\Delta_{12}(b)(a\otimes1\otimes1)\bigr)   \notag \\
&=(\Lambda\otimes\Lambda\otimes\Lambda)\bigl((\Delta\otimes\operatorname{id})(\Delta c)
\Delta_{12}(b)(a\otimes1\otimes1)\bigr)  \notag \\
&=(\Lambda\otimes\Lambda\otimes\Lambda)\bigl((\Delta\otimes\operatorname{id})
[(\Delta c)(b\otimes1)](a\otimes1\otimes1)\bigr).
\notag
\end{align}
In the next to last equation, we used the coassociativity property of $\Delta$.  Comparing, 
since $a,b,c\in{\mathfrak N}_{\varphi}$ are arbitrary, we see that $W_{12}^*W_{23}^*
=W_{23}^*W_{13}^*W_{12}^*$.  Or, $W_{12}W_{13}W_{23}=W_{23}W_{12}$.
\end{proof}

The above result resembles the ``pentagon equation'' for the multiplicative unitaries. 
As in that case, we needed to use the coassociativity of $\Delta$.  However, since 
the operator $W$ is no longer unitary, many results that would automatically 
follow from the pentagon equation need separate proofs in our case.  There are 
several such results and we collect some of them here, which will be useful later.  

\begin{prop}\label{Wpentagon_alt}
In ${\mathcal B}({\mathcal H}\otimes{\mathcal H}\otimes{\mathcal H})$, we have: 
\begin{enumerate}
  \item $W_{13}W_{23}W_{23}^*=W_{12}^*W_{12}W_{13}$
  \item $W_{12}^*W_{23}W_{12}=W_{13}W_{23}$
  \item $W_{23}W_{12}W_{23}^*=W_{12}W_{13}$
  \item $W_{12}W_{12}^*W_{23}=W_{23}W_{12}W_{12}^*$
  \item $W_{12}W_{23}^*W_{23}=W_{23}^*W_{23}W_{12}$
\end{enumerate}
\end{prop}

\begin{proof}
(1). For $a,b,c\in{\mathfrak N}_{\varphi}$, using the fact that $W^*W=E$ and the definition 
of $W$, we have: 
\begin{align}
&W_{13}^*W_{12}^*W_{12}\bigl(\Lambda(a)\otimes\Lambda(b)\otimes\Lambda(c)\bigr)
=W_{13}^*\bigl((\Lambda\otimes\Lambda\otimes\Lambda)(E(a\otimes b)\otimes c)\bigr)
\notag \\
&=(\Lambda\otimes\Lambda\otimes\Lambda)\bigl(\Delta_{13}(c)(E\otimes1)(a\otimes b\otimes1)\bigr).
\notag
\end{align}
Meanwhile, using the fact that $WW^*=G_L$ (Theorem~\ref{Wpartialisometry}), we have:
\begin{align}
&W_{23}W_{23}^*W_{13}^*\bigl(\Lambda(a)\otimes\Lambda(b)\otimes\Lambda(c)\bigr)
=W_{23}W_{23}^*\bigl((\Lambda\otimes\Lambda\otimes\Lambda)(\Delta_{13}(c)(a\otimes b\otimes1))\bigr)
\notag \\
&=(\Lambda\otimes\Lambda\otimes\Lambda)\bigl(\Delta_{13}(c)(E\otimes1)(a\otimes b\otimes1)\bigr).
\notag
\end{align}
Since $a,b,c\in{\mathfrak N}_{\varphi}$ are arbitrary, we see from above that 
$W_{13}^*W_{12}^*W_{12}=W_{23}W_{23}^*W_{13}^*$.  Or, by taking adjoints, 
$W_{13}W_{23}W_{23}^*=W_{12}^*W_{12}W_{13}$.

(2). By Proposition~\ref{Wpentagon} (pentagon equation) and from (1) above, we have:
$$
W_{12}^*W_{23}W_{12}=W_{12}^*W_{12}W_{13}W_{23}=W_{13}W_{23}W_{23}^*W_{23}=W_{13}W_{23},
$$
where we used the fact that $WW^*W=W$ in the last equality.

(3). Again by Proposition~\ref{Wpentagon} and from (1) above, we have:
$$
W_{23}W_{12}W_{23}^*=W_{12}W_{13}W_{23}W_{23}^*=W_{12}W_{12}^*W_{12}W_{13}=W_{12}W_{13}.
$$

(4). Let $\omega\in{\mathcal B}({\mathcal H})_*$ be arbitrary.  Then
$$
(\operatorname{id}\otimes\operatorname{id}\otimes\omega)(W_{12}W_{12}^*W_{23})
=WW^*(1\otimes x),
$$
where $x=(\operatorname{id}\otimes\omega)(W)$.  It is an element of $A$ (see the discussion 
right above Proposition~\ref{idomegaWclosureA}).  But, by Proposition~\ref{Wprop}\,(1), we know that 
$W^*(1\otimes x)=(\Delta x)W^*$, while the same proposition implies that $W(\Delta x)=(1\otimes x)W$. 
Combining, we have: $WW^*(1\otimes x)=(1\otimes x)WW^*$.  It follows that
$$
(\operatorname{id}\otimes\operatorname{id}\otimes\omega)(W_{12}W_{12}^*W_{23})
=(\operatorname{id}\otimes\operatorname{id}\otimes\omega)(W_{23}W_{12}W_{12}^*).
$$
Since $\omega$ was arbitrary, this proves the result: $W_{12}W_{12}^*W_{23}=W_{23}W_{12}W_{12}^*$.

(5). For $a,b,c\in{\mathfrak N}_{\varphi}$, since $W^*W=E$, we have: 
\begin{align}
&W_{12}^*W_{23}^*W_{23}\bigl(\Lambda(a)\otimes\Lambda(b)\otimes\Lambda(c)\bigr)
=W_{12}^*\bigl((\Lambda\otimes\Lambda\otimes\Lambda)(a\otimes E(b\otimes c))\bigr)
\notag \\
&=(\Lambda\otimes\Lambda\otimes\Lambda)\bigl((\Delta\otimes\operatorname{id})
(E(b\otimes c))(a\otimes1\otimes1)\bigr)    \notag \\
&=(\Lambda\otimes\Lambda\otimes\Lambda)\bigl((\Delta\otimes\operatorname{id})(E)
(\Delta\otimes\operatorname{id})(b\otimes c)(a\otimes1\otimes1)\bigr)    \notag \\
&=(\Lambda\otimes\Lambda\otimes\Lambda)\bigl((1\otimes E)(E\otimes1)\Delta_{12}(b)
(a\otimes1\otimes c)\bigr)   \notag \\
&=(\Lambda\otimes\Lambda\otimes\Lambda)\bigl((1\otimes E)\Delta_{12}(b)
(a\otimes1\otimes c)\bigr).
\notag
\end{align}
Next to last equality is using the Property\,(3) of $E$ given in Definition~\ref{definitionlcqgroupoid}. 
On the other hand, we have:
\begin{align}
W_{23}^*W_{23}W_{12}^*\bigl(\Lambda(a)\otimes\Lambda(b)\otimes\Lambda(c)\bigr)
&=W_{23}^*W_{23}(\Lambda\otimes\Lambda\otimes\Lambda)\bigl(\Delta_{12}(b)(a\otimes1\otimes c)\bigr) 
\notag \\
&=(\Lambda\otimes\Lambda\otimes\Lambda)\bigl((1\otimes E)\Delta_{12}(b)(a\otimes1\otimes c)\bigr).
\notag
\end{align}
We see that $W_{12}^*W_{23}^*W_{23}=W_{23}^*W_{23}W_{12}^*$. 
Or, $W_{12}W_{23}^*W_{23}=W_{23}^*W_{23}W_{12}$.
\end{proof}

Often, we may wish to consider the von Neumann algebra $\pi_{\varphi}(A)''$.  From the characterization 
of the $C^*$-algebra $A$ given in Proposition~\ref{idomegaWclosureA}, we have: 
\begin{equation}\label{(idomegaWvNclosureA)}
\pi_{\varphi}(A)''=\bigl\{\pi_{\varphi}((\operatorname{id}\otimes\omega)(W)):
\omega\in{\mathcal B}({\mathcal H}_{\varphi})_*\bigr\}''\,\bigl(\subseteq{\mathcal B}
({\mathcal H}_{\varphi})\bigr).
\end{equation}

We can extend $\Delta$ on $A$ to the level of the von Neumann algebra $\pi_{\varphi}(A)''$, 
by taking advantage of the characterization given in Proposition~\ref{DeltaW}.

\begin{prop}\label{vNqgroupoid}
Denote by $M$ the von Neumann algebra $\pi_{\varphi}(A)''$.  For $x\in M$, define: 
$$
\tilde{\Delta}(x)=W^*(1\otimes x)W.
$$
Then $\tilde{\Delta}$ is a ${}^*$-homomorphism from $M$ into $M\otimes M$, extending 
$\Delta$.  And, the coassociativity property holds: 
$(\tilde{\Delta}\otimes\operatorname{id})\tilde{\Delta}=(\operatorname{id}\otimes\tilde{\Delta})\tilde{\Delta}$.
\end{prop}

\begin{proof}
If $a\in\pi_{\varphi}(A)$, by twice using the result of Proposition~\ref{Wprop}\,(1), we see that 
$WW^*(1\otimes a)=(1\otimes a)WW^*$.  Therefore, for any $\theta\in{\mathcal B}({\mathcal H})_*$, 
we have $(\theta\otimes\operatorname{id})(WW^*)a=a(\theta\otimes\operatorname{id})(WW^*)$. 
So $(\theta\otimes\operatorname{id})(WW^*)\in\pi_{\varphi}(A)'$.  It follows that 
for $x\in M=\pi_{\varphi}(A)''$, we have: $(\theta\otimes\operatorname{id})(WW^*)x
=x(\theta\otimes\operatorname{id})(WW^*)$.  Or, equivalently, 
$(\theta\otimes\operatorname{id})\bigl(WW^*(1\otimes x)\bigr)
=(\theta\otimes\operatorname{id})\bigl((1\otimes x)WW^*\bigr)$.
Since $\theta$ is arbitrary, we conclude that $WW^*(1\otimes x)=(1\otimes x)WW^*$, 
true for any $x\in M$.  So, if $x_1,x_2\in M$, we have:
\begin{align}
\tilde{\Delta}(x_1)\tilde{\Delta}(x_2)&=W^*(1\otimes x_1)WW^*(1\otimes x_2)W
\notag \\
&=W^*(1\otimes x_1)(1\otimes x_2)WW^*W
=W^*(1\otimes x_1x_2)W=\tilde{\Delta}(x_1x_2).
\notag
\end{align}
It is clear that $\tilde{\Delta}(x)^*=\tilde{\Delta}(x^*)$.  So we see that $\tilde{\Delta}$ is 
a ${}^*$-homomorphism of $M$ into ${\mathcal B}({\mathcal H}_{\varphi}\otimes{\mathcal B}_{\varphi})$. 
By Proposition~\ref{DeltaW}, it is also clear that $\tilde{\Delta}|_{\pi_{\varphi}(A)}=\Delta$.

Since the elements of the form $(\operatorname{id}\otimes\omega)(W)\in\pi_{\varphi}(A)$ are 
dense in $M$, while $\tilde{\Delta}\bigl(\pi_{\varphi}(A)\bigr)=\Delta\bigl(\pi_{\varphi}(A)\bigr)
\subseteq(\pi_{\varphi}\otimes\pi_{\varphi})\bigl(M(A\otimes A)\bigr)$, which is contained in $M\otimes M$, 
we can say that in fact, $\tilde{\Delta}:M\to M\otimes M$.

For the coassociativity, suppose $x\in M$.  Then, by the pentagon relation (Proposition~\ref{Wpentagon}), 
we have: 
\begin{align}
W_{12}^*W_{23}^*(1\otimes1\otimes x)W_{23}W_{12}
&=W_{23}^*W_{13}^*W_{12}^*(1\otimes1\otimes x)W_{12}W_{13}W_{23}  \notag \\
&=W_{23}^*W_{13}^*(1\otimes1\otimes x)W_{12}^*W_{12}W_{13}W_{23}  \notag \\
&=W_{23}^*W_{13}^*(1\otimes1\otimes x)W_{13}W_{23}W_{23}^*W_{23}  \notag \\
&=W_{23}^*W_{13}^*(1\otimes1\otimes x)W_{13}W_{23},
\notag
\end{align}
where we used Proposition~\ref{Wpentagon_alt}\,(1), for the third equality.  It follows that
$$
(\tilde{\Delta}\otimes\operatorname{id})(\tilde{\Delta}x)
=(\operatorname{id}\otimes\tilde{\Delta})(\tilde{\Delta}x).
$$
\end{proof}

There exist results similar to all of the above propositions in this subsection, in terms of 
the operator $V$. See Proposition~\ref{omegaidV} later.

\section{Antipode}\label{sec4}

The aim of this section is to construct the antipode map $S$ for a given locally compact 
quantum groupoid of separable type $(A,\Delta,E,B,\nu,\varphi,\psi)$.  Since $S$ is in general 
unbounded even in the quantum group case, this needs some care.  The standard approach 
nowadays is to use the ``polar decomposition'':  The idea originally goes back to an unpublished 
work by Eberhard Kirchberg, which was adopted by the authors of the quantum group literature. 
See \cite{KuVa}, \cite{KuVavN}, \cite{MNW}, \cite{VDvN}.  Similar strategy can be employed 
in our case as well (see \S\ref{subsec4.4} below), with some necessary modifications to incorporate 
the fact that the operators $V$ and $W$ are not unitaries.

But first, we will construct an (unbounded) involutive operator $K$.  This is the operator 
implementing the antipode at the Hilbert space level, and the underlying idea here is essentially 
same as in the  quantum group case (in particular, see section~2 of \cite{VDvN}).

\subsection{The involutive operator $K$}\label{subsec4.1}

Keep the notation as in the previous sections, while we take ${\mathcal H}={\mathcal H}_{\psi}$. 
So we will have $A=\pi_{\psi}(A)\subseteq{\mathcal B}({\mathcal H}_{\psi})$ via the GNS 
representation.  The operator $V$, defined in Proposition~\ref{Vdefn}, will be contained 
in ${\mathcal B}({\mathcal H}_{\psi}\otimes{\mathcal H}_{\psi})$.

\begin{defn}\label{D(K)}
Let $\xi\in{\mathcal H}_{\psi}$.  We will say $\xi\in{\mathcal D}(K)$, if there exists a vector 
$\tilde{\xi}\in{\mathcal H}_{\psi}$ such that for any $\varepsilon>0$ and any vectors 
$\zeta_1,\zeta_2,\dots,\zeta_n\in{\mathcal H}_{\psi}$, we can find elements 
$p_1,p_2,\dots,p_m;q_1,q_2,\dots,q_m\in{\mathfrak N}_{\psi}$ satisfying
\begin{align}
&\left\|V\left(\sum_{j=1}^m\Lambda_{\psi}(p_j)\otimes q_j^*\zeta_k\right)
-E(\xi\otimes\zeta_k)\right\|\,<\,\varepsilon \qquad {\text {and}}    \notag \\
&\left\|V\left(\sum_{j=1}^m\Lambda_{\psi}(q_j)\otimes p_j^*\zeta_k\right)
-E(\tilde{\xi}\otimes\zeta_k)\right\|\,<\,\varepsilon,
\notag
\end{align}
for all $k=1,2,\dots,n$.  Here, we wrote $E=(\pi_{\psi}\otimes\pi_{\psi})(E)
\in{\mathcal B}({\mathcal H}_{\psi}\otimes{\mathcal H}_{\psi})$.
\end{defn}

Our aim is to define a linear operator $K$ by the map ${\mathcal D}(K)\ni\xi\mapsto\tilde{\xi}$.  
However, at present we do not know whether this is a well-defined map. See the next proposition.

\begin{prop}\label{Kwelldef}
If $\xi=0$, then $\xi\in{\mathcal D}(K)$ and $\tilde{\xi}=0$.
\end{prop}

\begin{proof}
It is clear that $0\in{\mathcal D}(K)$, because we can just let $\tilde{\xi}$ and the $p_j$ 
and the $q_j$  all zero.  But the main issue here is to show whether $\tilde{\xi}=0$ 
necessarily.

So suppose $\xi=0$, and let $\tilde{\xi}$ be as in Definition~\ref{D(K)}. Then for arbitrary 
$\zeta_1,\zeta_2\in{\mathcal H}_{\psi}$ and $\varepsilon>0$, we can find elements 
$p_1,p_2,\dots,p_m;q_1,q_2,\dots,q_m\in{\mathfrak N}_{\psi}$, such that 

\begin{equation}\label{(K1)}
\left\|V\left(\sum_{j=1}^m\Lambda_{\psi}(p_j)\otimes q_j^*\zeta_1\right)\right\|\,<\varepsilon,
\end{equation}
and
\begin{equation}\label{(K2)}
\left\|V\left(\sum_{j=1}^m\Lambda_{\psi}(q_j)\otimes p_j^*\zeta_2\right)
-E(\tilde{\xi}\otimes\zeta)\right\|\,<\varepsilon.
\end{equation}
Here, we may assume without loss of generality that the elements $p_j$ are of the form 
$p_j=(\operatorname{id}\otimes\varphi)\bigl(\Delta(r_j^*y_j)(1\otimes s_j)\bigr)$, for 
$r_j,s_j\in{\mathfrak N}_{\varphi}$, $y_j\in{\mathfrak N}_{\psi}$, so that $p_j\otimes b
\in{\mathcal D}(Q_R)$, for $b\in A$. 

Consider two ``right bounded'' vectors $\rho_1$, $\rho_2$ in ${\mathcal H}={\mathcal H}_{\psi}$. 
Recall that a vector $\rho$ is {\em right bounded\/} in ${\mathcal H}$, if there is a (unique) 
bounded operator, denoted as $\pi_R(\rho)$, satisfying $x\rho=\pi_R(\rho)\Lambda_{\psi}(x)$ 
for all $x\in{\mathfrak N}_{\psi}$.  See Definition~1.7 of \cite{Tk2}.  Such elements are dense 
in ${\mathcal H}_{\psi}$.  Meanwhile, recall that $V$ is a partial isometry, with the terminal projection 
$V^*V=G_R$ (see Theorem~\ref{K_psi=H_psi} and Proposition~\ref{G_R}).  By the definition of the 
operator $G_R$ as given in Equation~\eqref{(G_R)}, we have:
\begin{align}\label{(K3)}
&\left\langle V\left(\sum_j\Lambda_{\psi}(p_j)\otimes q_j^*\zeta_1\right),
V\bigl(\pi_R(\rho_1)^*\zeta_2\otimes\rho_2\bigr)\right\rangle
\notag \\
&=\sum_j\bigl\langle(\pi_R(\rho_1)\otimes1)G_R\bigl(\Lambda_{\psi}(p_j)\otimes q_j^*\zeta_1\bigr),
\zeta_2\otimes\rho_2\bigr\rangle
\notag \\
&=\sum_j\bigl\langle Q_R(p_j\otimes q_j^*)(\rho_1\otimes\zeta_1),
\zeta_2\otimes\rho_2\bigr\rangle
\notag  \\
&=\sum_j\bigl\langle\rho_1\otimes\zeta_1,
Q_R(p_j\otimes q_j^*)^*(\zeta_2\otimes\rho_2)\bigr\rangle
=\sum_j\bigl\langle\rho_1\otimes\zeta_1,
Q_{\rho}(p_j^*\otimes q_j)(\zeta_2\otimes\rho_2)\bigr\rangle
\notag \\
&=\sum_j\bigl\langle\rho_1\otimes\pi_R(\rho_2)^*\zeta_1,
G_{\rho}(p_j^*\zeta_2\otimes\Lambda_{\psi}(q_j))\bigr\rangle.
\end{align}
We used the result of Proposition~\ref{RR'LL'}, and in the last line, 
we are using the projection operator $G_{\rho}$ defined in Corollary of Proposition~\ref{G_R}.

Meanwhile, observe that we have 
\begin{align}\label{(K4)}
&\bigl|\langle E(\tilde{\xi}\otimes\zeta_2),V(\pi_R(\rho_2)^*\zeta_1\otimes\rho_1)\rangle\bigr|  \notag \\
&\le\left|\left\langle E(\tilde{\xi}\otimes\zeta_2)-V\left(\sum_j\Lambda_{\psi}(q_j)\otimes p_j^*\zeta_2\right),
V\bigl(\pi_R(\rho_2)^*\zeta_1\otimes\rho_1\bigr)\right\rangle\right|  \notag \\
&\qquad+\left|\left\langle V\left(\sum_j\Lambda_{\psi}(q_j)\otimes p_j^*\zeta_2\right),
V\bigl(\pi_R(\rho_2)^*\zeta_1\otimes\rho_1\bigr)\right\rangle\right|
\notag \\
&\le\varepsilon\bigl\|\pi_R(\rho_2)^*\zeta_1\otimes\rho_1\bigr\|
+\left\|\sum_j\bigl\langle G_R(\Lambda_{\psi}(q_j)\otimes p_j^*\zeta_2),
\pi_R(\rho_2)^*\zeta_1\otimes\rho_1\bigr\rangle\right\|,
\end{align}
by Equation~\eqref{(K2)}, the fact that $\|V\|\le1$,  and by noting that $V^*V=G_R$.

Because of the way the projection operators $G_R$ and $G_{\rho}$ are defined and by 
Proposition~\ref{RR'LL'}, it is evident that $\bigl\|G_R(\eta_1\otimes\eta_2)\bigr\|
=\bigl\|G_{\rho}(\eta_1\otimes\eta_2)\bigr\|$, and also $\bigl\|G_R(\eta_1\otimes\eta_2)\bigr\|
=\bigl\|G_{\rho}(\eta_2\otimes\eta_1)\bigr\|$, for any $\eta_1,\eta_2\in{\mathcal H}_{\psi}$.
Therefore, 
\begin{align}
&\bigl\|\sum_j\bigl\langle G_R(\Lambda_{\psi}(q_j)\otimes p_j^*\zeta_2),
\pi_R(\rho_2)^*\zeta_1\otimes\rho_1\bigr\rangle\bigr\|    \notag \\
&=\bigl\|\sum_j\langle\rho_1\otimes\pi_R(\rho_2)^*\zeta_1,
G_{\rho}(p_j^*\zeta_2\otimes\Lambda_{\psi}(q_j))\rangle\bigr\|. 
\notag
\end{align}
It follows that by Equation~\eqref{(K3)}, we have:
\begin{align}\label{(K5)}
&\bigl\|\sum_j\bigl\langle G_R(\Lambda_{\psi}(q_j)\otimes p_j^*\zeta_2),
\pi_R(\rho_2)^*\zeta_1\otimes\rho_1\bigr\rangle\bigr\|    \notag \\
&=\left\|\left\langle V\left(\sum_j\Lambda_{\psi}(p_j)\otimes q_j^*\zeta_1\right),
V\bigl(\pi_R(\rho_1)^*\zeta_2\otimes\rho_2\bigr)\right\rangle\right\|
\le\varepsilon\bigl\|\pi_R(\rho_1)^*\zeta_2\otimes\rho_2\bigr\|.
\end{align}
The last inequality is due to Equation~\eqref{(K1)} and the fact that $\|V\|\le1$.

Combining Equations~\eqref{(K4)} and \eqref{(K5)}, we now conclude that 
$$
\bigl|\langle E(\tilde{\xi}\otimes\zeta_2),V(\pi_R(\rho_2)^*\zeta_1\otimes\rho_1)\rangle\bigr|
\le\varepsilon\bigl\|\pi_R(\rho_2)^*\zeta_1\bigr\|\bigl\|\rho_1\bigr\|
+\varepsilon\bigl\|\pi_R(\rho_1)^*\zeta_2\bigr\|\bigl\|\rho_2\bigr\|.
$$
As this is true for any $\varepsilon$, it follows that
$$
\bigl\langle E(\tilde{\xi}\otimes\zeta_2),V(\pi_R(\rho_2)^*\zeta_1\otimes\rho_1)\bigr\rangle=0, 
$$
true for any right bounded vectors $\rho_1,\rho_2$, and any $\zeta_1\in{\mathcal H}_{\psi}$. 
Since the $\pi_R(\rho_2)^*\zeta_1\otimes\rho_1$ span a dense subspace of ${\mathcal H}_{\psi}
\otimes{\mathcal H}_{\psi}$, we see that $E(\tilde{\xi}\otimes\zeta_2)\in\operatorname{Ran}
(V)^{\perp}$.  Since $\operatorname{Ran}(V)=\operatorname{Ran}(E)$, this means 
that $E(\tilde{\xi}\otimes\zeta_2)=0_{{\mathcal H}_{\psi}\otimes{\mathcal H}_{\psi}}$, where 
$\zeta_2$ is arbitrary in ${\mathcal H}_{\psi}$.

We claim that this implies $\tilde{\xi}=0_{{\mathcal H}_{\psi}}$.  To see this, take, without loss 
of generality $\zeta_2=\Lambda_{\psi}(a)$, for an arbitrary $a\in{\mathfrak N}_{\psi}$.  Note that 
for any $c\in C$, we have: $E(\tilde{\xi}\otimes c\zeta_2)=E\bigl(\tilde{\xi}\otimes\Lambda_{\psi}(ca)\bigr)
=0_{{\mathcal H}_{\psi}\otimes{\mathcal H}_{\psi}}$.  So for any $d\in{\mathfrak N}_{\psi}$, 
we have: 
$$
(\operatorname{id}\otimes\omega_{\Lambda_{\psi}(a),\Lambda_{\psi}(d)})\bigl(E(1\otimes c)\bigr)\tilde{\xi}
=\bigl\langle E(\tilde{\xi}\otimes c\Lambda_{\psi}(a)),\,\cdot\,\otimes\Lambda_{\psi}(d)\bigr\rangle
=0_{{\mathcal H}_{\psi}}.  
$$
Recall that the elements of the form $(\operatorname{id}\otimes\omega)\bigl(E(1\otimes c)\bigr)$ span 
a dense subspace in $B$ (``$E$ is full'').  It follows that $b\tilde{\xi}=0$ for any $b\in B$. 
Since $B$ is a non-degenerate subalgebra of $M(A)$, this is enough to show that $\tilde{\xi}
=0_{{\mathcal H}_{\psi}}$.  This finishes the proof.
\end{proof}

We can now give the definition of our operator $K$.  Proposition~\ref{Kwelldef} assures 
us that it is well-defined.

\begin{defn}\label{Kdefn}
For $\xi\in{\mathcal D}(K)$ and $\tilde{\xi}$ as in Definition~\ref{D(K)}, write $K\xi:=\tilde{\xi}$. 
\end{defn}

\begin{rem}
The definitions of ${\mathcal D}(K)$ and $K$ are analogous to the definitions given in section~2 of 
\cite{VDvN}, in the quantum group case.  However, with $E\ne1\otimes1$ and the operator $V$ 
not unitary (only a partial isometry), the proof became more complicated.
\end{rem}

\begin{prop}\label{Koperator}
$K$, as defined above, is a closed operator on ${\mathcal H}_{\psi}$, in general unbounded. 
Moreover, if $\xi\in{\mathcal D}(K)$, then $K\xi\in{\mathcal D}(K)$ and $K(K\xi)=\xi$.
\end{prop}

\begin{proof}
(1). Suppose $(\xi_l)\to\xi$, with each $\xi_l\in{\mathcal D}(K)$, and $(K\xi_l)\to\zeta$.  We need to 
show that $\xi\in{\mathcal D}(K)$ and that $\zeta=K\xi=\tilde{\xi}$.  This can be done by a standard 
argument using triangle inequality.  So $K$ is a closed operator.

(2). Second claim is also immediate from the symmetry of Definition~\ref{D(K)}, which implies 
that $\tilde{\tilde{\xi}}=\xi$.
\end{proof}

We now have a valid operator $K$, but at this stage, we do not know whether it is densely-defined. 
This is the problem we will turn to next.  For this part, it is useful to work with the operator 
$W\in{\mathcal B}({\mathcal H}\otimes{\mathcal H}_{\varphi})$, where we still let ${\mathcal H}
={\mathcal H}_{\psi}$ as we did at the beginning of this subsection.

\begin{lem}\label{lemmaidomegacW}
Let $c\in{\mathfrak N}_{\psi}$ and let $\omega\in{\mathcal B}({\mathcal H}_{\varphi})_*$. 
Then we have: 
$$\bigl(\operatorname{id}\otimes\omega(c\,\cdot\,)\bigr)(W)\in{\mathfrak N}_{\psi}.$$
\end{lem}

\begin{proof}
Write $x=\bigl(\operatorname{id}\otimes\omega(c\,\cdot\,)\bigr)(W)
=(\operatorname{id}\otimes\omega)\bigl((1\otimes c)W\bigr)$.  By Propositions~\ref{idomegaW} 
and \ref{idomegaWclosureA}, we already know that $x\in A$.   We then have: 
\begin{align}
x^*x&=(\operatorname{id}\otimes\omega)\bigl((1\otimes c)W\bigr)^*\,(\operatorname{id}
\otimes\omega)\bigl((1\otimes c)W\bigr)  \notag \\
&\le\|\omega\|\,\bigl(\operatorname{id}\otimes|\omega|\bigr)\bigl(W^*(1\otimes c^*)
(1\otimes c)W\bigr),
\notag
\end{align}
where we used the result of Lemma~4.4 in Part~I \cite{BJKVD_qgroupoid1}.  Meanwhile, 
note that $W^*(1\otimes c^*)(1\otimes c)W=W^*(1\otimes c^*c)W=\Delta(c^*c)$, 
by Proposition~\ref{DeltaW}.  Also, since $c\in{\mathfrak N}_{\psi}$, we have 
$c^*c\in{\mathfrak M}_{\psi}$ and $\Delta(c^*c)\in{\mathfrak M}_{\psi\otimes\operatorname{id}}$, 
by the right invariance of $\psi$.  It follows that 
$$
\psi(x^*x)\le\psi\bigl(\|\omega\|\,(\operatorname{id}\otimes|\omega|)
(\Delta(c^*c))\bigr)
=\|\omega\|\,|\omega|\bigl((\psi\otimes\operatorname{id})(\Delta(c^*c))\bigr)<\infty.
$$
Therefore, $x\in{\mathfrak N}_{\psi}$.
\end{proof}

In the below is a result for the case when the comultiplication is applied to the 
element of the type given in the previous lemma:

\begin{prop}\label{Deltaidomega}
Let $x=\bigl(\operatorname{id}\otimes\omega(c\,\cdot\,)\bigr)(W)\in{\mathfrak N}_{\psi}$ as above, 
where $c\in{\mathfrak N}_{\psi}$, $\omega\in{\mathcal B}({\mathcal H}_{\varphi})_*$. 
Then we have:
$$
\Delta x=\bigl(\operatorname{id}\otimes\operatorname{id}\otimes\omega(c\,\cdot\,)\bigr)
(W_{13}W_{23}),
$$
which is contained in ${\mathfrak N}_{\psi\otimes\operatorname{id}}$.
\end{prop}

\begin{proof}
As we are taking ${\mathcal H}={\mathcal H}_{\psi}$, we have $W\in{\mathcal B}
({\mathcal H}_{\psi}\otimes{\mathcal H}_{\varphi})$.  For this discussion, introduce also 
${}_{\varphi}\!W\in{\mathcal B}({\mathcal H}_{\varphi}\otimes{\mathcal H}_{\varphi})$, 
which is defined in exactly the same way as the operator $W$ but by taking ${\mathcal H}
={\mathcal H}_{\varphi}$ in Proposition~\ref{Wdefn}.   It is clear that 
$\pi_{\varphi}(x)=\bigl(\operatorname{id}\otimes\omega(c\,\cdot\,)\bigr)({}_{\varphi}\!W)$.

By Proposition~\ref{DeltaW}, we then have:
$$
(\pi\otimes\pi_{\varphi})(\Delta x)=W^*\bigl(1\otimes\pi_{\varphi}(x)\bigr)W
=\bigl(\operatorname{id}\otimes\operatorname{id}\otimes\omega(c\,\cdot\,)\bigr)
(W_{12}^*\,{}_{\varphi}\!W_{23}W_{12}).
$$
From Proposition~\ref{Wpentagon_alt}\,(2), it is evident that 
$W_{12}^*\,{}_{\varphi}\!W_{23}W_{12}=W_{13}\,{}_{\varphi}\!W_{23}$, as operators 
on ${\mathcal H}_{\psi}\otimes{\mathcal H}_{\varphi}\otimes{\mathcal H}_{\varphi}$. 
So $(\pi\otimes\pi_{\varphi})(\Delta x)=\bigl(\operatorname{id}\otimes\operatorname{id}
\otimes\omega(c\,\cdot\,)\bigr)(W_{13}\,{}_{\varphi}\!W_{23})$.  With our convention 
that $A=\pi(A)\,\subseteq{\mathcal B}({\mathcal H})$, we can see that 
$$
\Delta x=(\pi\otimes\pi)(\Delta x)=\bigl(\operatorname{id}\otimes\operatorname{id}
\otimes\omega(c\,\cdot\,)\bigr)(W_{13}W_{23})\,\in M(A\otimes A).
$$
By the right invariance of $\psi$, it is in fact contained in ${\mathfrak N}_{\psi\otimes\operatorname{id}}$.
\end{proof}

\begin{rem}
Eventually, we will see that ${\mathcal H}_{\psi}$ and ${\mathcal H}_{\varphi}$ can be identified. 
Then, the result like the previous proposition would become unnecessary, as it is just a corollary 
of Proposition~\ref{DeltaW}.  Until then, similar technique as in the above proof can be used 
when we work with elements of the form $(\operatorname{id}\otimes\omega)(W)$, to be viewed 
as in $\pi(A)$ or in $\pi_{\varphi}(A)$.  
\end{rem}

The next proposition shows that there are sufficiently many elements contained in ${\mathcal D}(K)$.

\begin{prop}\label{Kdensedomain}
Let $c,d\in{\mathfrak N}_{\psi}$ and let $\omega\in{\mathcal B}({\mathcal H}_{\varphi})_*$. 
Consider:
$$
\xi=\Lambda_{\psi}\bigl((\operatorname{id}\otimes\omega(c\,\cdot\,d^*))(W)\bigr).
$$
Then we have $\xi\in{\mathcal D}(K)$ and $K\xi=\tilde{\xi}=\Lambda_{\psi}\bigl((\operatorname{id}
\otimes\bar{\omega}(d\,\cdot\,c^*))(W)\bigr)$.
\end{prop}

\begin{proof}
Without loss of generality, we may assume $\omega=\omega_{\xi',\zeta'}$, for $\xi',\zeta'
\in{\mathcal H}_{\varphi}$.  Let $(e_j)_{j\in J}$ be an orthonormal basis for ${\mathcal H}_{\varphi}$, 
and define:
\begin{align}
p_j&:=(\operatorname{id}\otimes\omega_{e_j,c^*\zeta'})(W)=(\operatorname{id}\otimes
\omega_{e_j,\zeta'}(c\,\cdot\,))(W)\in{\mathfrak N}_{\psi},   \notag \\
q_j&:=(\operatorname{id}\otimes\omega_{e_j,d^*\xi'})(W)=(\operatorname{id}\otimes
\omega_{e_j,\xi'}(d\,\cdot\,))(W)\in{\mathfrak N}_{\psi}.
\notag
\end{align}

Suppose $\zeta=\Lambda_{\psi}(a)\in{\mathcal H}$, for an arbitrary $a\in{\mathfrak N}_{\psi}$.  
In view of Definition~\ref{D(K)}, the proof will be done by verifying the following:
\begin{align}
&V\left(\sum_{j\in J}\Lambda_{\psi}(p_j)\otimes q_j^*\zeta\right)\longrightarrow
E(\xi\otimes\zeta), \label{(conv1)}\\
&V\left(\sum_{j\in J}\Lambda_{\psi}(q_j)\otimes p_j^*\zeta\right)\longrightarrow
E(\tilde{\xi}\otimes\zeta).  \label{(conv2)}
\end{align}
This notation is being understood that the net of sums over the finite sets $\Gamma\subseteq J$ 
converges, in norm.

To verify Equation~\eqref{(conv1)}, note first that by Proposition~\ref{Deltaidomega}, we have:
$\Delta(p_j)=(\operatorname{id}\otimes\operatorname{id}\otimes
\omega_{e_j,c^*\zeta'})(W_{13}W_{23})$.
By Lemma~\ref{omega_xizetaLem}\,(2), we also know that $q_j^*=(\operatorname{id}
\otimes\omega_{d^*\xi',e_j})(W^*)$.  Then 
\begin{align}
&V\left(\sum_{j\in J}\Lambda_{\psi}(p_j)\otimes q_j^*\zeta\right)
=\sum_{j\in J}(\Lambda_{\psi}\otimes\Lambda_{\psi})\bigl(\Delta(p_j)(1\otimes q_j^*)(1\otimes a)\bigr)
\notag \\
&=\sum_{j\in J}(\Lambda_{\psi}\otimes\Lambda_{\psi})\bigl((\operatorname{id}\otimes\operatorname{id}
\otimes\omega_{e_j,c^*\zeta'})(W_{13}W_{23})(\operatorname{id}\otimes\operatorname{id}\otimes
\omega_{d^*\xi',e_j})(W_{23}^*)(1\otimes a)\bigr).
\notag
\end{align}
By Lemma~\ref{omega_xizetaLem}\,(4), this sum converges to
$$
(\Lambda_{\psi}\otimes\Lambda_{\psi})\bigl((\operatorname{id}\otimes\operatorname{id}\otimes
\omega_{d^*\xi',c^*\zeta'})(W_{13}W_{23}W_{23}^*)\,(1\otimes a)\bigr).
$$
Now use Proposition~\ref{Wpentagon_alt}\,(1), namely $W_{13}W_{23}W_{23}^*
=W_{12}^*W_{12}W_{13}=E_{12}W_{13}$.  Here again, there is the subtle issue concerning 
the Hilbert spaces ${\mathcal H}$ and ${\mathcal H}_{\varphi}$, which can be easily managed 
by working with ${}_{\varphi}\!W$ as in the proof of Proposition~\ref{Deltaidomega}.  It follows 
that 
\begin{align}
V\left(\sum_{j\in J}\Lambda_{\psi}(p_j)\otimes q_j^*\zeta\right)
\longrightarrow\,&(\Lambda_{\psi}\otimes\Lambda_{\psi})\bigl(E(\operatorname{id}\otimes
\operatorname{id}\otimes\omega_{d^*\xi',c^*\zeta'})(W_{13})(1\otimes a)\bigr)   \notag \\
&=E\bigl(\xi\otimes\Lambda_{\psi}(a)\bigr)=E(\xi\otimes\zeta).
\notag
\end{align}
Note here that $\Lambda_{\psi}
\bigl((\operatorname{id}\otimes\omega_{d^*\xi',c^*\zeta'})(W)\bigr)
=\Lambda_{\psi}\bigl((\operatorname{id}\otimes\omega_{\xi',\zeta'}(c\,\cdot\,d^*))(W)\bigr)=\xi$, 
because $\omega=\omega_{\xi',\zeta'}$.  In this way, we verified Equation~\eqref{(conv1)}.

By using the same notation and same method, we can show that 
$$
V\left(\sum_{j\in J}\Lambda_{\psi}(q_j)\otimes p_j^*\zeta\right)
=\sum_{j\in J}(\Lambda_{\psi}\otimes\Lambda_{\psi})\bigl(\Delta(q_j)(1\otimes p_j^*)(1\otimes a)\bigr)
$$
converges to 
\begin{align}
&(\Lambda_{\psi}\otimes\Lambda_{\psi})\bigl((\operatorname{id}\otimes\operatorname{id}\otimes
\omega_{c^*\zeta',d^*\xi'})(W_{13}W_{23}W_{23}^*)\,(1\otimes a)\bigr)  \notag \\
&=(\Lambda_{\psi}\otimes\Lambda_{\psi})\bigl(E(\operatorname{id}\otimes\operatorname{id}
\otimes\omega_{c^*\zeta',d^*\xi'})(W_{13})(1\otimes a)\bigr)
=E\bigl(\tilde{\xi}\otimes\Lambda_{\psi}(a)\bigr),
\notag
\end{align}
where $\tilde{\xi}=\Lambda_{\psi}\bigl((\operatorname{id}\otimes\omega_{c^*\zeta',d^*\xi'})(W)\bigr)
=\Lambda_{\psi}\bigl((\operatorname{id}\otimes\bar{\omega}(d\,\cdot\,c^*))(W)\bigr)$.  Note here 
that $\bar{\omega}=\overline{\omega_{\xi',\zeta'}}=\omega_{\zeta',\xi'}$.  This verifies 
Equation~\eqref{(conv2)}.
\end{proof}

\begin{cor}
${\mathcal D}(K)$ is dense in ${\mathcal H}_{\psi}$.  Moreover, the elements of the form $\Lambda_{\psi}
\bigl((\operatorname{id}\otimes\omega(c\,\cdot\,d^*))(W)\bigr)$, $\omega\in{\mathcal B}
({\mathcal H}_{\varphi})_*$, $c,d\in{\mathfrak N}_{\psi}$, form a core for $K$.
\end{cor}

\begin{proof}
By Proposition~\ref{Kdensedomain},  the elements of the form 
$\Lambda_{\psi}\bigl((\operatorname{id}\otimes\omega(c\,\cdot\,d^*))(W)\bigr)$, 
$c,d\in{\mathfrak N}_{\psi}$, are contained in ${\mathcal D}(K)$.  Since ${\mathfrak N}_{\psi}$ is 
dense in $A$, these elements are dense in the space generated by $\Lambda_{\psi}
\bigl((\operatorname{id}\otimes\omega)(W)\bigr)$, $\omega\in{\mathcal B}({\mathcal H}_{\varphi})_*$.

Meanwhile, by Proposition~\ref{idomegaW}, for $\xi'=\Lambda_{\varphi}(s)$, $\zeta'=\Lambda_{\varphi}(r)$,  
$s,r\in{\mathfrak N}_{\varphi}$, we have 
$\Lambda_{\psi}\bigl((\operatorname{id}\otimes\omega_{\xi',\zeta'})(W)\bigr)=\Lambda_{\psi}
\bigl((\operatorname{id}\otimes\varphi)(\Delta(r^*)(1\otimes s))\bigr)$.  By Equation~\eqref{(K_psi)} 
and by Theorem~\ref{K_psi=H_psi}, we know such elements are dense in ${\mathcal H}_{\psi}$.  
Combining the results, we can see that ${\mathcal D}(K)$ is dense in the Hilbert space ${\mathcal H}_{\psi}$.
\end{proof}

So far, we have established that $K$ is a well-defined, closed, densely-defined operator on 
${\mathcal H}_{\psi}$.  We wrap up this section with the following result, involving the operator $K$ 
and one leg of the operator $W$.

\begin{prop}\label{idomegaWK}
For any $\xi\in{\mathcal D}(K)$ and any $\omega\in{\mathcal B}({\mathcal H}_{\varphi})_*$, we have: 
$(\operatorname{id}\otimes\omega)(W)\xi\in{\mathcal D}(K)$, and 
$$
K\bigl((\operatorname{id}\otimes\omega)(W)\xi\bigr)=(\operatorname{id}\otimes\bar{\omega})(W)(K\xi).
$$
\end{prop}

\begin{proof}
Without loss of generality, we may consider $(\operatorname{id}\otimes\omega_{\xi'',\zeta''})(W)\xi$, 
with $\xi'',\zeta''\in{\mathcal H}_{\varphi}$, while $\xi=\Lambda_{\psi}\bigl((\operatorname{id}\otimes
\omega_{\xi',\zeta'}(c\,\cdot\,d^*))(W)\bigr)$, with $\xi',\zeta'\in{\mathcal H}_{\varphi}$, 
$c,d\in{\mathfrak N}_{\psi}$.  We may imitate the proof of Proposition~\ref{Kdensedomain}.
  
Let $(e_j)_{j\in J}$ be an orthonormal basis for ${\mathcal H}_{\varphi}$.  For $j,l\in J$, define:
\begin{align}
p_{jl}&:=(\operatorname{id}\otimes\omega_{e_j,\zeta''}\otimes\omega_{e_l,c^*\zeta'})(W_{12}W_{13}), 
\notag \\
q_{jl}&:=(\operatorname{id}\otimes\omega_{e_j,\xi''}\otimes\omega_{e_l,d^*\xi'})(W_{12}W_{13}).
\notag
\end{align}
Note that we can write $p_{jl}=(\operatorname{id}\otimes\omega_{e_j,\zeta''})(W)\bigl(\operatorname{id}
\otimes\omega_{e_l,\zeta'}(c\,\cdot\,)\bigr)(W)$.  Since $\bigl(\operatorname{id}\otimes\omega_{e_l,\zeta'}
(c\,\cdot\,)\bigr)(W)\in{\mathfrak N}_{\psi}$ (by Lemma~\ref{lemmaidomegacW}) and since 
${\mathfrak N}_{\psi}$ is a left ideal in $M(A)$, we can see that $p_{jl}\in{\mathfrak N}_{\psi}$. 
Similarly, $q_{jl}\in{\mathfrak N}_{\psi}$. 

Note that we have:
$$
\Delta(p_{jl})=(\operatorname{id}\otimes\operatorname{id}\otimes\omega_{e_j,\zeta''}\otimes
\omega_{e_l,c^*\zeta'})(W_{12}^*W_{23}W_{24}W_{12}).
$$
Therefore, if we let $\zeta=\Lambda_{\psi}(a)$, for an arbitrary $a\in{\mathfrak N}_{\psi}$, 
we have:
\begin{align}
&V\left(\sum_{j,l\in J}\Lambda_{\psi}(p_{jl})\otimes q_{jl}^*\zeta\right)
=\sum_{j,l\in J}(\Lambda_{\psi}\otimes\Lambda_{\psi})\bigl(\Delta(p_{jl})(1\otimes q_{jl}^*)(1\otimes a)\bigr)
\notag \\
&=\sum_{j,l\in J}(\Lambda_{\psi}\otimes\Lambda_{\psi})\bigl((\operatorname{id}\otimes\operatorname{id}\otimes
\omega_{e_j,\zeta''}\otimes\omega_{e_l,c^*\zeta'})(W_{12}^*W_{23}W_{24}W_{12}) \notag \\
&\qquad\qquad\times(\operatorname{id}\otimes\operatorname{id}\otimes\omega_{\xi'',e_j}\otimes\omega_{d^*\xi',e_l})
(W_{24}^*W_{23}^*)\,(1\otimes a)\bigr).
\notag
\end{align}
For the characterization of $q_{jl}^*$, we used Lemma~\ref{omega_xizetaLem}\,(2).

By twice applying Lemma~\ref{omega_xizetaLem}\,(4), this sum converges to
\begin{equation}\label{(conv11a)}
(\Lambda_{\psi}\otimes\Lambda_{\psi})\bigl((\operatorname{id}\otimes\operatorname{id}\otimes
\omega_{\xi'',\zeta''}\otimes\omega_{d^*\xi',c^*\zeta'})(W_{12}^*W_{23}W_{24}W_{12}W_{24}^*W_{23}^*)
(1\otimes a)\bigr).
\end{equation}
But note that 
\begin{align}
W_{12}^*W_{23}W_{24}W_{12}W_{24}^*W_{23}^*&=W_{12}^*W_{23}W_{12}W_{14}W_{23}^*
=W_{12}^*W_{23}W_{12}W_{23}^*W_{14}
 \notag \\
&=W_{12}^*W_{12}W_{13}W_{14} =E_{12}W_{13}W_{14},
\notag
\end{align}
by twice using Proposition~\ref{Wpentagon_alt}\,(3).  Therefore, the limit given in Equation~\eqref{(conv11a)} 
becomes
$$
E\,\bigl[\Lambda_{\psi}\bigl((\operatorname{id}\otimes\omega_{\xi'',\zeta''})
(W)(\operatorname{id}\otimes\omega_{d^*\xi',c^*\zeta'})(W)\bigr)\otimes
\Lambda_{\psi}(a)\bigr].
$$
In other words, we have shown that 
\begin{equation}\label{(conv11)}
V\left(\sum_{j\in J, l\in J}\Lambda_{\psi}(p_{jl})\otimes q_{jl}^*\zeta\right)
\longrightarrow
E\,\bigl[(\operatorname{id}\otimes\omega_{\xi'',\zeta''})(W)\xi\otimes\zeta].
\end{equation}

In a similar way, we can also show that 
\begin{equation}\label{(conv22)}
V\left(\sum_{j\in J, l\in J}\Lambda_{\psi}(q_{jl})\otimes p_{jl}^*\zeta\right)
\longrightarrow
E\,\bigl[(\operatorname{id}\otimes\omega_{\zeta'',\xi''})(W)\tilde{\xi}\otimes\zeta],
\end{equation}
where $\tilde{\xi}=\Lambda_{\psi}\bigl((\operatorname{id}\otimes\omega_{c^*\zeta',d^*\xi'})
(W)\bigr)=K\xi$, as observed in Proposition~\ref{Kdensedomain}.

Since $\xi'',\zeta'',\xi',\zeta',c,d,a$ are arbitrary, Equations~\eqref{(conv11)} 
and \eqref{(conv22)} show that $(\operatorname{id}\otimes\omega)(W)\xi\in{\mathcal D}(K)$, 
$\forall\xi\in{\mathcal D}(K),\forall\omega\in{\mathcal B}({\mathcal H})_*$, by Definition~\ref{D(K)}.   Also, 
$K\bigl((\operatorname{id}\otimes\omega)(W)\xi\bigr)=(\operatorname{id}\otimes\bar{\omega})(W)
(K\xi)$.
\end{proof}

\begin{cor}
\begin{enumerate}
  \item The result of Proposition~\ref{idomegaWK} means that we have the inclusion: 
$(\operatorname{id}\otimes\bar{\omega})(W)\,K\subseteq K\,(\operatorname{id}\otimes\omega)(W)$, 
$\forall\omega\in{\mathcal B}({\mathcal H_{\varphi}})_*$.
  \item By taking adjoints, we also have: $(\operatorname{id}\otimes\bar{\omega})(W^*)\,K^*\subseteq 
K^*\,(\operatorname{id}\otimes\omega)(W^*)$.  Phrased in a different way, this means that for 
$\xi'\in{\mathcal D}(K^*)$ and $\omega\in{\mathcal B}({\mathcal H}_{\varphi})_*$, we have: 
$(\operatorname{id}\otimes\omega)(W^*)\xi'\in{\mathcal D}(K^*)$, and 
$$
K^*\bigl((\operatorname{id}\otimes\omega)(W^*)\xi'\bigr)=(\operatorname{id}\otimes\bar{\omega})(W^*)(K^*\xi').
$$
\end{enumerate}
\end{cor}

\subsection{The operator $T$ and some modular theory}\label{subsec4.2}

Consider the operator $T$ on ${\mathcal H}_{\varphi}$, which is the closure of the map $\Lambda_{\varphi}(x)
\mapsto\Lambda_{\varphi}(x^*)$, for $x\in{\mathfrak N}_{\varphi}\cap{\mathfrak N}_{\varphi}^*$.  This is 
exactly the operator studied in the Tomita--Takesaki modular theory \cite{Tk2}.  The usual notation 
in the modular theory for this operator is $S$, but we use $T$ in our case to avoid confusion with the 
antipode map. 

To be precise, denote by $\tilde{\varphi}$ the lift of $\varphi$ to the von Neumann algebra 
$M=\pi_{\varphi}(A)''$.  It is a normal, semi-finite, faithful (n.s.f.) weight.  While the operator $T$ 
is typically associated with the von Neumann algebra weight, it is known that the GNS Hilbert spaces 
${\mathcal H}_{\tilde{\varphi}}$ and ${\mathcal H}_{\varphi}$ are same.  The operator $T$ is conjugate 
linear and involutive.  There exists the polar decomposition $T=J\nabla^{\frac12}$, where $J$ is 
a conjugate linear isomorphism (called the ``modular conjugation'') and $\nabla=T^*T$ is a positive 
operator (called the ``modular operator'').  We have $J=J^*$, $J^2=1$, and $J\nabla J=\nabla^{-1}$. 
If we regard the elements of $M$ as operators on ${\mathcal H}_{\varphi}$, then it is known that 
$JMJ=M'$.  For all $t\in\mathbb{R}$, we also have $\nabla^{it}M\nabla^{-it}=M$.  From this last 
observation, one defines the ``modular automorphism group'' $(\sigma_t^{\tilde{\varphi}})_{t\in\mathbb{R}}$, 
by 
$$\sigma_t^{\tilde{\varphi}}(x):=\nabla^{it}x\nabla^{-it},\quad x\in M.$$  
We have $\tilde{\varphi}=\tilde{\varphi}\circ\sigma_t^{\tilde{\varphi}}$, and the modular 
automorphism group satisfies the ``KMS condition''.  All these are standard results.  Refer to \cite{Tk2}, \cite{Str}. 
In our case, our $\varphi$ being a KMS weight implies that $(\sigma_t^{\tilde{\varphi}})$ leaves the 
$C^*$-algebra $A$ invariant, and the restriction of $(\sigma_t^{\tilde{\varphi}})$ to $A$ coincides 
with the automorphism group $(\sigma_t^{\varphi})$ associated with $\varphi$.

Our aim here is to show that there is a result similar to Proposition~\ref{idomegaWK}, 
involving the operator $T$ and the other leg of the operator $W$. 

\begin{prop}\label{omegaidW*T}
Let $T$ be as above. 
\begin{enumerate}
  \item For $\zeta\in{\mathcal D}(T)$ and for any $\theta\in{\mathcal B}({\mathcal H})_*$, 
we have: $(\theta\otimes\operatorname{id})(W^*)\zeta\in{\mathcal D}(T)$, and
$$
T\bigl((\theta\otimes\operatorname{id})(W^*)\zeta\bigr)=(\bar{\theta}\otimes\operatorname{id})(W^*)(T\zeta).
$$
So we have: $(\bar{\theta}\otimes\operatorname{id})(W^*)\,T\subseteq 
T\,(\theta\otimes\operatorname{id})(W^*)$.
  \item For $\zeta'\in{\mathcal D}(T^*)$ and $\theta\in{\mathcal B}({\mathcal H})_*$, 
we have: $(\theta\otimes\operatorname{id})(W)\zeta'\in{\mathcal D}(T^*)$, and
$$
T^*\bigl((\theta\otimes\operatorname{id})(W)\zeta'\bigr)=(\bar{\theta}\otimes\operatorname{id})(W)(T^*\zeta').
$$
So we have: $(\bar{\theta}\otimes\operatorname{id})(W)\,T^*\subseteq 
T^*\,(\theta\otimes\operatorname{id})(W)$.
\end{enumerate}
\end{prop}

\begin{proof}
Without loss of generality, let $\zeta=\Lambda_{\varphi}(x)$, for some $x\in{\mathfrak N}_{\varphi}
\cap{\mathfrak N}_{\varphi}^*$.  By the left invariance of $\varphi$, we know 
$(\theta\otimes\operatorname{id})(\Delta x)\in{\mathfrak N}_{\varphi}\cap{\mathfrak N}_{\varphi}^*$. 
Also, by definition of $W$, we have: $(\theta\otimes\operatorname{id})(W^*)\Lambda_{\varphi}(x)
=\Lambda_{\varphi}\bigl((\theta\otimes\operatorname{id})(\Delta x)\bigr)$.  Therefore, 
\begin{align}
T\bigl((\theta\otimes\operatorname{id})(W^*)\Lambda_{\varphi}(x)\bigr)
&=T\bigl(\Lambda_{\varphi}[(\theta\otimes\operatorname{id})(\Delta x)]\bigr)
=\Lambda_{\varphi}\bigl([(\theta\otimes\operatorname{id})(\Delta x)]^*\bigr)   \notag \\
&=\Lambda_{\varphi}\bigl((\bar{\theta}\otimes\operatorname{id})(\Delta(x^*))\bigr)
=(\bar{\theta}\otimes\operatorname{id})(W^*)\Lambda_{\varphi}(x^*)   \notag \\
&=(\bar{\theta}\otimes\operatorname{id})(W^*)\bigl(T\Lambda_{\varphi}(x)\bigr). 
\notag
\end{align}
Since $\zeta=\Lambda_{\varphi}(x)$ is arbitrary, this proves the first result.
The second result is obtained by taking the adjoints.
\end{proof}

\subsection{Some technical results concerning the polar decompositions of the 
operators $K$ and $T$.}\label{subsec4.3}

In \S\ref{subsec4.2}, we considered the operator $T$ on ${\mathcal H}_{\varphi}$, its 
polar decomposition $T=J\nabla^{\frac12}$, the modular automorphism group 
$(\sigma^{\tilde{\varphi}}_t)_{t\in\mathbb{R}}$, and their properties. 

In \S\ref{subsec4.1}, we defined the operator $K$ on ${\mathcal H}={\mathcal H}_{\psi}$. 
As in the case of $T$, since $K$ is also an involutive conjugate linear operator, we can consider 
its polar decomposition, given by $K=IL^{\frac12}$, where $L=K^*K$ is a positive operator 
and $I$ is a conjugate linear isomorphism.  We will have: $I^*=I$, $I^2=1$, and $ILI=L^{-1}$. 
In addition, we can consider the ``scaling group'' $(\tau_t)_{t\in\mathbb{R}}$, which is 
a norm-continuous automorphism group on ${\mathcal B}({\mathcal H})$ given by 
$\tau_t(\,\cdot\,):=L^{it}\,\cdot\,L^{-it}$.

Main goal in this subsection is to show that $(L^{it}\otimes\nabla^{it})W
(L^{-it}\otimes\nabla^{-it})=W$, $t\in\mathbb{R}$ (Proposition~\ref{LNablaequality}) and 
that $(I\otimes J)W(I\otimes J)=W^*$ (Proposition~\ref{IJWequality}).  From these results, 
we can show later (in \S\ref{subsec4.4}) that the map defined by $S:=R\circ\tau_{-\frac{i}{2}}$, 
where $R(x)=Ix^*I$ and $\tau_{-\frac{i}{2}}$ is the analytic generator of $(\tau_t)$ at $t=-\frac{i}{2}$, 
is our antipode. 

In fact, this is exactly how things are done in the quantum group case (see Proposition~2.20 
in \cite{VDvN}, for instance).  Our situation is more complicated, however, because the 
methods that worked in the quantum group case cannot be applied, due to the fact that 
$W$ is not a unitary (it is only a partial isometry).  This makes our discussion below 
to appear somewhat long-winded, but the overall strategy is still the same.

As before, regard $A=\pi(A)\in{\mathcal B}({\mathcal H})$, via the GNS 
representation.  Let us first begin by collecting some technical results, given in the 
following two lemmas.  Lemma~\ref{aplemma1} below is analogous to Lemma~3.14 of \cite{KuVa}. 
Here, ${\mathcal T}_{\tilde{\varphi}}$ denotes the Tomita ${}^*$-algebra, a dense subalgebra 
in $\pi_{\varphi}(A)''$. 

\begin{lem}\label{aplemma1}
Let $a,b\in{\mathcal T}_{\tilde{\varphi}}$ and $p\in{\mathfrak N}_{\varphi}$.  Then
$$
(\operatorname{id}\otimes\omega_{\Lambda_{\tilde{\varphi}}(a\sigma^{\tilde{\varphi}}_{-i}(b^*)),
\Lambda_{\varphi}(p)})(W)=\pi\bigl((\operatorname{id}\otimes\omega_{\Lambda_{\tilde{\varphi}}(a),
\Lambda_{\tilde{\varphi}}(b)})(\Delta(p^*))\bigr).
$$
\end{lem}

\begin{proof}
For any $q\in{\mathfrak N}_{\varphi}$, by an earlier result in Proposition~\ref{idomegaW}, we have
$$
(\operatorname{id}\otimes\omega_{\Lambda_{\varphi}(q),\Lambda_{\varphi}(p)})(W)
=(\operatorname{id}\otimes\varphi)\bigl(\Delta(p^*)(1\otimes q)\bigr).
$$
Let $\theta$ be an arbitrary linear functional.  Then from the above, we have 
\begin{align}
\theta\bigl((\operatorname{id}\otimes\omega_{\Lambda_{\varphi}(q),\Lambda_{\varphi}(p)})(W)\bigr)
&=\theta\bigl((\operatorname{id}\otimes\varphi)(\Delta(p^*)(1\otimes q))\bigr)
=\varphi\bigl((\theta\otimes\operatorname{id})(\Delta(p^*))q\bigr)   \notag \\
&=\varphi\bigl(((\bar{\theta}\otimes\operatorname{id})(\Delta p))^*q\bigr) 
=\bigl\langle\Lambda_{\varphi}(q),\Lambda_{\varphi}((\bar{\theta}\otimes\operatorname{id})
(\Delta p))\bigr\rangle.
\notag
\end{align}
In terms of the lifted weight $\tilde{\varphi}$ (recall that ${\mathcal H}_{\tilde{\varphi}}
={\mathcal H}_{\varphi}$), this is same as
\begin{equation}\label{(aplemma1eqn1)}
\theta\bigl((\operatorname{id}\otimes\omega_{\Lambda_{\varphi}(q),\Lambda_{\varphi}(p)})(W)\bigr)
=\bigl\langle\Lambda_{\tilde{\varphi}}(\pi(q)),\Lambda_{\tilde{\varphi}}
(\pi[(\bar{\theta}\otimes\operatorname{id})(\Delta p)])\bigr\rangle,
\end{equation}
where we wrote $\pi(\,\cdot\,)$ to distinguish the elements in $A$ from the generic 
elements in the von Neumann algebra $\pi_{\varphi}(A)''$.

Equation~\eqref{(aplemma1eqn1)} holds true for any $q\in{\mathfrak N}_{\varphi}$.  This means that 
for any $x\in{\mathfrak N}_{\tilde{\varphi}}$, we will also have
\begin{equation}\label{(aplemma1eqn2)}
\theta\bigl((\operatorname{id}\otimes\omega_{\Lambda_{\tilde{\varphi}}(x),\Lambda_{\varphi}(p)})(W)\bigr)
=\bigl\langle\Lambda_{\tilde{\varphi}}(x),\Lambda_{\tilde{\varphi}}(\pi[(\bar{\theta}\otimes
\operatorname{id})(\Delta p)])\bigr\rangle.
\end{equation}

In particular, let $x=a\sigma^{\tilde{\varphi}}_{-i}(b^*)$, for $a,b\in{\mathcal T}_{\tilde{\varphi}}$. 
Then Equation~\eqref{(aplemma1eqn2)} becomes
\begin{align}
&\theta\bigl((\operatorname{id}\otimes\omega_{\Lambda_{\tilde{\varphi}}(a\sigma^{\tilde{\varphi}}_{-i}(b^*)),
\Lambda_{\varphi}(p)})(W)\bigr)
=\tilde{\varphi}\bigl(\pi[(\theta\otimes\operatorname{id})(\Delta(p^*))]x\bigr) \notag \\
&=\tilde{\varphi}\bigl(\pi[(\theta\otimes\operatorname{id})(\Delta(p^*))]
a\sigma^{\tilde{\varphi}}_{-i}(b^*)\bigr) 
=\tilde{\varphi}\bigl(b^*\pi[(\theta\otimes\operatorname{id})(\Delta(p^*))]a\bigr) \notag \\
&=\bigl\langle\pi[(\theta\otimes\operatorname{id})(\Delta(p^*))]\Lambda_{\tilde{\varphi}}(a),
\Lambda_{\tilde{\varphi}}(b)\bigr\rangle
=\omega_{\Lambda_{\tilde{\varphi}}(a),\Lambda_{\tilde{\varphi}}(b)}
\bigl(\pi[(\theta\otimes\operatorname{id})(\Delta(p^*))]\bigr)   \notag \\
&=\theta\bigl(\pi[(\operatorname{id}\otimes
\omega_{\Lambda_{\tilde{\varphi}}(a),\Lambda_{\tilde{\varphi}}(b)})(\Delta(p^*))]\bigr).
\notag
\end{align}
This result is true for any functional $\theta$.  It follows that 
$$
(\operatorname{id}\otimes\omega_{\Lambda_{\tilde{\varphi}}(a\sigma^{\tilde{\varphi}}_{-i}(b^*)),
\Lambda_{\varphi}(p)})(W)=\pi\bigl((\operatorname{id}\otimes\omega_{\Lambda_{\tilde{\varphi}}(a),
\Lambda_{\tilde{\varphi}}(b)})(\Delta(p^*))\bigr),
$$
as elements in $A=\pi(A)$.
\end{proof}

\begin{rem}
By the property of the Tomita ${}^*$-algebra, we know that the elements of the form 
$\Lambda_{\tilde{\varphi}}(a\sigma^{\tilde{\varphi}}_{-i}(b^*))$, $a,b\in{\mathcal T}_{\tilde{\varphi}}$, 
span a dense subset in ${\mathcal H}_{\tilde{\varphi}}={\mathcal H}_{\varphi}$.  Meanwhile, now that 
we have proved the result, and since we can see that the equality holds as elements in $A=\pi(A)$, 
we may as well write the result of the lemma as
\begin{equation}\label{(aplemma1eqn3)}
(\operatorname{id}\otimes\omega_{\Lambda_{\tilde{\varphi}}(a\sigma^{\tilde{\varphi}}_{-i}(b^*)),
\Lambda_{\varphi}(p)})(W)=(\operatorname{id}\otimes\omega_{\Lambda_{\tilde{\varphi}}(a),
\Lambda_{\tilde{\varphi}}(b)})(\Delta(p^*)),
\end{equation}
for $a,b\in{\mathcal T}_{\tilde{\varphi}}$ and $p\in{\mathfrak N}_{\varphi}$.
\end{rem}

The following Lemma~\ref{aplemma2}, which is analogous to Lemma~5.7 of \cite{KuVa}, is a consequence of 
the previous result.  

\begin{lem}\label{aplemma2}
Let $a\in{\mathcal T}_{\tilde{\varphi}}$ and $x\in{\mathfrak N}_{\tilde{\varphi}}
\cap{\mathfrak N}_{\tilde{\varphi}}^*$.  Then
$$
\bigl((\operatorname{id}\otimes\omega_{\Lambda_{\tilde{\varphi}}(x),
\Lambda_{\tilde{\varphi}}(a)})(W^*)\bigr)^*
=(\operatorname{id}\otimes\omega_{\Lambda_{\tilde{\varphi}}(x^*),
\Lambda_{\tilde{\varphi}}(\sigma^{\tilde{\varphi}}_{-i}(a^*))})(W^*),
$$
which can be also written as
$$
(\operatorname{id}\otimes\omega_{\Lambda_{\tilde{\varphi}}(a),
\Lambda_{\tilde{\varphi}}(x)})(W)
=(\operatorname{id}\otimes\omega_{\Lambda_{\tilde{\varphi}}(x^*),
\Lambda_{\tilde{\varphi}}(\sigma^{\tilde{\varphi}}_{-i}(a^*))})(W^*).
$$
\end{lem}

\begin{proof}
Let $p\in{\mathfrak N}_{\varphi}$, $a,b\in{\mathcal T}_{\tilde{\varphi}}$, and consider 
the result of the previous lemma, given in Equation~\eqref{(aplemma1eqn3)}.  Take the adjoint. 
Then we have:
$$
(\operatorname{id}\otimes\omega_{\Lambda_{\varphi}(p),
\Lambda_{\tilde{\varphi}}(a\sigma^{\tilde{\varphi}}_{-i}(b^*))})(W^*)
=(\operatorname{id}\otimes\omega_{\Lambda_{\tilde{\varphi}}(b),\Lambda_{\tilde{\varphi}}(a)})(\Delta p).
$$
This holds true for any $p\in{\mathfrak N}_{\varphi}$.  Therefore, the result extends to any 
$x\in{\mathfrak N}_{\tilde{\varphi}}\,\bigl(\subseteq\pi_{\varphi}(A)''\bigr)$, with 
$\tilde{\Delta}x:=W^*(1\otimes x)W$ (see Proposition~\ref{vNqgroupoid}).  That is, 
\begin{equation}\label{(aplemma2eqn0)}
(\operatorname{id}\otimes\omega_{\Lambda_{\tilde{\varphi}}(x),
\Lambda_{\tilde{\varphi}}(a\sigma^{\tilde{\varphi}}_{-i}(b^*))})(W^*)
=(\operatorname{id}\otimes\omega_{\Lambda_{\tilde{\varphi}}(b),\Lambda_{\tilde{\varphi}}(a)})
(\tilde{\Delta}x),{\text { for $x\in{\mathfrak N}_{\tilde{\varphi}}$.}}
\end{equation}
If $x\in{\mathfrak N}_{\tilde{\varphi}}\cap{\mathfrak N}_{\tilde{\varphi}}^*$, from 
Equation~\eqref{(aplemma2eqn0)}, we will have:
\begin{align}
\bigl((\operatorname{id}\otimes\omega_{\Lambda_{\tilde{\varphi}}(x),
\Lambda_{\tilde{\varphi}}(a\sigma^{\tilde{\varphi}}_{-i}(b^*))})(W^*)\bigr)^*
&=(\operatorname{id}\otimes\omega_{\Lambda_{\tilde{\varphi}}(a),\Lambda_{\tilde{\varphi}}(b)})
(\tilde{\Delta}(x^*))   \notag \\
&=(\operatorname{id}\otimes\omega_{\Lambda_{\tilde{\varphi}}(x^*),
\Lambda_{\tilde{\varphi}}(b\sigma^{\tilde{\varphi}}_{-i}(a^*))})(W^*).
\label{(aplemma2eqn1)}
\end{align}

Next, we need a technical result from the weight theory: Namely, find a bounded net 
$(b_k)_{k\in P}$ in ${\mathcal T}_{\tilde{\varphi}}$, where $P$ is an index set, such that 
$(b_k)_{k\in P}$ converges strongly${}^*$ to 1, and that 
$\bigl(\sigma^{\tilde{\varphi}}_{\frac{i}{2}}(b_k)\bigr)_{k\in P}$ is bounded and 
converges strongly${}^*$ to 1.  For instance, take a bounded net $(u_k)_{x\in P}$ 
in ${\mathfrak N}_{\tilde{\varphi}}\cap{\mathfrak N}_{\tilde{\varphi}}^*$ converging 
strongly${}^*$ to 1, and define
$b_k=\frac{1}{\sqrt{\pi}}\int\exp(-t^2)\sigma^{\tilde{\varphi}}_t(u_k)\,dt$.
(See the proof of Lemma~5.7 in \cite{KuVa}, or the proof of Proposition~1.14 
in \cite{KuVaweightC*}.)  Then, by a standard result in modular theory (see also 
Lemma~1.2 in Part I \cite{BJKVD_qgroupoid1}), we have the following convergence: 
$$
\Lambda_{\tilde{\varphi}}\bigl(a\sigma^{\tilde{\varphi}}_{-i}(b_k^*)\bigr)
=J_{\tilde{\varphi}}\pi_{\tilde{\varphi}}\bigl(\sigma^{\tilde{\varphi}}_{-\frac{i}{2}}(b_k^*)\bigr)^*
J_{\tilde{\varphi}}\Lambda_{\tilde{\varphi}}(a)
=J_{\tilde{\varphi}}\pi_{\tilde{\varphi}}\bigl(\sigma^{\tilde{\varphi}}_{\frac{i}{2}}(b_k)\bigr)
J_{\tilde{\varphi}}\Lambda_{\tilde{\varphi}}(a)\longrightarrow\Lambda_{\tilde{\varphi}}(a).
$$
Now let $b=b_k$ in Equation~\eqref{(aplemma2eqn1)} and consider the limit.  Then we obtain:
$$
\bigl((\operatorname{id}\otimes\omega_{\Lambda_{\tilde{\varphi}}(x),
\Lambda_{\tilde{\varphi}}(a)})(W^*)\bigr)^*
=(\operatorname{id}\otimes\omega_{\Lambda_{\tilde{\varphi}}(x^*),
\Lambda_{\tilde{\varphi}}(\sigma^{\tilde{\varphi}}_{-i}(a^*))})(W^*),
$$
proving the lemma.  The second equation is a quick consequence of Lemma~\ref{omega_xizetaLem}.
\end{proof}

Using Lemma~\ref{aplemma2}, we can prove the next result, which plays a central role 
in what follows.  Recall here that $L=K^*K$ and $\nabla=T^*T$, the positive operators 
arising from the polar decompositions of $K$ and $T$, respectively.

\begin{prop}\label{inequalityWLNabla}
$$
W^*(L\otimes\nabla)\subseteq(L\otimes\nabla)W^*
\quad{{\text and }}\quad W(L\otimes\nabla)\subseteq(L\otimes\nabla)W.
$$
\end{prop}

\begin{proof}
Let $\xi,\xi'\in{\mathcal D}(L)$ and $\zeta,\zeta'\in{\mathcal D}(\nabla)$.  Without loss of generality, 
we may let $\zeta=\Lambda_{\tilde{\varphi}}(a)$, $\zeta'=\Lambda_{\tilde{\varphi}}(b)$, 
for $a,b\in{\mathcal T}_{\tilde{\varphi}}$.  Then $\nabla\zeta=\nabla\Lambda_{\tilde{\varphi}}(a)
=\Lambda_{\tilde{\varphi}}\bigl(\sigma^{\tilde{\varphi}}_{-i}(a)\bigr)$, by modular theory. 
So we have:
\begin{equation}\label{(inequalityWLNablaeqn1)}
\bigl\langle W^*(L\xi\otimes\nabla\zeta),\xi'\otimes\zeta'\bigr\rangle=\bigl\langle(\operatorname{id}
\otimes\omega_{\Lambda_{\tilde{\varphi}}(\sigma^{\tilde{\varphi}}_{-i}(a)),\Lambda_{\tilde{\varphi}}(b)})
(W^*)K^*K\xi,\xi'\bigr\rangle,
\end{equation}
where we used $L=K^*K$.  But by Proposition~\ref{idomegaWK}, we have:
$$
(\operatorname{id}\otimes\omega_{\Lambda_{\tilde{\varphi}}
(\sigma^{\tilde{\varphi}}_{-i}(a)),\Lambda_{\tilde{\varphi}}(b)})(W^*)K^*K\xi
=K^*(\operatorname{id}\otimes\omega_{\Lambda_{\tilde{\varphi}}(b),
\Lambda_{\tilde{\varphi}}(\sigma^{\tilde{\varphi}}_{-i}(a))})(W^*)K\xi.
$$
Meanwhile, by Lemma~\ref{aplemma2}, we have
$$
(\operatorname{id}\otimes\omega_{\Lambda_{\tilde{\varphi}}(b),
\Lambda_{\tilde{\varphi}}(\sigma^{\tilde{\varphi}}_{-i}(a))})(W^*)
=(\operatorname{id}\otimes\omega_{\Lambda_{\tilde{\varphi}}(a^*),
\Lambda_{\tilde{\varphi}}(b^*)})(W).
$$
So Equation~\eqref{(inequalityWLNablaeqn1)} now becomes
\begin{align}
&\bigl\langle W^*(L\xi\otimes\nabla\zeta),\xi'\otimes\zeta'\bigr\rangle   \notag \\
&=\bigl\langle K^*(\operatorname{id}\otimes\omega_{\Lambda_{\tilde{\varphi}}(a^*),
\Lambda_{\tilde{\varphi}}(b^*)})(W)K\xi,\xi'\bigr\rangle   
=\big\langle K^*K(\operatorname{id}\otimes\omega_{\Lambda_{\tilde{\varphi}}(b^*),
\Lambda_{\tilde{\varphi}}(a^*)})(W)\xi,\xi'\bigr\rangle    \notag \\
&=\big\langle (\operatorname{id}\otimes\omega_{\Lambda_{\tilde{\varphi}}(b^*),
\Lambda_{\tilde{\varphi}}(a^*)})(W)\xi,K^*K\xi'\bigr\rangle 
=\bigl\langle W(\xi\otimes\Lambda_{\tilde{\varphi}}(b^*)),K^*K\xi'
\otimes\Lambda_{\tilde{\varphi}}(a^*)\bigr\rangle  \notag \\
&=\bigl\langle W(\xi\otimes T\zeta'),L\xi'\otimes T\zeta\bigr\rangle.
\label{(inequalityWLNablaeqn2)}
\end{align}
The second equality used Proposition~\ref{idomegaWK}.  Since 
$a,b\in{\mathcal T}_{\tilde{\varphi}}$ are arbitrary, Equation~\eqref{(inequalityWLNablaeqn2)} 
is true for any  $\xi,\xi'\in{\mathcal D}(L)$ and any $\zeta,\zeta'\in{\mathcal D}(\nabla)$.

Next, consider the right hand side of Equation~\eqref{(inequalityWLNablaeqn2)}. 
We have
\begin{align}
\bigl\langle W(\xi\otimes T\zeta'),L\xi'\otimes T\zeta\bigr\rangle
&=\bigl\langle(\omega_{\xi,L\xi'}\otimes\operatorname{id})(W)T\zeta',T\zeta\bigr\rangle
\notag \\
&=\bigl\langle T\zeta',(\omega_{L\xi',\xi}\otimes\operatorname{id})(W^*)T\zeta\bigr\rangle
\notag \\
&=\bigl\langle T\zeta',T(\omega_{\xi,L\xi'}\otimes\operatorname{id})(W^*)\zeta\bigr\rangle,
\label{(inequalityWLNablaeqn3)}
\end{align}
where we used Proposition~\ref{omegaidW*T}\,(1) for the last equality.  Since $T$ is 
a conjugate linear operator, the right hand side becomes
$$
=\overline{\bigl\langle T^*T\zeta',(\omega_{\xi,L\xi'}\otimes\operatorname{id})(W^*)\zeta\bigr\rangle}
=\bigl\langle(\omega_{\xi,L\xi'}\otimes\operatorname{id})(W^*)\zeta,\nabla\zeta'\bigr\rangle, 
$$
using $\nabla=T^*T$.  Putting these together, Equation~\eqref{(inequalityWLNablaeqn3)} 
then becomes
\begin{align}
&\bigl\langle W(\xi\otimes T\zeta'),L\xi'\otimes T\zeta\bigr\rangle
\notag \\
&=\bigl\langle W^*(\xi\otimes\zeta),L\xi'\otimes\nabla\zeta'\bigr\rangle
=\bigl\langle (L\otimes\nabla)W^*(\xi\otimes\zeta),\xi'\otimes\zeta'\bigr\rangle.
\label{(inequalityWLNablaeqn4)}
\end{align}

Combine Equations~\eqref{(inequalityWLNablaeqn2)} and \eqref{(inequalityWLNablaeqn4)}. 
We conclude that
$$
W^*(L\xi\otimes\nabla\zeta)=(L\otimes\nabla)W^*(\xi\otimes\zeta),
\quad\forall\xi\in{\mathcal D}(L),\forall\zeta\in{\mathcal D}(\nabla).
$$
This proves the inclusion $W^*(L\otimes\nabla)\subseteq(L\otimes\nabla)W^*$.  The second 
inclusion is an immediate consequence obtained by taking adjoints.

\end{proof}

\begin{rem}
In the quantum group case, once the result like $W^*(L\otimes\nabla)\subseteq(L\otimes\nabla)W^*$ 
is obtained, one can immediately show that it is actually an equality, by taking advantage of the 
fact that $W$ is a unitary.  Unfortunately, the operator $W$ is not unitary in our situation, 
so the analogous approach does not work, and in fact the equality does not hold.  This requires us 
to look for a different approach, based on the fact that $W$ is a partial isometry. 
\end{rem}

We know from Theorem~\ref{Wpartialisometry} that $W$ is a partial isometry in ${\mathcal B}
({\mathcal H}\otimes{\mathcal H}_{\varphi})$, with $W^*W=E=(\pi\otimes\pi_{\varphi})(E)$ 
and $WW^*=G_L$.  In the below 
are some consequences of Proposition~\ref{inequalityWLNabla}.

\begin{prop}\label{lemWLNablaEG}
We have:
\begin{enumerate}
  \item $E(L\otimes\nabla)\subseteq(L\otimes\nabla)E$ and $G_L(L\otimes\nabla)\subseteq(L\otimes\nabla)G_L$
  \item $W^*(L\otimes\nabla)G_L=E(L\otimes\nabla)W^*$
  \item $W^*(L\otimes\nabla)W=E(L\otimes\nabla)E$ and $W(L\otimes\nabla)W^*=G_L(L\otimes\nabla)G_L$
\end{enumerate}
\end{prop}

\begin{proof}
(1). From Proposition~\ref{inequalityWLNabla}, we know that $W^*(L\otimes\nabla)\subseteq(L\otimes\nabla)W^*$ 
and $W(L\otimes\nabla)\subseteq(L\otimes\nabla)W$.  Combining, we have:
$$
E(L\otimes\nabla)=W^*W(L\otimes\nabla)\subseteq W^*(L\otimes\nabla)W\subseteq (L\otimes\nabla)W^*W
=(L\otimes\nabla)E.
$$

Similarly, we have:
$$
G_L(L\otimes\nabla)=WW^*(L\otimes\nabla)\subseteq W(L\otimes\nabla)W^*\subseteq (L\otimes\nabla)WW^*
=(L\otimes\nabla)G_L.
$$

(2). Again from $W^*(L\otimes\nabla)\subseteq(L\otimes\nabla)W^*$, we have:
$$
G_L(L\otimes\nabla)W=WW^*(L\otimes\nabla)W\subseteq W(L\otimes\nabla)W^*W=W(L\otimes\nabla)E.
$$
By taking adjoints, it becomes: $E(L\otimes\nabla)W^*\subseteq W^*(L\otimes\nabla)G_L$.  
Meanwhile, since $EW^*=W^*WW^*=W^*$ and $W^*G_L=W^*WW^*=W^*$, we have:
$$
W^*(L\otimes\nabla)G_L=EW^*(L\otimes\nabla)G_L\subseteq E(L\otimes\nabla)W^*G_L=E(L\otimes\nabla)W^*.
$$

Combining the two results, we see that
$$
W^*(L\otimes\nabla)G_L=E(L\otimes\nabla)W^*.
$$

(3). Since $G_LW=WW^*W=W$, from (2) we have:
$$
W^*(L\otimes\nabla)W=W^*(L\otimes\nabla)G_LW=E(L\otimes\nabla)W^*W=E(L\otimes\nabla)E.
$$

From (2), taking adjoint, we also see that $W(L\otimes\nabla)E=G_L(L\otimes\nabla)W$.
Since $EW^*=W^*WW^*=W^*$, we have:
$$
W(L\otimes\nabla)W^*=W(L\otimes\nabla)EW^*=G_L(L\otimes\nabla)WW^*=G_L(L\otimes\nabla)G_L.
$$
\end{proof}

\begin{cor}
We may write the two results in (3) of the above proposition in a more symmetric form:
$$
W^*\bigl(G_L(L\otimes\nabla)G_L\bigr)W=E(L\otimes\nabla)E
$$
$$
W\bigl(E(L\otimes\nabla)E\bigr)W^*=G_L(L\otimes\nabla)G_L.
$$
\end{cor}

\begin{proof}
Use $WE=W$, $G_LW=W$, $EW^*=W^*$, $W^*G_L=W^*$, then re-write (3) of 
Proposition~\ref{lemWLNablaEG}.
\end{proof}

The Corollary suggests us to look more closely at the operators $E(L\otimes\nabla)E$ 
and $G_L(L\otimes\nabla)G_L$.  For this purpose, it is useful to note that as a consequence 
of $W$ being a partial isometry, we can write ${\mathcal H}\otimes{\mathcal H}_{\varphi}
=\operatorname{Ker}(W)\oplus\operatorname{Ker}(W)^{\perp}
=\operatorname{Ker}(W)\oplus\operatorname{Ran}(E)$ and also 
${\mathcal H}\otimes{\mathcal H}_{\varphi}=\operatorname{Ker}(W^*)\oplus\operatorname{Ran}(G_L)$. 
Then the operators $E(L\otimes\nabla)E$ and $G_L(L\otimes\nabla)G_L$ are none other than 
the restrictions of $L\otimes\nabla$ onto the subspaces $\operatorname{Ran}(E)$ and 
$\operatorname{Ran}(G_L)$, respectively.

\begin{prop}\label{LNablaonRanERanG}
We have:
\begin{enumerate}
  \item $(L\otimes\nabla)E=E(L\otimes\nabla)E$
  \item $(L\otimes\nabla)G_L=G_L(L\otimes\nabla)G_L$
\end{enumerate}  
\end{prop}

\begin{proof}
By Proposition~\ref{lemWLNablaEG}\,(1), we have: $E(L\otimes\nabla)E\subseteq (L\otimes\nabla)E$, 
with $E$ being a projection.  We may regard them as operators restricted to the subspace $\operatorname{Ran}(E)$. 
It is clear that ${\mathcal D}\bigl(E(L\otimes\nabla)E\bigr)=\operatorname{Ran}(E)\cap{\mathcal D}(L\otimes\nabla)
={\mathcal D}\bigl((L\otimes\nabla)E\bigr)$.  Since the domains are same, we see that $E(L\otimes\nabla)E
=(L\otimes\nabla)E$.

Similarly, regarded as operators restricted to the subspace $\operatorname{Ran}(G_L)$, we have: 
$G_L(L\otimes\nabla)G_L=(L\otimes\nabla)G_L$.
\end{proof}

\begin{cor}
For the (unbounded) operator $L\otimes\nabla$ on ${\mathcal H}\otimes{\mathcal H}_{\varphi}$, we have: 
\begin{enumerate}
  \item $(L\otimes\nabla)|_{\operatorname{Ran}(E)}$ is an operator on 
the subspace $\operatorname{Ran}(E)$.
  \item $(L\otimes\nabla)|_{\operatorname{Ran}(G_L)}$ is an operator on the subspace 
$\operatorname{Ran}(G_L)$.
  \item $(L\otimes\nabla)|_{\operatorname{Ker}(W)}$ is an operator on 
the subspace $\operatorname{Ker}(W)$.
  \item $(L\otimes\nabla)|_{\operatorname{Ker}(W^*)}$ is an operator on 
the subspace $\operatorname{Ker}(W^*)$.
  \item A typical element of ${\mathcal D}\bigl(W(L\otimes\nabla)\bigr)$ 
can be written as $\zeta=\zeta_0\oplus\zeta_1$, where $\zeta_0\in\operatorname{Ker}(W)
\cap{\mathcal D}(L\otimes\nabla)$ and $\zeta_1\in\operatorname{Ran}(E)\cap{\mathcal D}
(L\otimes\nabla)$.
  \item A typical element of ${\mathcal D}\bigl(W^*(L\otimes\nabla)\bigr)$ 
can be written as $\zeta=\zeta_0\oplus\zeta_1$, where $\zeta_0\in\operatorname{Ker}(W^*)
\cap{\mathcal D}(L\otimes\nabla)$ and $\zeta_1\in\operatorname{Ran}(E)\cap{\mathcal D}
(L\otimes\nabla)$.
\end{enumerate}
\end{cor}

\begin{proof}
(1), (2). These are immediate consequences of Proposition~\ref{LNablaonRanERanG}.

(3), (4). Consider $(L\otimes\nabla)(1-E)$ and $(1-E)(L\otimes\nabla)(1-E)$, and consider an arbitrary element 
$\zeta\in\operatorname{Ran}(1-E)\cap{\mathcal D}(L\otimes\nabla)$, the common domain for both.  Then 
by Proposition~\ref{lemWLNablaEG}\,(1), we have $E(L\otimes\nabla)\zeta=(L\otimes\nabla)E\zeta=0$.
From this it is easy to conclude that 
$$(L\otimes\nabla)(1-E)=(1-E)(L\otimes\nabla)(1-E).$$
Since $\operatorname{Ker}(W)=\operatorname{Ran}(1-E)$, the result follows.  
Proof for (4) is similar, since $\operatorname{Ker}(W^*)=\operatorname{Ran}(1-G_L)$.

(5), (6). Consider $\zeta\in{\mathcal D}\bigl(W(L\otimes\nabla)\bigr)={\mathcal D}(L\otimes\nabla)$ and write 
$\zeta=\zeta_0\oplus\zeta_1$, where $\zeta_0\in\operatorname{Ker}(W)$ and $\zeta_1\in\operatorname{Ran}(E)$.  
Then, from the fact that $W(L\otimes\nabla)\subseteq (L\otimes\nabla)W$ (by Proposition~\ref{inequalityWLNabla}), 
we must have $W\zeta\in{\mathcal D}(L\otimes\nabla)$, and 
$$
W(L\otimes\nabla)\zeta=(L\otimes\nabla)W\zeta=(L\otimes\nabla)W\zeta_1
=(L\otimes\nabla)G_LW\zeta_1,
$$
because $G_LW=WW^*W=W$.  But 
$$
(L\otimes\nabla)G_LW=G_L(L\otimes\nabla)G_LW=G_L(L\otimes\nabla)W
=W(L\otimes\nabla)E,
$$
using Proposition~\ref{LNablaonRanERanG}\,(2) and Proposition~\ref{lemWLNablaEG}\,(4).  Combining 
the results, we have $W(L\otimes\nabla)(\zeta_0\oplus\zeta_1)=W(L\otimes\nabla)\zeta_1$, since $\zeta_1
\in\operatorname{Ran}(E)$. 
In this way, we showed that $\zeta_1\in\operatorname{Ran}(E)\cap{\mathcal D}(L\otimes\nabla)$ such that 
$W(L\otimes\nabla)\zeta_1=(L\otimes\nabla)W\zeta_1$.  Also as a consequence, we see that 
$\zeta_0\in\operatorname{Ker}(W)\cap{\mathcal D}(L\otimes\nabla)$. 
Proof for (6) is similar to that of (5).
\end{proof}

\begin{rem}
In (5) of the above Corollary, we saw that a typical element of ${\mathcal D}\bigl(W(L\otimes\nabla)\bigr)$ 
can be written as $\zeta=\zeta_0\oplus\zeta_1$, where $\zeta_0\in\operatorname{Ker}(W)
\cap{\mathcal D}(L\otimes\nabla)$ and $\zeta_1\in\operatorname{Ran}(E)\cap{\mathcal D}
(L\otimes\nabla)$.  We also saw that $W\zeta_1\in{\mathcal D}(L\otimes\nabla)$.  On the other hand, 
an element in ${\mathcal D}\bigl((L\otimes\nabla)W\bigr)$ would be written as $\zeta=\zeta_0\oplus\zeta_1$, 
where $\zeta_0\in\operatorname{Ker}(W)$ and $\zeta_1\in\operatorname{Ran}(E)$ such that 
$W\zeta_1\in{\mathcal D}(L\otimes\nabla)$.  We do not expect to have $\operatorname{Ker}(W)
\cap{\mathcal D}(L\otimes\nabla)=\operatorname{Ker}(W)$ in general.  Therefore, 
$$W(L\otimes\nabla)\subsetneqq(L\otimes\nabla)W.$$  
Similarly, $W^*(L\otimes\nabla)\subsetneqq(L\otimes\nabla)W^*$. 
\end{rem}

This is different from the quantum group case, which makes things tricky.  We can find a way, nonetheless. 
The following is a re-interpretation of the results in Corollary of Proposition~\ref{lemWLNablaEG}

\begin{prop}\label{LNablaequationAB}
We have:
\begin{enumerate}
  \item $(L\otimes\nabla)|_{\operatorname{Ran}(E)}=W^*(L\otimes\nabla)|_{\operatorname{Ran}(G_L)}W$, 
as operators on $\operatorname{Ran}(E)$.
  \item $(L\otimes\nabla)|_{\operatorname{Ran}(G_L)}=W(L\otimes\nabla)|_{\operatorname{Ran}(E)}W^*$, 
as operators on $\operatorname{Ran}(G_L)$.
\end{enumerate}
\end{prop}

\begin{proof}
When restricted to subspaces, note that $W|_{\operatorname{Ran}(E)}$ and $W^*|_{\operatorname{Ran}(G_L)}$ 
may be regarded as onto isometries between $\operatorname{Ran}(E)$ and $\operatorname{Ran}(G_L)$. Also 
$W^*|_{\operatorname{Ran}(G_L)}W|_{\operatorname{Ran}(E)}
=\operatorname{Id}_{\operatorname{Ran}(E)}$ and $W|_{\operatorname{Ran}(E)}W^*|_{\operatorname{Ran}(G_L)}
=\operatorname{Id}_{\operatorname{Ran}(G_L)}$.
In this way, the results in Corollary of Proposition~\ref{lemWLNablaEG} can be re-interpreted as above.
\end{proof}

Since $L$, $\nabla$ are self-adjoint operators, we can perform functional calculus.  
In particular, we can consider $L^z$, $\nabla^z$, for any $z\in\mathbb{C}$.  
We next wish to explore the behavior of the operator $L^z\otimes\nabla^z$. 
Note first that by Corollary of Proposition~\ref{LNablaonRanERanG}, it is clear that 
$(L^z\otimes\nabla^z)|_{\operatorname{Ran}(E)}$ is an operator on $\operatorname{Ran}(E)$ 
and $(L^z\otimes\nabla^z)|_{\operatorname{Ker}(W)}$ is an operator on $\operatorname{Ker}(W)$. 
Similar also for the operators $(L^z\otimes\nabla^z)|_{\operatorname{Ran}(G_L)}$ and 
$(L^z\otimes\nabla^z)|_{\operatorname{Ker}(W^*)}$.

As for $(L^z\otimes\nabla^z)|_{\operatorname{Ran}(E)}$ and $(L^z\otimes\nabla^z)|_{\operatorname
{Ran}(G_L)}$, we have the following:

\begin{prop}\label{LNablafunctcalc}
For $z\in\mathbb{C}$, we have:
\begin{enumerate}
  \item $(L^z\otimes\nabla^z)|_{\operatorname{Ran}(E)}
=W^*(L^z\otimes\nabla^z)|_{\operatorname{Ran}(G_L)}W$, as operators on $\operatorname{Ran}(E)$.
  \item $(L^z\otimes\nabla^z)|_{\operatorname{Ran}(G_L)}
=W(L^z\otimes\nabla^z)|_{\operatorname{Ran}(E)}W^*$, as operators on $\operatorname{Ran}(G_L)$.
  \item $W(L^z\otimes\nabla^z)|_{\operatorname{Ran}(E)}=(L^z\otimes\nabla^z)|_{\operatorname{Ran}(G_L)}W$,
as operators on $\operatorname{Ran}(E)$.
  \item $W^*(L^z\otimes\nabla^z)|_{\operatorname{Ran}(G_L)}=(L^z\otimes\nabla^z)|_{\operatorname{Ran}(E)}W^*$, 
as operators on $\operatorname{Ran}(G_L)$.
\end{enumerate}
\end{prop}

\begin{proof}
(1), (2). It is clear that
$$
\bigl((L\otimes\nabla)|_{\operatorname{Ran}(E)}\bigr)^z=(L^z\otimes\nabla^z)|_{\operatorname{Ran}(E)}
$$
$$
\bigl((L\otimes\nabla)|_{\operatorname{Ran}(G_L)}\bigr)^z=(L^z\otimes\nabla^z)|_{\operatorname{Ran}(G_L)}
$$
for $z\in\mathbb{C}$.  Combine this result with that of Proposition~\ref{LNablaequationAB}, using 
the fact that $W|_{\operatorname{Ran}(E)}$ and $W^*|_{\operatorname{Ran}(G_L)}$ are inverses 
of each other.

(3), (4). By (1), we have:
$$
W(L^z\otimes\nabla^z)|_{\operatorname{Ran}(E)}=WW^*(L^z\otimes\nabla^z)W|_{\operatorname{Ran}(E)}
=G_L(L^z\otimes\nabla^z)W|_{\operatorname{Ran}(E)},
$$
which is just $(L^z\otimes\nabla^z)W|_{\operatorname{Ran}(E)}$, because 
$(L^z\otimes\nabla^z)|_{\operatorname{Ran}(G_L)}$ is an operator on $\operatorname{Ran}(G_L)$. 
(4) is proved similarly, using (2).
\end{proof}

On the level of the whole space ${\mathcal H}\otimes{\mathcal H}_{\varphi}$, however, we do not 
in general expect a result like $W(L^z\otimes\nabla^z)=(L^z\otimes\nabla^z)W$.  The domains 
may not agree, as noted in the case when $z=1$.  The best we can expect is the following:

\begin{prop}\label{inequalityWLNabla^z}
For any $z\in\mathbb{C}$, we have:
$$
W(L^z\otimes\nabla^z)\subseteq(L^z\otimes\nabla^z)W
\quad{\text { and }}\quad
W^*(L^z\otimes\nabla^z)\subseteq(L^z\otimes\nabla^z)W^*.
$$
\end{prop}

\begin{proof}
Suppose $\zeta\in{\mathcal D}(L^z\otimes\nabla^z)$ and write $\zeta=\zeta_0\oplus\zeta_1$, where 
$\zeta_0\in\operatorname{Ker}(W)$, $\zeta_1\in\operatorname{Ran}(E)$.  Then each of $\zeta_0$ 
and $\zeta_1$ must also be contained in the domain ${\mathcal D}(L^z\otimes\nabla^z)$.  That was 
the case when $z=1$ (see Corollary to Proposition~\ref{LNablaonRanERanG}), and it is because 
$(L^z\otimes\nabla^z)|_{\operatorname{Ker}(W)}$ and $(L^z\otimes\nabla^z)|_{\operatorname{Ran}(E)}$ 
are operators on $\operatorname{Ker}(W)$ and $\operatorname{Ran}(E)$ respectively.  Since 
$(L^z\otimes\nabla^z)|_{\operatorname{Ker}(W)}:\operatorname{Ker}(W)\cap{\mathcal D}(L^z\otimes\nabla^z)
\to\operatorname{Ker}(W)$, it is then easy to see that $W(L^z\otimes\nabla^z)\zeta_0$ is valid and vanishes. 
So we have:
$$
W(L^z\otimes\nabla^z)\zeta=W(L^z\otimes\nabla^z)(\zeta_0\oplus\zeta_1)
=W(L^z\otimes\nabla^z)\zeta_1=(L^z\otimes\nabla^z)W\zeta_1. 
$$
For the last equality, we used the fact that $\zeta_1\in\operatorname{Ran}(E)$, applying 
the result of the previous proposition.  It follows that for any $\zeta=\zeta_0\oplus\zeta_1$ in 
${\mathcal D}(L^z\otimes\nabla^z)$, 
$$
W(L^z\otimes\nabla^z)(\zeta_0\oplus\zeta_1)=(L^z\otimes\nabla^z)W\zeta_1
=(L^z\otimes\nabla^z)W(\zeta_0\oplus\zeta_1).
$$
This proves the first inclusion.  The second one is similar, using 
${\mathcal H}\otimes{\mathcal H}_{\varphi}=\operatorname{Ker}(W^*)\oplus\operatorname{Ran}(G_L)$.
\end{proof}

While we only have ``$\subseteq$'' in general, the situation is better if $z\in\mathbb{C}$ is purely 
imaginary, that is $z=it$, $t\in\mathbb{R}$.  If so, the operators $L^{it}$, $\nabla^{it}$ are bounded, 
so ${\mathcal D}(L^{it}\otimes\nabla^{it})$ would be the whole space ${\mathcal H}\otimes{\mathcal H}_{\varphi}$.  
Then we will have the equalities, as can be seen in the proposition below.

\begin{prop}\label{LNablaequality}
Let $t\in\mathbb{R}$.  Then we have the following equalities on the whole space 
${\mathcal H}\otimes{\mathcal H}_{\varphi}$.
\begin{enumerate}
  \item $W(L^{it}\otimes\nabla^{it})=(L^{it}\otimes\nabla^{it})W$
  \item $W^*(L^{it}\otimes\nabla^{it})=(L^{it}\otimes\nabla^{it})W^*$
\end{enumerate}
\end{prop}

\begin{proof}
Since $L^{it}$, $\nabla^{it}$ are bounded operators, there is no issue with their domains. 
We already know that $W(L^z\otimes\nabla^z)\subseteq(L^z\otimes\nabla^z)W$, 
in general.  So we indeed have the equality (1).  Similar for (2).
\end{proof}

Proposition~\ref{LNablaequality} plays a significant role as we define the antipode map later.
The next main result to show is Proposition~\ref{IJWequality} below.  
But first, we need some preparation.

\begin{lem}\label{lemIJW1}
Let $v\in{\mathcal D}(L^{\frac12})$ and $w\in{\mathcal D}(L^{-\frac12})$.  Then we have:
$$
(\omega_{v,w}\otimes\operatorname{id})(W)^*
=(\omega_{IL^{\frac12}v,IL^{-\frac12}w}\otimes\operatorname{id})(W).
$$
\end{lem}

\begin{proof}
For any $a,b\in{\mathfrak N}_{\varphi}$, we have:
$$
\bigl\langle(\omega_{IL^{\frac12}v,IL^{-\frac12}w}\otimes\operatorname{id})(W)
\Lambda_{\varphi}(a),\Lambda_{\varphi}(b)\big\rangle
=\bigl\langle(\operatorname{id}\otimes\omega_{\Lambda_{\varphi}(a),\Lambda_{\varphi}(b)})
(W)IL^{\frac12}v,IL^{-\frac12}w\bigr\rangle.
$$
Recall that $ILI=L^{-1}$, so $IL^{-\frac12}=L^{\frac12}I$, and $K=IL^{\frac12}=L^{-\frac12}I$. 
Also $(\operatorname{id}\otimes\omega_{\Lambda_{\varphi}(a),
\Lambda_{\varphi}(b)})(W)K\subseteq K(\operatorname{id}\otimes\omega_{\Lambda_{\varphi}(b),
\Lambda_{\varphi}(a)})(W)$, by Proposition~\ref{idomegaWK}.  Putting these all together, 
the above expression becomes
$$
=\bigl\langle L^{-\frac12}I(\operatorname{id}\otimes\omega_{\Lambda_{\varphi}(b),
\Lambda_{\varphi}(a)})(W)v,L^{\frac12}Iw\bigr\rangle
=\bigl\langle I(\operatorname{id}\otimes\omega_{\Lambda_{\varphi}(b),
\Lambda_{\varphi}(a)})(W)v,Iw\bigr\rangle.
$$
As $I^*=I$, $I^2=I$, and $I$ is conjugate linear, this becomes
\begin{align}
&=\overline{\bigl\langle(\operatorname{id}\otimes\omega_{\Lambda_{\varphi}(b),
\Lambda_{\varphi}(a)})(W)v,w\bigr\rangle}
=\bigl\langle w,(\operatorname{id}\otimes\omega_{\Lambda_{\varphi}(b),
\Lambda_{\varphi}(a)})(W)v\bigr\rangle   \notag \\
&=\bigl\langle(\operatorname{id}\otimes\omega_{\Lambda_{\varphi}(a),
\Lambda_{\varphi}(b)})(W^*)w,v\bigr\rangle
=\bigl\langle(\omega_{w,v}\otimes\operatorname{id})(W^*)
\Lambda_{\varphi}(a),\Lambda_{\varphi}(b)\big\rangle.
\notag
\end{align}
Since $a$, $b$ are arbitrary, we conclude that 
$$
(\omega_{IL^{\frac12}v,IL^{-\frac12}w}\otimes\operatorname{id})(W)
=(\omega_{w,v}\otimes\operatorname{id})(W^*)=(\omega_{v,w}\otimes\operatorname{id})(W)^*.
$$
\end{proof}

We are now ready to prove one remaining main result of the subsection:

\begin{prop}\label{IJWequality}
We have:
$$
(I\otimes J)W(I\otimes J)=W^*.
$$
\end{prop}

\begin{proof}
Let $v\in{\mathcal D}(L^{\frac12})$ and $w\in{\mathcal D}(L^{-\frac12})$.  From Proposition~\ref{inequalityWLNabla^z}, 
we know that $W^*(L^{\frac12}\otimes\nabla^{\frac12})\subseteq(L^{\frac12}\otimes\nabla^{\frac12})W^*$. 
So for any $p,q\in{\mathcal D}(\nabla^{\frac12})$, we have:
\begin{align}
\bigl\langle W^*(v\otimes p),w\otimes\nabla^{\frac12}q\bigr\rangle
&=\bigl\langle W^*(v\otimes p),L^{\frac12}L^{-\frac12}w\otimes\nabla^{\frac12}q\bigr\rangle  \notag \\
&=\bigl\langle(L^{\frac12}\otimes\nabla^{\frac12})W^*(v\otimes p),L^{-\frac12}w\otimes q\bigr\rangle
\notag \\
&=\bigl\langle W^*(L^{\frac12}v\otimes\nabla^{\frac12}p),L^{-\frac12}w\otimes q\bigr\rangle.
\notag
\end{align}
Re-writing, this becomes:
$$
\bigl\langle(\omega_{v,w}\otimes\operatorname{id})(W^*)p,\nabla^{\frac12}q\bigr\rangle
=\bigl\langle(\omega_{L^{\frac12}v,L^{-\frac12}w}\otimes\operatorname{id})(W^*)
\nabla^{\frac12}p,q\bigr\rangle,\quad\forall p,q\in{\mathcal D}(\nabla^{\frac12}).
$$
Or equivalently, $\bigl\langle\nabla^{\frac12}(\omega_{v,w}\otimes\operatorname{id})(W^*)
p,q\bigr\rangle=\bigl\langle(\omega_{L^{\frac12}v,L^{-\frac12}w}\otimes\operatorname{id})(W^*)
\nabla^{\frac12}p,q\bigr\rangle$.   This means that for any $p\in{\mathcal D}(\nabla^{\frac12})$, 
we have $(\omega_{v,w}\otimes\operatorname{id})(W^*)p\in{\mathcal D}(\nabla^{\frac12})$ and that 
$\nabla^{\frac12}(\omega_{v,w}\otimes\operatorname{id})(W^*)p=(\omega_{L^{\frac12}v,L^{-\frac12}w}
\otimes\operatorname{id})(W^*)\nabla^{\frac12}p$.  In other words, we have shown that 
\begin{equation}\label{(propIJWequalityeqn1)}
(\omega_{L^{\frac12}v,L^{-\frac12}w}\otimes\operatorname{id})(W^*)\nabla^{\frac12}
\subseteq\nabla^{\frac12}(\omega_{v,w}\otimes\operatorname{id})(W^*).
\end{equation}
By Lemma~\ref{lemIJW1}, we have
$$
(\omega_{L^{\frac12}v,L^{-\frac12}w}\otimes\operatorname{id})(W^*)
=(\omega_{L^{-\frac12}w,L^{\frac12}v}\otimes\operatorname{id})(W)^*
=(\omega_{Iw,Iv}\otimes\operatorname{id})(W).
$$
So Equation~\eqref{(propIJWequalityeqn1)} now becomes
\begin{equation}\label{(propIJWequalityeqn2)}
(\omega_{Iw,Iv}\otimes\operatorname{id})(W)\nabla^{\frac12}
\subseteq\nabla^{\frac12}(\omega_{v,w}\otimes\operatorname{id})(W^*).
\end{equation}

Meanwhile, since $T=J\nabla^{\frac12}$ (so $JT=\nabla^{\frac12}$), we have:
\begin{align}\label{(propIJWequalityeqn3)}
J(\omega_{w,v}\otimes\operatorname{id})(W^*)J\nabla^{\frac12}
&=J(\omega_{w,v}\otimes\operatorname{id})(W^*)T    \notag \\
&\subseteq JT(\omega_{v,w}\otimes\operatorname{id})(W^*)
=\nabla^{\frac12}(\omega_{v,w}\otimes\operatorname{id})(W^*),
\end{align}
where we used the result of Proposition~\ref{omegaidW*T}.

Compare the two inclusions obtained in \eqref{(propIJWequalityeqn2)} and 
\eqref{(propIJWequalityeqn3)}, where we observe that the right hand sides are exactly same. 
Since $\nabla^{\frac12}$ has dense range, this implies that the two bounded operators 
in the left sides of these inclusions should be same.  That is, 
$$
(\omega_{Iw,Iv}\otimes\operatorname{id})(W)=J(\omega_{w,v}\otimes\operatorname{id})(W^*)J,
$$
true for any $v\in{\mathcal D}(L^{\frac12})$ and $w\in{\mathcal D}(L^{-\frac12})$.  In fact, 
this equality will actually hold for any $v,w\in{\mathcal H}$, because the equation involves 
only the bounded operators and since the domains  ${\mathcal D}(L^{\frac12})$ and 
${\mathcal D}(L^{-\frac12})$ are dense in ${\mathcal H}$.  In other words, 
\begin{equation}\label{(propIJWequalityeqn4)}
(\omega_{Iw,Iv}\otimes\operatorname{id})(W)
=J(\omega_{w,v}\otimes\operatorname{id})(W^*)J,\quad\forall v,w\in{\mathcal H}.
\end{equation}

Finally, let $v,w\in{\mathcal H}$ and $p,q\in{\mathcal H}_{\varphi}$ be arbitrary.  We have: 
\begin{align}
\bigl\langle(I\otimes J)W^*(I\otimes J)(v\otimes p),w\otimes q\bigr\rangle
&=\overline{\bigl\langle W^*(I\otimes J)(v\otimes p),Iw\otimes Jq\bigr\rangle}  \notag \\
&=\overline{\bigl\langle(\omega_{Iv,Iw}\otimes\operatorname{id})(W^*)Jp,Jq\bigr\rangle},
\label{(propIJWequalityeqn5)}
\end{align}
because $I\otimes J$ is conjugate linear.  But, by Equation~\eqref{(propIJWequalityeqn4)}, 
we have:
$$
(\omega_{Iv,Iw}\otimes\operatorname{id})(W^*)Jp
=J(\omega_{v,w}\otimes\operatorname{id})(W)p.
$$
So Equation~\eqref{(propIJWequalityeqn5)} becomes:
\begin{align}
&\bigl\langle(I\otimes J)W^*(I\otimes J)(v\otimes p),w\otimes q\bigr\rangle  
=\overline{\bigl\langle J(\omega_{v,w}\otimes\operatorname{id})(W)p,Jq\bigr\rangle} 
\notag \\
&=\bigl\langle (\omega_{v,w}\otimes\operatorname{id})(W)p,q\bigr\rangle
=\bigl\langle W(v\otimes p),w\otimes q\bigr\rangle, 
\quad \forall v,w\in{\mathcal H}, \forall p,q\in{\mathcal H}_{\varphi}.
\notag
\end{align}
This proves that $(I\otimes J)W^*(I\otimes J)=W$.
\end{proof}

\subsection{Antipode map in terms of its polar decomposition}\label{subsec4.4} 

Now that we have gathered results about the operators $K$ and $T$, including their 
polar decompositions, we are now ready to define our antipode map $S$.  See the 
discussion given in the beginning paragraphs of \S\ref{subsec4.3}.

Analogous to the theory of locally compact quantum groups (\cite{KuVa}, \cite{MNW}, 
\cite{KuVavN}, \cite{VDvN}), the two main ingredients are the {\em unitary antipode\/} and 
the {\em scaling group\/}.  With the technical results obtained in the previous subsection,
which needed some approaches that are more general than the ones used for the 
quantum group theory, the rest of the construction becomes mostly similar to the 
quantum group case.

As a consequence of Proposition~\ref{IJWequality}, we have the following observation:

\begin{prop}\label{unitaryantipodeeqn}
Let $a,b\in{\mathcal T}_{\tilde{\varphi}}$ be arbitrary.  We have:
$$
I(\operatorname{id}\otimes\omega_{\Lambda_{\tilde{\varphi}}(a),\Lambda_{\tilde{\varphi}}(b)})(W)^*I
=(\operatorname{id}\otimes\omega_{\Lambda_{\tilde{\varphi}}
(\sigma^{\tilde{\varphi}}_{-\frac{i}{2}}(b^*)),\Lambda_{\tilde{\varphi}}
(\sigma^{\tilde{\varphi}}_{-\frac{i}{2}}(a^*))})(W).
$$
\end{prop}

\begin{proof}
For any $v,w\in{\mathcal H}$, we have:
\begin{align}
\bigl\langle I(\operatorname{id}\otimes\omega_{\Lambda_{\tilde{\varphi}}(b),\Lambda_{\tilde{\varphi}}(a)})(W^*)Iv,w\bigr\rangle
&=\overline{\bigl\langle (\operatorname{id}\otimes\omega_{\Lambda_{\tilde{\varphi}}(b),\Lambda_{\tilde{\varphi}}(a)})(W^*)Iv,Iw\bigr\rangle}
\notag \\
&=\overline{\bigl\langle W^*(Iv\otimes\Lambda_{\tilde{\varphi}}(b)),Iw\otimes\Lambda_{\tilde{\varphi}}(a)\bigr\rangle}.
\notag
\end{align}
We may write $\Lambda_{\tilde{\varphi}}(a)=J\Lambda_{\tilde{\varphi}}\bigl(\sigma^{\tilde{\varphi}}_{-\frac{i}2}(a^*)\bigr)$, 
and also $\Lambda_{\tilde{\varphi}}(b)=J\Lambda_{\tilde{\varphi}}\bigl(\sigma^{\tilde{\varphi}}_{-\frac{i}2}(b^*)\bigr)$. Then by 
using the fact that $(I\otimes J)W^*(I\otimes J)=W$ (from Proposition~\ref{IJWequality}), the expression above becomes:
\begin{align}
&=\overline{\bigl\langle W^*(I\otimes J)(v\otimes\Lambda_{\tilde{\varphi}}(\sigma^{\tilde{\varphi}}_{-\frac{i}2}(b^*))),
(I\otimes J)(w\otimes\Lambda_{\tilde{\varphi}}(\sigma^{\tilde{\varphi}}_{-\frac{i}2}(a^*)))\bigr\rangle}
\notag \\
&=\bigl\langle (I\otimes J)W^*(I\otimes J)(v\otimes\Lambda_{\tilde{\varphi}}(\sigma^{\tilde{\varphi}}_{-\frac{i}2}(b^*))),
w\otimes\Lambda_{\tilde{\varphi}}(\sigma^{\tilde{\varphi}}_{-\frac{i}2}(a^*))\bigr\rangle
\notag \\
&=\bigl\langle W(v\otimes\Lambda_{\tilde{\varphi}}(\sigma^{\tilde{\varphi}}_{-\frac{i}2}(b^*))),
w\otimes\Lambda_{\tilde{\varphi}}(\sigma^{\tilde{\varphi}}_{-\frac{i}2}(a^*))\bigr\rangle
\notag \\
&=\bigr\langle(\operatorname{id}\otimes\omega_{\Lambda_{\tilde{\varphi}}
(\sigma^{\tilde{\varphi}}_{-\frac{i}{2}}(b^*)),\Lambda_{\tilde{\varphi}}
(\sigma^{\tilde{\varphi}}_{-\frac{i}{2}}(a^*))})(W)v,w\bigr\rangle.
\notag
\end{align}

As $v,w\in{\mathcal H}$ are arbitrary, this shows that 
$$
(\operatorname{id}\otimes\omega_{\Lambda_{\tilde{\varphi}}
(\sigma^{\tilde{\varphi}}_{-\frac{i}{2}}(b^*)),\Lambda_{\tilde{\varphi}}
(\sigma^{\tilde{\varphi}}_{-\frac{i}{2}}(a^*))})(W)=
I(\operatorname{id}\otimes\omega_{\Lambda_{\tilde{\varphi}}(b),\Lambda_{\tilde{\varphi}}(a)})(W^*)I,
$$
which is in turn equal to $I(\operatorname{id}\otimes\omega_{\Lambda_{\tilde{\varphi}}(a),\Lambda_{\tilde{\varphi}}(b)})(W)^*I$.
\end{proof}

Since $A=\pi(A)$ is generated by the elements of the form $(\operatorname{id}\otimes\omega)(W)$, 
Proposition~\ref{unitaryantipodeeqn} assures us that the map $R_A:x\mapsto Ix^*I$ is 
a map from $A$ onto itself.  For this map, it is clear that $R_A(x)^*=R_A(x^*)$, $\forall x\in A$, 
and that  $R_A(x_1x_2)=R_A(x_2)R_A(x_1)$, $\forall x_1,x_2\in A$.  We also have $R_A\circ R_A=R_A$. 
We summarize the result below.

\begin{defn}\label{unitaryantipode}
For $x\in A$, define $R_A(x)$ by $R_A(x):=Ix^*I$.  It is a ${}^*$-anti-isomorphism from $A$ 
onto $A$ and is involutive.  We call $R_A$ the {\em unitary antipode\/}.
\end{defn}

\begin{rem}
We may regard $R_A$ as a ${}^*$-isomorphism from $A$ into $A^{\operatorname{op}}$. 
So the map can be naturally extended to the level of the multiplier algebra $M(A)$. 
If we restrict this extended map to the $C^*$-subalgebras $B$ and $C$, the ``source 
algebra'' and the ``target algebra'', then we can show later that $R_A(B)=C$ and that 
the restriction map of $R_A$ onto $B$ coincides with the map $R_{BC}:B\to C$ that we 
considered as part of the defining data for the separability triple $(E,B,\nu)$.  See 
Proposition~\ref{SrestrictedtoBandC} in Section~\ref{sec5} below.  In this sense, 
we will from now on be a little lazy and denote simply by $R=R_A$ the unitary antipode 
on $A$. 
\end{rem}

Next, consider the result $(L^{it}\otimes\nabla^{it})W=W(L^{it}\otimes\nabla^{it})$, 
true for any $t\in\mathbb{R}$ (see Proposition~\ref{LNablaequality}).  By multiplying 
$(\operatorname{id}\otimes\nabla^{-it})$ from left and $(L^{-it}\otimes\operatorname{id})$ 
from right, we obtain:
$$
(L^{it}\otimes\operatorname{id})W(L^{-it}\otimes\operatorname{id})
=(\operatorname{id}\otimes\nabla^{-it})W(\operatorname{id}\otimes\nabla^{it}).
$$
So for any $\omega\in{\mathcal B}({\mathcal H}_{\varphi})_*$, we have:
\begin{equation}\label{(scalingeqn)}
L^{it}(\operatorname{id}\otimes\omega)(W)L^{-it}=\bigl(\operatorname{id}\otimes
\omega(\nabla^{-it}\,\cdot\,\nabla^{it})\bigr)(W).
\end{equation}

In \S\ref{subsec4.3}, we considered the (norm-continuous) ``scaling group'' $(\tau_t)_{t\in\mathbb{R}}$ 
on ${\mathcal B}({\mathcal H})$, defined by $\tau_t(\,\cdot\,)=L^{it}\,\cdot\,L^{-it}$.  Since 
we know that $A=\pi(A)$ is generated by the elements $(\operatorname{id}\otimes\omega)(W)$, 
Equation~\eqref{(scalingeqn)} suggests the following:

\begin{defn}\label{scalinggroup}
For $x\in A$, define $\tau_t(x):=L^{it}xL^{-it}$, for $t\in\mathbb{R}$.  Then $\tau_t(A)=A$, 
and $(\tau_t)_{t\in\mathbb{R}}$ is a norm-continuous one-parameter group of 
automorphisms on $A$.  It is referred to as the {\em scaling group\/} on $A$.
\end{defn}

\begin{prop}\label{Rtaucommute}
$R$ and $\tau_t$ commute, for all $t\in\mathbb{R}$.  That is, 
$R\circ\tau_t=\tau_t\circ R$. 
\end{prop}

\begin{proof}
Recall that $ILI=L^{-1}$.  Furthermore, $I$ is conjugate linear, so we have $IL^{it}I=L^{it}$. 
Therefore, we have: 
$$
R\bigl(\tau_t(x)\bigr)=I\tau_t(x)^*I=IL^{it}x^*L^{-it}I=L^{it}Ix^*IL^{-it}=\tau_t\bigl(R(x)\bigr).
$$
\end{proof}

Just as in the modular theory, let us consider $\tau_{-\frac{i}{2}}$, which is the analytic generator 
of $(\tau_t)$ at $t=-\frac{i}{2}$.  From Equation~\eqref{(scalingeqn)}, it is easy to conclude that 
the elements of the form $(\operatorname{id}\otimes\omega)(W)$, $\omega\in{\mathcal B}
({\mathcal H}_{\varphi})_*$, are contained in ${\mathcal D}(\tau_{-\frac{i}{2}})$. So $\tau_{-\frac{i}{2}}$ 
is densely defined.  Using this, we can give our definition of the antipode, in terms of its polar 
decomposition.  Observe the resemblance with the quantum group case.

\begin{defn}\label{theantipode}
Define the {\em antipode map\/} $S$ on $A$, as follows:
$$
S:=R\circ\tau_{-\frac{i}{2}}.
$$

Furthermore, with the characterization of $M=\pi_{\varphi}(A)''$ given in Equation~\eqref{(idomegaWvNclosureA)} 
and Proposition~\ref{vNqgroupoid}, we may also define $\tilde{S}$, an extension of $S$ to $M$, 
by $\tilde{S}:=\tilde{R}\circ\tilde{\tau}_{-\frac{i}{2}}$,
where $\tilde{R}:x\mapsto Ix^*I$ is a ${}^*$-anti-automorphism of $M$ such that $\tilde{R}\circ\pi_{\varphi}
=\pi_{\varphi}\circ R$, whereas $(\tilde{\tau}_t)_{t\in\mathbb{R}}$ is a (strongly continuous) one-parameter 
group of automorphisms  on $M$ given by $\tilde{\tau}_t:x\mapsto L^{it}xL^{-it}$, such that $\tilde{\tau}_t
\circ\pi_{\varphi}=\pi_{\varphi}\circ\tau_t$.  It is clear that $\tilde{S}\circ\pi_{\varphi}=\pi_{\varphi}
\circ S$.
\end{defn}

Here is an immediate consequence of the definition (compare with Proposition~5.22 of \cite{KuVa}):

\begin{prop}\label{antipodeprop}
Let $S$ be the antipode on $A$, as defined in Definition~\ref{theantipode}.  Then $S$ satisfies 
the following properties:
\begin{enumerate}
  \item $S=R\circ\tau_{-\frac{i}{2}}=\tau_{-\frac{i}{2}}\circ R$.
  \item $S$ is a closed densely-defined map having a dense range.
  \item $S$ is injective and $S^{-1}=R\circ\tau_{\frac{i}{2}}=\tau_{\frac{i}{2}}\circ R$.  We also have 
$S^{-1}$ injective with dense range.
  \item $S$ is antimultiplicative: For $x,y\in{\mathcal D}(S)$, we have $xy\in{\mathcal D}(S)$ and 
$S(xy)=S(y)S(x)$.
  \item For all $x\in{\mathcal D}(S)$, we have $S(x)^*\in{\mathcal D}(S)$ and $S\bigl(S(x)^*\bigr)^*=x$.
  \item $S^2=\tau_{-i}$.  In particular, $S^2\ne\operatorname{Id}$ in general.
  \item $RS=SR$.
  \item $S\circ\tau_t=\tau_t\circ S$ for all $t\in\mathbb{R}$. 
\end{enumerate}
Similar results hold true at the von Neumann algebra level, 
for $\tilde{S}=\tilde{R}\circ\tilde{\tau}_{-\frac{i}{2}}$.
\end{prop}

\begin{proof}
(1). This is a consequence of Proposition~\ref{Rtaucommute}.

(2). True because that is the case for $\tau_{-\frac{i}{2}}$.

(3). $S$ is injective because $R$ and $\tau_{-\frac{i}{2}}$ are injective.  The characterization 
for $S^{-1}$ is easy to see.

(4). True because $R$ is an anti-isomorphism.

(5). If $x\in{\mathcal D}(S)$, consider $S(x)^*=(L^{\frac12}Ix^*IL^{-\frac12})^*=L^{-\frac12}IxIL^{\frac12}$, 
which is contained in $D(\tau_{-\frac{i}{2}})={\mathcal D}(S)$.  Moreover, 
$$
S\bigl(S(x)^*\bigr)^*=S(L^{-\frac12}IxIL^{\frac12})^*
=\bigl(I(L^{\frac12}L^{-\frac12}IxIL^{\frac12}L^{-\frac12})^*I\bigr)^*=(x^*)^*=x.
$$

(6), (7), (8). These are easy to verify, using (1).
\end{proof}

Here is a useful characterization of the antipode, again analogous to the quantum group case:

\begin{prop}\label{antipodeidW}
For any $\omega\in{\mathcal B}({\mathcal H}_{\varphi})_*$, we have: $(\operatorname{id}
\otimes\omega)(W)\in{\mathcal D}(S)$, and
$$
S\bigl((\operatorname{id}\otimes\omega)(W)\bigr)=(\operatorname{id}\otimes\omega)(W^*).
$$
The space $\bigl\{(\operatorname{id}\otimes\omega)(W):\omega\in{\mathcal B}
({\mathcal H}_{\varphi})_*\bigr\}$ is invariant under the scaling automorphisms $\tau_t$, 
and it forms a core for $S$.
\end{prop}

\begin{proof}
(1). We know from Equation~\eqref{(scalingeqn)} that the space $\bigl\{(\operatorname{id}
\otimes\omega)(W):\omega\in{\mathcal B}({\mathcal H}_{\varphi})_*\bigr\}$, which is dense 
in all of $A$, is contained in ${\mathcal D}(\tau_{-\frac{i}{2}})={\mathcal D}(S)$ and is invariant 
under the $\tau_t$.  It thus forms a core for $S$.

(2). Let $\omega\in{\mathcal B}({\mathcal H}_{\varphi})_*$.  Write 
$x=(\operatorname{id}\otimes\omega)(W)$ and $\tilde{x}
=(\operatorname{id}\otimes\bar{\omega})(W)$.  By Proposition~\ref{idomegaWK}, 
we know that for $\xi\in{\mathcal D}(K)$, we have $\tilde{x}\xi\in{\mathcal D}(K)$ 
and $K\tilde{x}\xi=xK\xi$. 

Let $\xi\in{\mathcal D}(K)={\mathcal D}(L^{\frac12})$ and $\zeta\in{\mathcal D}(K^*)={\mathcal D}(L^{-\frac12})$ 
be arbitrary.  Then 
$\bigl\langle\tau_{-\frac{i}{2}}(x)\zeta,\xi\bigr\rangle
=\langle L^{\frac12}xL^{-\frac12}\zeta,\xi\rangle
=\langle xL^{-\frac12}\zeta,L^{\frac12}\xi\rangle$. 
Using this result, since $I$ is conjugate linear, we have:
$$
\bigl\langle S(x)\zeta,\xi\bigr\rangle=\bigl\langle\tau_{-\frac{i}{2}}(Ix^*I)\zeta,\xi\bigr\rangle
=\langle Ix^*IL^{-\frac12}\zeta,L^{\frac12}\xi\rangle
=\overline{\langle x^*IL^{-\frac12}\zeta,IL^{\frac12}\xi\rangle}.
$$
But $K=IL^{\frac12}$ and $K^*=L^{\frac12}I=IL^{-\frac12}$.  So the above equation becomes:
$$
\bigl\langle S(x)\zeta,\xi\bigr\rangle
=\overline{\langle x^*K^*\zeta,K\xi\rangle}=\overline{\langle K^*\zeta,xK\xi\rangle}
=\overline{\langle K^*\zeta,K\tilde{x}\xi\rangle}=\langle\zeta,\tilde{x}\xi\rangle
=\langle\tilde{x}^*\zeta,\xi\rangle,
$$
where we are using $K\tilde{x}\xi=xK\xi$, and the fact that $K$ is involutive and conjugate linear. 

Since this is true for any $\xi\in{\mathcal D}(K)$, $\zeta\in{\mathcal D}(K^*)$, both being dense 
in ${\mathcal H}$, it follows that $S(x)=\tilde{x}^*$, as elements in $A=\pi(A)\,\bigl(\subseteq{\mathcal B}
({\mathcal H})\bigr)$.  Or,  
$$
S\bigl((\operatorname{id}\otimes\omega)(W)\bigr)
=\bigl((\operatorname{id}\otimes\bar{\omega})(W)\bigr)^*
=(\operatorname{id}\otimes\omega)(W^*).
$$
\end{proof}

The following result is the {\em strong left invariance\/} of the weight $\varphi$, which looks 
familiar from the Kac algebra theory and the locally compact quantum group theory.

\begin{prop}\label{strongleftinvariance}
For $a,b\in{\mathfrak N}_{\varphi}$, we have $(\operatorname{id}\otimes\varphi)
\bigl(\Delta(a^*)(1\otimes b)\bigr)\in{\mathcal D}(S)$, and
$$
S\bigl((\operatorname{id}\otimes\varphi)(\Delta(a^*)(1\otimes b))\bigr)
=(\operatorname{id}\otimes\varphi)\bigl((1\otimes a^*)(\Delta b)\bigr).
$$
Also, $\operatorname{span}\bigl\{(\operatorname{id}\otimes\varphi)
(\Delta(a^*)(1\otimes b)):a,b\in{\mathfrak N}_{\varphi}\bigr\}$ forms a core for $S$.
\end{prop}

\begin{proof}
Recall the result of Proposition~\ref{idomegaW}.  We have: 
$$
(\operatorname{id}\otimes\varphi)\bigl(\Delta(a^*)(1\otimes b)\bigr)
=(\operatorname{id}\otimes\omega_{\Lambda_{\varphi}(b),\Lambda_{\varphi}(a)})(W).
$$
Such elements span a dense core for $S$.  By Proposition~\ref{antipodeidW}, we have: 
\begin{align}
S\bigl((\operatorname{id}\otimes\varphi)(\Delta(a^*)(1\otimes b))\bigr)
&=(\operatorname{id}\otimes\omega_{\Lambda_{\varphi}(b),\Lambda_{\varphi}(a)})(W^*)
=\bigl((\operatorname{id}\otimes\omega_{\Lambda_{\varphi}(a),\Lambda_{\varphi}(b)})(W)\bigr)^*
\notag \\
&=\bigl[(\operatorname{id}\otimes\varphi)(\Delta(b^*)(1\otimes a))\bigr]^*
=(\operatorname{id}\otimes\varphi)\bigl((1\otimes a^*)(\Delta b)\bigr).
\notag
\end{align}
\end{proof}

\section{Properties of the antipode map}\label{sec5}

In this section, we will gather various properties involving the antipode map. 
We will first consider some results relating the modular group $(\sigma^{\varphi}_t)$ 
and the scaling group $(\tau_t)$, with the maps $R$, $S$, $\Delta$.  For convenience, 
let us from now on denote $\sigma^{\varphi}_t$ by $\sigma_t$.

\subsection{Some relations between $(\sigma_t)$, $(\tau_t)$ and the maps $R$, $S$, $\Delta$}

\begin{prop}\label{Deltasigma_t}
\begin{enumerate}
  \item For all $x\in M=\pi_{\varphi}(A)''$ and $t\in\mathbb{R}$, we have:
$\tilde{\Delta}\bigl(\sigma^{\tilde{\varphi}}_t(x)\bigr)=(\tilde{\tau}_t\otimes
\sigma^{\tilde{\varphi}}_t)(\tilde{\Delta}x)$.
  \item Similarly, for $x\in A$ and $t\in\mathbb{R}$, we have:
$\Delta\bigl(\sigma_t(x)\bigr)=(\tau_t\otimes\sigma_t)(\Delta x)$, where 
$\sigma_t$ and $\tau_t$ are naturally extended to $M(A)$.
\end{enumerate}
\end{prop}

\begin{proof}
By definition, we know that $\sigma^{\tilde{\varphi}}_t(x)=\nabla^{it}x\nabla^{-it}$, 
$x\in M$.  So we have:
\begin{align}
\tilde{\Delta}\bigl(\sigma^{\tilde{\varphi}}_t(x)\bigr)
&=\tilde{\Delta}(\nabla^{it}x\nabla^{-it})=W^*(1\otimes\nabla^{it}x\nabla^{-it})W
\notag \\
&=W^*(L^{it}\otimes\nabla^{it})(1\otimes x)(L^{-it}\otimes\nabla^{-it})W   \notag \\
&=(L^{it}\otimes\nabla^{it})W^*(1\otimes x)W(L^{-it}\otimes\nabla^{-it})
=(\tilde{\tau}_t\otimes\sigma^{\tilde{\varphi}}_t)(\tilde{\Delta}x),
\notag
\end{align}
by Proposition~\ref{LNablaequality} and by definition of $\tilde{\tau}_t$.  This proves 
the statement at the von Neumann algebra level.  

Meanwhile, since $\sigma^{\tilde{\varphi}}_t\circ\pi_{\varphi}=\pi_{\varphi}\circ\sigma_t$, 
$\tilde{\tau}_t\circ\pi_{\varphi}=\pi_{\varphi}\circ\tau_t$, and $\tilde{\Delta}
\circ\pi_{\varphi}=(\pi_{\varphi}\otimes\pi_{\varphi})\circ\Delta$, it follows 
immediately that $\Delta\bigl(\sigma_t(x)\bigr)=(\tau_t\otimes\sigma_t)(\Delta x)$, 
for $x\in A$.
\end{proof}

In addition to giving a useful relation, the above proposition actually characterizes the 
scaling group $(\tau_t)$, in the sense below:

\begin{cor}
The equation $\Delta\bigl(\sigma_t(x)\bigr)=(\tau_t\otimes\sigma_t)(\Delta x)$, $\forall x\in A$, 
$\forall t\in\mathbb{R}$, uniquely characterizes $(\tau_t)$.
\end{cor}

\begin{proof}
By the fullness of $\Delta$, the elements of the form $(\operatorname{id}\otimes\omega)
(\Delta x)$, for $x\in A$, $\omega\in A^*$, span a dense subspace in $A$.  But note that 
$$
\tau_t\bigl((\operatorname{id}\otimes\omega)(\Delta x)\bigr)
=(\operatorname{id}\otimes\omega\circ\sigma_{-t})\bigl((\tau_t\otimes\sigma_t)(\Delta x)\bigr)
=(\operatorname{id}\otimes\omega\circ\sigma_{-t})\bigl(\Delta(\sigma_t(x))\bigr).
$$
This shows that $\tau_t$ is completely determined by $(\sigma_t)$ and $\Delta$.
\end{proof}

In particular, we observe from the Corollary that $(\tau_t)$ does not explicitly depend on 
the choice of $\psi$.  In the below is another result concerning $(\tau_t)$.

\begin{prop}\label{Deltatau_t}
\begin{enumerate}
  \item For all $x\in A$ and $t\in\mathbb{R}$, we have: $$\Delta\bigl(\tau_t(x)\bigr)
=(\tau_t\otimes\tau_t)(\Delta x).$$
  \item For $x\in M=\pi_{\varphi}(A)''$ and $t\in\mathbb{R}$, we have: $\tilde{\Delta}
\bigl(\tilde{\tau}_t(x)\bigr)=(\tilde{\tau}_t\otimes\tilde{\tau}_t)(\tilde{\Delta} x)$.
\end{enumerate}
\end{prop}

\begin{proof}
For $a\in A$, by the coassociativity of $\Delta$ and by the result of 
Proposition~\ref{Deltasigma_t}, we have:
$$
(\Delta\otimes\operatorname{id})\Delta\bigl(\sigma_t(a)\bigr)
=(\operatorname{id}\otimes\Delta)\Delta\bigl(\sigma_t(a)\bigr)
=(\tau_t\otimes\tau_t\otimes\sigma_t)\bigl((\operatorname{id}\otimes\Delta)(\Delta a)\bigr).
$$
Again using Proposition~\ref{Deltasigma_t} (in the left) and the coassociativity of $\Delta$ 
(in the right), this can be written as follows:
$$
\bigl((\Delta\circ\tau_t)\otimes\sigma_t)(\Delta a)
=(\tau_t\otimes\tau_t\otimes\sigma_t)\bigl((\Delta\otimes\operatorname{id})(\Delta a)\bigr).
$$
Apply here $\operatorname{id}\otimes\operatorname{id}\otimes\sigma_{-t}$, to both sides. 
Then it becomes:
$$
\bigl((\Delta\circ\tau_t)\otimes\operatorname{id}\bigr)(\Delta a)
=\bigl((\tau_t\otimes\tau_t)\circ\Delta\otimes\operatorname{id}\bigr)(\Delta a).
$$
Apply now $\operatorname{id}\otimes\operatorname{id}\otimes\omega$, for an arbitrary 
$\omega\in A^*$.  Then we have:
$$
\Delta\bigl(\tau_t((\operatorname{id}\otimes\omega)(\Delta a))\bigr)
=(\tau_t\otimes\tau_t)\bigl(\Delta((\operatorname{id}\otimes\omega)(\Delta a))\bigr).
$$
Recall that the elements of the form $(\operatorname{id}\otimes\omega)(\Delta a)$, 
for $a\in A$, $\omega\in A^*$, generate $A$.  So we have: $\Delta\bigl(\tau_t(x)\bigr)
=(\tau_t\otimes\tau_t)(\Delta x)$, for any $x\in A$. 

At the same time, since the elements $(\operatorname{id}\otimes\omega)(\Delta a)$, 
$a\in A$, $\omega\in A^*$, also generate $M$, we see that 
$\tilde{\Delta}\bigl(\tilde{\tau}_t(x)\bigr)=(\tilde{\tau}_t\otimes\tilde{\tau}_t)
(\tilde{\Delta} x)$, for $x\in M$.
\end{proof}

We next wish to explore the relationship concerning the antipode $S$ and the comultiplication 
$\Delta$.  This is somewhat tricky, as $S$ is an unbounded map.  Nevertheless, we may 
use an argument similar to the one in the quantum group case (for instance, see Lemma~5.25 
and Proposition~5.26 of \cite{KuVa}).

\begin{lem}\label{lemma_SDelta}
Let $\rho,\theta\in A^*$ be such that $\rho\circ S$ and $\theta\circ S$ are bounded, which 
can be naturally extended to the maps on $M(A)$.  Then for all $x\in{\mathcal D}(S)$, 
we have:
$$
\left(\rho\circ S\,\otimes\,\theta\circ S\right)(\Delta x)
=(\theta\otimes\rho)\bigl(\Delta(S(x))\bigr).
$$
\end{lem}

\begin{proof}
It suffices to consider $x=(\operatorname{id}\otimes\omega)(W)$ for $\omega
\in{\mathcal B}({\mathcal H}_{\varphi})_*$, as such elements form a core for $S$. 
By the same argument given in Proposition~\ref{Deltaidomega}, we have: 
$\Delta x=(\operatorname{id}\otimes\operatorname{id}\otimes\omega)(W_{13}W_{23})$.
So we have:
$$
\left(\rho\circ S\,\otimes\,\theta\circ S\right)(\Delta x)
=\left(\rho\circ S\,\otimes\,\theta\circ S\otimes\omega\right)
(W_{13}W_{23})=(\rho\otimes\theta\otimes\omega)(W_{13}^*W_{23}^*),
$$
since $S\bigl((\operatorname{id}\otimes\omega)(W)\bigr)
=(\operatorname{id}\otimes\omega)(W^*)$, by Proposition~\ref{antipodeidW}. 

Meanwhile, again using an argument as in Proposition~\ref{Deltaidomega}, 
together with $W_{12}^*W_{23}^*W_{12}=W_{23}^*W_{13}^*$ from 
Proposition~\ref{Wpentagon_alt}\,(2), we can show that
$$
\Delta\bigl(S(x)\bigr)=\Delta\bigl((\operatorname{id}\otimes\omega)(W^*)\bigr)
=(\operatorname{id}\otimes\operatorname{id}\otimes\omega)(W_{23}^*W_{13}^*).
$$

Combining the two results, we have:
\begin{align}
\left(\rho\circ S\,\otimes\,\theta\circ S\right)(\Delta x)
&=(\rho\otimes\theta\otimes\omega)(W_{13}^*W_{23}^*)  \notag \\
&=(\theta\otimes\rho\otimes\omega)(W_{23}^*W_{13}^*) 
=(\theta\otimes\rho)\bigl(\Delta(S(x))\bigr),
\notag
\end{align}
where we used the flip operation (between leg~1 and leg~2) in the 
second equality.  This proves our claim.
\end{proof}

\begin{prop}\label{DeltaR}
For all $x\in A$, we have: 
$$
(R\otimes R)(\Delta x)=\Delta^{\operatorname{cop}}\bigl(R(x)\bigr).
$$

Similarly, for $x\in M=\pi_{\varphi}(A)''$, we have: $(\tilde{R}\otimes\tilde{R})
(\tilde{\Delta}x)=\tilde{\Delta}^{\operatorname{cop}}(\tilde{R}(x))$.
\end{prop}

\begin{proof}
Suppose $\rho,\theta\in A^*$ are analytic with respect to $\tau$.  This means 
that the map $\mathbb{R}\ni t\mapsto\rho\circ\tau_t$ has an analytic extension 
to a function from $\mathbb{C}$ to $A^*$.  Then $\rho\circ\tau_z$ is bounded 
for any $z\in\mathbb{C}$, and the function $z\mapsto\rho\circ\tau_z$ 
is analytic.  Similar for $z\mapsto\theta\circ\tau_z$ and 
$z\mapsto\rho\circ\tau_z\,\otimes\,\theta\circ\tau_z$.

For $a\in{\mathcal D}(S)$, being analytic with respect to $\tau$, we will have: 
$z\mapsto\left(\rho\circ\tau_z\,\otimes\,\theta\circ\tau_z\right)(\Delta a)$ and 
$z\mapsto(\rho\otimes\theta)\Delta\bigl(\tau_z(a)\bigr)$ both analytic as 
functions from $\mathbb{C}$ to $\mathbb{C}$.  They actually coincide, 
by Proposition~\ref{Deltatau_t}.  Moreover, by Lemma~\ref{lemma_SDelta}, 
we have:
\begin{align}
(\theta\otimes\rho)\bigl((R\otimes R)\Delta(S(a))\bigr)
&=\left(\rho RS\,\otimes\,\theta RS\right)(\Delta a)
=\bigl(\rho\circ\tau_{-\frac{i}{2}}\,\otimes\,\theta\circ\tau_{-\frac{i}{2}}\bigr)
(\Delta a)  \notag \\
&=(\rho\otimes\theta)\Delta\bigl(\tau_{-\frac{i}{2}}(a)\bigr) 
=(\rho\otimes\theta)\bigl(\Delta(R(S(a)))\bigr).
\notag
\end{align}
Since $S$ has a dense range, and by the boundedness of the other maps involved, 
this means that 
\begin{equation}\label{(DeltaReqn1)}
(\theta\otimes\rho)\bigl((R\otimes R)(\Delta x)\bigr)
=(\rho\otimes\theta)\bigl(\Delta(R(x))\bigr),\quad{\text { for all $x\in A$.}}
\end{equation}

So far, Equation~\eqref{(DeltaReqn1)} holds true only for $\rho,\theta\in A^*$ 
that are analytic with respect to $\tau$.  But for any $\rho\in A^*$, one can 
find a norm-bounded sequence $\bigl\{\rho(n)\bigr\}$ of functionals that are analytic 
with respect to $\tau$ and $\rho(n)(a)\to\rho(a)$, for all $a\in A$.  For instance, 
see proof of Proposition~5.26 in \cite{KuVa}, where $\rho(n)(a)
=\frac{n}{\sqrt{\pi}}\int\exp(-n^2t^2)\rho\bigl(\tau_t(a)\bigr)\,dt$.  
By Equation~\eqref{(DeltaReqn1)}, we will have 
$$
\bigl(\theta(n)\otimes\rho(n)\bigr)\bigl((R\otimes R)(\Delta x)\bigr)
=\bigl(\rho(n)\otimes\theta(n)\bigr)\bigl(\Delta(R(x))\bigr).
$$
Letting $n\to\infty$, we have, now for any $\rho,\theta\in A^*$ and $x\in A$,
$$
(\theta\otimes\rho)\bigl((R\otimes R)(\Delta x)\bigr)
=(\rho\otimes\theta)\bigl(\Delta(R(x))\bigr)
=(\theta\otimes\rho)\bigl(\Delta^{\operatorname{cop}}(R(x))\bigr).
$$
Therefore, we have: $(R\otimes R)(\Delta x)
=\Delta^{\operatorname{cop}}(R(x))$, $\forall x\in A$.

Same argument will hold at the von Neumann algebra level.
\end{proof}

\begin{prop}\label{tausigmaR_corollaryE}
As an immediate consequence of Propositions~\ref{Deltasigma_t}, \ref{Deltatau_t}, 
\ref{DeltaR}, we have, for all $t\in\mathbb{R}$, the following results: 
\begin{enumerate}
  \item $(\tau_t\otimes\sigma_t)(E)=E$
  \item $(\tau_t\otimes\tau_t)(E)=E$
  \item $(R\otimes R)(E)=\varsigma E$
\end{enumerate}
\end{prop}

\begin{proof}
The results follow from the earlier propositions, with $E=\Delta(1)$.
\end{proof}

\subsection{An alternative characterization of the antipode map}

While we do have a valid definition of the antipode (Definition~\ref{theantipode}), 
it is true that this definition is based on specific Hilbert space operators $K$ and $L$, 
and so depends on the choice of the weights $\varphi$, $\psi$.  The goal in this section 
is to give an alternative characterization of the antipode that do not explicitly rely on 
the Haar weights.  The formulation presented below (Definition~\ref{D_0} and 
Proposition~\ref{D_0prop}) is analogous to the quantum group case, as given 
in section~1 of \cite{VDvN}.

We need some technical preparation:
Let $B$ be a $C^*$-algebra and consider an index set $I$.  Borrowing from Kustermans--Vaes 
(see Definition~5.27 of \cite{KuVa}), define $MC_I(B)$ as the set of $I$-tuples $(x_i)_{i\in I}$ 
in $M(B)$ such that $(x_i^*x_i)_{i\in I}$ is strictly summable in $M(B)$; and similarly define 
$MR_I(B)$ as the set of $I$-tuples $(x_i)_{i\in I}$ in $M(B)$ such that $(x_ix_i^*)_{i\in I}$ is 
strictly summable in $M(B)$. 

In a sense, elements of $MC_I(B)$ may be thought of as infinite columns, elements of 
$MR_I(B)$ as infinite rows.  It is clear that the ${}^*$-operation provides a bijection between 
the two sets.  These notions will be useful in what follows, as we will occasionally deal with 
strict convergence.  Here are some results:

\begin{lem}\label{ColumnRow}
Consider a $C^*$-algebra $B$ and an index set $I$.  
\begin{enumerate}
 \item Let $(x_i),(y_i)\in MC_I(B)$.  Then 
$(x_i^*y_i)_{i\in I}$ is strictly summable.  Also the net $\bigl\{\sum_{i\in\Gamma}x_i^*y_i\bigr\}_{\Gamma}$ 
over the finite subsets $\Gamma\subset I$ is bounded.
 \item Let $(x_i)\in MR_I(B)$ and $(y_i)\in MC_I(B)$.  Then $(x_iy_i)_{i\in I}$ is strictly summable.  
Also the net $\bigl\{\sum_{i\in\Gamma}x_iy_i\bigr\}_{\Gamma}$ over the finite subsets $\Gamma\subset I$ 
is bounded.
 \item $MR_I(B)$ and $MC_I(B)$ are vector spaces for the componentwise addition and scalar 
multiplication.
\end{enumerate}
\end{lem}

\begin{proof}
See Lemma~5.28, Results~5.29, 5.30 in \cite{KuVa}.
\end{proof}

Also, here is another lemma that will be useful below:

\begin{lem}\label{e_lnet}
We can find a net $(e_l)_{l\in L}$ in $A$ such that 
\begin{itemize}
  \item each $e_l$, $l\in L$, is contained in ${\mathcal D}(\tau_{-\frac{i}{2}})$;
  \item $(e_l)_{l\in L}$ is bounded and converges strictly to 1;
  \item $\bigl(\tau_{-\frac{i}{2}}(e_l)\bigr)_{l\in L}$ is bounded 
and converges strictly to 1.
\end{itemize}
As a consequence, we see that $(e_l)_{l\in L}$ is a net contained in ${\mathcal D}(S)$ 
such that $\bigl(S(e_l)\bigr)_{l\in L}$ converges strictly to 1.
\end{lem}

\begin{proof}
This is similar to the case considered in the proof of Lemma~\ref{aplemma2}
(See also Remark~5.32 in \cite{KuVa}.).  Consider a bounded net 
$(u_l)_{l\in L}$ such that $(u_l)_{l\in L}$ converges strictly to 1.  For each $l\in L$, 
define $e_l:=\frac{1}{\sqrt{\pi}}\int\exp(-t^2)\tau_t(u_l)\,dt$. Then the three properties 
are easy to verify. The last statement is straightforward, because ${\mathcal D}(S)=
{\mathcal D}(\tau_{-\frac{i}{2}})$ and $S=R\circ\tau_{-\frac{i}{2}}$.
\end{proof}

Let us now gather some results concerning the antipode map $S$ and 
the strict topology on the multiplier algebra $M(A)$. 

\begin{prop}\label{antipodestrictA}
\begin{enumerate}
  \item Suppose $a,b\in A$ are such that for every $z\in{\mathcal D}(S)$ we have 
$za\in{\mathcal D}(S)$ and $S(za)=bS(z)$.  Then we have $a\in{\mathcal D}(S)$ 
and $S(a)=b$.
  \item Similarly, if for every $z\in{\mathcal D}(S)$ we have $az\in{\mathcal D}(S)$ and 
$S(az)=S(z)b$, then $a\in{\mathcal D}(S)$ and $S(a)=b$.
\end{enumerate}
\end{prop}

\begin{proof}
By the previous lemma, we can find a net $(e_l)_{l\in L}$ contained in ${\mathcal D}(S)$ 
such that $\bigl(S(e_l)\bigr)_{l\in L}$ converges strictly to 1.  Suppose $a,b\in A$ are 
such that for every $z\in{\mathcal D}(S)$ we have $za\in{\mathcal D}(S)$ and $S(za)=bS(z)$. 
Then $e_la\in{\mathcal D}(S)$, $\forall l\in L$, and $e_la\to a$.  Since $S$ is a closed map 
on $A$, we see that $a\in{\mathcal D}(S)$ and that $S(a)=\lim_{l\in L}S(e_la)=\lim_{l\in L}bS(e_l)=b$. 
The second statement is proved in a similar way. 
\end{proof}

\begin{cor}
$S$ is closed with respect to the strict topology on $A$.
\end{cor}

\begin{proof}
Suppose $a_n\to a$ strictly, where each $a_n\in{\mathcal D}(S)$, and also $S(a_n)\to b$ 
strictly.  Then for any $z\in{\mathcal D}(S)$, we would have: $za_n\to za$, in norm.  Moreover, 
by the strict convergence $S(a_n)\to b$, we also know that $S(za_n)=S(a_n)S(z)
\to bS(z)$, in norm.  Since $S$ is a closed map, it follows that $za\in{\mathcal D}(S)$ 
and that $S(za)=bS(z)$.

In other words, we see that $za\in{\mathcal D}(S)$ for any $z\in{\mathcal D}(S)$ and that 
$S(za)=bS(z)$.  By the Proposition~\ref{antipodestrictA}, we conclude that $a\in{\mathcal D}(S)$ 
and $S(a)=b$, showing that $S$ is closed under the strict topology.
\end{proof}

We now turn to giving an alternative description of the antipode.  First, we define 
a subspace ${\mathcal D}_0$ of $A$.  

\begin{defn}\label{D_0}
Let $x\in A$.  We will say $x\in{\mathcal D}_0$, if there exists 
an element $\tilde{x}\in A$ and an index set $I$ with elements 
$\{p_1,p_2,\dots\}_{i\in I}$, $\{q_1,q_2,\dots\}_{i\in I}$ contained 
in $MR_I(A)$, such that 
\begin{equation}\label{(D0eqn1)}
\sum_{i\in I}\Delta(p_i)(1\otimes q_i^*) {\text { converges strictly to }} 
E(x\otimes1),
\end{equation}
\begin{equation}\label{(D0eqn2)}
\sum_{i\in I}\Delta(q_i)(1\otimes p_i^*) {\text { converges strictly to }} 
E(\tilde{x}\otimes1).
\end{equation}
\end{defn}

One may notice its resemblance to the definition at the Hilbert space level for the domain of 
the operator $K$ (Definition~\ref{D(K)}).  Here is an important observation:

\begin{prop}\label{D_0prop}
Let $x\in{\mathcal D}_0$, and let $\tilde{x}\in A$ be as in Definition~\ref{D_0}. 
Then $x\in{\mathcal D}(S)$ and $S(x)=\tilde{x}^*$.
\end{prop}

\begin{proof}
Suppose $x\in{\mathcal D}_0$ and $\tilde{x}\in A$ be as in the definition, with 
elements $(p_i)_{i\in I}$, $(q_i)_{i\in I}$ in $MR_I(A)$, satisfying Equation~\eqref{(D0eqn1)} 
and Equation~\eqref{(D0eqn2)}. Note that Equation~\eqref{(D0eqn2)} is equivalent to saying
\begin{equation}\label{(D0eqn3)}
\sum_{i\in I}(1\otimes p_i)\Delta(q_i^*) {\text { converges strictly to }} 
(\tilde{x}^*\otimes1)E.
\end{equation}

Let $c,d\in{\mathfrak N}_{\varphi}$.  Then, by Proposition~\ref{strongleftinvariance}, 
for any finite subset $\Gamma\subseteq I$, we will have
$\sum_{i\in\Gamma}(\operatorname{id}\otimes\varphi)\bigl(\Delta(c^*p_i)(1\otimes q_i^*d)\bigr)
\in{\mathcal D}(S)$, 
and
\begin{equation}\label{(D0eqn4)}
S\left(\sum_{i\in\Gamma}(\operatorname{id}\otimes\varphi)
\bigl(\Delta(c^*p_i)(1\otimes q_i^*d)\bigr)\right)
=\sum_{i\in\Gamma}(\operatorname{id}\otimes\varphi)
\bigl((1\otimes c^*p_i)\Delta(q_i^*d)\bigr).
\end{equation}

Meanwhile, by Lemma~\ref{ColumnRow}, the net $\bigl\{\sum_{i\in\Gamma}p_iq_i^*\bigr\}_{\Gamma}$ 
over the finite subsets $\Gamma\subset I$ is bounded.  Since $\Delta$ is a ${}^*$-homomorphism, 
this in turn means that the net $\bigl\{\sum_{i\in\Gamma}\Delta(p_i)(1\otimes q_i^*)\bigr\}_{\Gamma}$ 
over the finite subsets of $I$ is bounded.  Because of the boundedness of the net, 
Equation~\eqref{(D0eqn1)} leads us to the strict convergence:
\begin{align}
\sum_{i\in I}(\operatorname{id}\otimes\varphi)\bigl(\Delta(c^*p_i)(1\otimes q_i^*d)\bigr)
\longrightarrow
&\,(\operatorname{id}\otimes\varphi)\bigl(\Delta(c^*)[E(x\otimes1)](1\otimes d)\bigr) 
\notag \\
&=(\operatorname{id}\otimes\varphi)\bigl(\Delta(c^*)(1\otimes d)\bigr)\,x,
\notag 
\end{align} 
where we used the fact that $\Delta(c^*)E=\Delta(c^*)$.  Similarly, by Equation~\eqref{(D0eqn3)} 
and the boundedness of the net, we have the following strict convergence: 
\begin{align}
\sum_{i\in I}(\operatorname{id}\otimes\varphi)\bigl((1\otimes c^*p_i)\Delta(q_i^*d)\bigr)
\longrightarrow
&\,(\operatorname{id}\otimes\varphi)\bigl((1\otimes c^*)[(\tilde{x}^*\otimes1)E](\Delta d)\bigr)
\notag \\
&=\tilde{x}^*\,(\operatorname{id}\otimes\varphi)\bigl((1\otimes c^*)(\Delta d)\bigr).
\notag
\end{align}

These observations, together with the fact that $S$ is closed under the strict topology 
(recall Proposition~\ref{antipodestrictA}), mean that at the limit, Equation~\eqref{(D0eqn4)} 
becomes:
$$
S\bigl((\operatorname{id}\otimes\varphi)(\Delta(c^*)(1\otimes d))\,x\bigr)
=\tilde{x}^*\,(\operatorname{id}\otimes\varphi)\bigl((1\otimes c^*)(\Delta d)\bigr).
$$

It follows that for all elements $z=(\operatorname{id}\otimes\varphi)
\bigl(\Delta(c^*)(1\otimes d)\bigr)$, which form a core for $S$, by Proposition~\ref{strongleftinvariance} 
we have: 
$$
S(zx)=\tilde{x}^*S(z).
$$
Applying Proposition~\ref{antipodestrictA}, 
we conclude that $x\in{\mathcal D}(S)$ and $S(x)=\tilde{x}^*$.
\end{proof}

From Proposition~\ref{D_0prop}, we see that ${\mathcal D}_0\subseteq{\mathcal D}(S)$ and 
$S_0=S|_{{\mathcal D}_0}$, where $S_0$ is the map given by ${\mathcal D}_0
\ni x\mapsto\tilde{x}^*$.  Meanwhile, by imitating the proof of 
Proposition~\ref{Kdensedomain}, we can also show the following result:

\begin{prop}\label{D_0dense}
For $\omega\in{\mathcal B}({\mathcal H}_{\varphi})_*$, let $x=(\operatorname{id}
\otimes\omega)(W)$ and $\tilde{x}=(\operatorname{id}\otimes\bar{\omega})(W)$. 
Then $x\in{\mathcal D}_0$, with $\tilde{x}$ is as in Definition~\ref{D_0}. 
\end{prop}

\begin{proof}
Without loss of generality, we may assume $\omega=\omega_{\xi,\zeta}$, where 
$\xi,\zeta\in{\mathcal H}_{\varphi}$.  Also let $(e_j)_{j\in J}$ be an orthonormal 
basis for ${\mathcal H}_{\varphi}$.  Define:
$$
p_j:=(\operatorname{id}\otimes\omega_{e_j,\zeta})(W),
\quad {\text { and }}\quad
q_j:=(\operatorname{id}\otimes\omega_{e_j,\xi})(W).
$$
Note that by Lemma~\ref{omega_xizetaLem}, we have: $\sum_{j\in J}p_jp_j^*
=(\operatorname{id}\otimes\omega_{\zeta,\zeta})(WW^*)$, where the net of finite sums 
is bounded and convergent.  So $(p_j)_{j\in J}\in MR_J(A)$.  Similarly, we have 
$(q_j)_{j\in J}\in MR_J(A)$.  

Meanwhile by Proposition~\ref{Deltaidomega}, we have: $\Delta(p_j)
=(\operatorname{id}\otimes\operatorname{id}\otimes\omega_{e_j,\zeta})(W_{13}W_{23})$, 
while $(1\otimes q_j^*)=(\operatorname{id}\otimes\operatorname{id}\otimes\omega_{\xi,e_j})
(W_{23}^*)$.  So by a similar argument as in the proof of Proposition~\ref{Kdensedomain}, 
we have the convergences:
\begin{equation}\label{(D_0denseeqn1)}
\sum_{j\in J}\Delta(p_j)(1\otimes q_j^*)\longrightarrow
E\,(\operatorname{id}\otimes\operatorname{id}\otimes\omega_{\xi,\zeta})(W_{13})
=E(x\otimes1),
\end{equation}
\begin{equation}\label{(D_0denseeqn2)}
\sum_{j\in J}\Delta(q_j)(1\otimes p_j^*)\longrightarrow
E\,(\operatorname{id}\otimes\operatorname{id}\otimes\omega_{\zeta,\xi})(W_{13})
=E(\tilde{x}\otimes1).
\end{equation}
The two convergences are under the strict topology.  To be precise, a modified 
version of Lemma~\ref{omega_xizetaLem}\,(4) is needed, regarding the strict 
topology.  But, that is not fundamentally different.

Comparing Equations~\eqref{(D_0denseeqn1)} and \eqref{(D_0denseeqn2)} with 
the defining equations in Definition~\ref{D_0}, we see that $x\in{\mathcal D}_0$.
\end{proof}

\begin{theorem}
Let $x\in{\mathcal D}_0$, with $\tilde{x}$ be as in Definition~\ref{D_0}. Let 
$S_0$ be the closure of the map given by ${\mathcal D}_0\ni x\mapsto\tilde{x}^*$. 
Then $S_0=S$.
\end{theorem}

\begin{proof}
Following the proof of Proposition~\ref{D_0prop}, we already observed that 
${\mathcal D}_0\subseteq{\mathcal D}(S)$ and $S_0=S|_{{\mathcal D}_0}$. 
On the other hand, from Proposition~\ref{D_0dense}, we saw that ${\mathcal D}_0$ 
contains the elements of the form $(\operatorname{id}\otimes\omega)(W)$, 
for $\omega\in{\mathcal B}({\mathcal H}_{\varphi})_*$.  But we know such elements 
form a core for $S$ (Proposition~\ref{antipodeidW}).  Therefore, we conclude that 
$S_0=S$.
\end{proof}

\begin{rem}
The significance of this latest observation is that the antipode map $S$ can be 
characterized as the map $S_0$ (defined on ${\mathcal D}_0$), which does 
not explicitly refer to the weights $\varphi$ or $\psi$.  Moreover, the pair $R$ 
and $\tau$ are completely determined by the following properties:
\begin{itemize}
  \item $R$ is a ${}^*$-anti-automorphism on $A$;
  \item $\tau=(\tau_t)_{t\in\mathbb{R}}$ is a norm-continuous one-parameter group 
on $A$;
  \item $R$ and $\tau$ commute;
  \item $S=R\circ\tau_{-\frac{i}{2}}$.
\end{itemize}
Indeed, if we know $S$ (like the characterization $S_0$), then from the above 
properties, we will have $S^2=\tau_{-i}$, from which we can determine $\tau$. 
We can then determine $R$ using $R=S\circ\tau_{\frac{i}{2}}$.  What this means 
is that $R$ and $\tau$ also do not explicitly depend on the choice of the weights 
$\varphi$ and $\psi$.
\end{rem}

Since $M=\pi_{\varphi}(A)''$ is generated by the $(\operatorname{id}\otimes\omega)(W)$, 
$\omega\in{\mathcal B}({\mathcal H}_{\varphi})_*$, the same argument as above 
will provide an analogous characterization of $\tilde{S}$.  As above, the pair 
$\tilde{R}$, $\tilde{\tau}$ are completely determined by $\tilde{S}$. 

The result of Proposition~\ref{antipodestrictA} can be further extended 
to the level of $M(A)$.  First, write $\bar{S}=\bar{R}\circ\bar{\tau}_{-\frac{i}{2}}$ to be 
the strict closure of $S$, where $\bar{R}$ is the ${}^*$-anti-automorphism on $M(A)$ 
extending $R$, and $\bar{\tau}_{-\frac{i}{2}}$ is the strict closure of $\tau_{-\frac{i}{2}}$. 
Using the net $(e_l)_{l\in L}$ from Lemma~\ref{e_lnet}, we have the following result.  
The proof is essentially same as in Proposition~\ref{antipodestrictA}.

\begin{prop}\label{antipodestrictMA}
\begin{enumerate}
  \item Suppose $a,b\in M(A)$ are such that for every $z\in{\mathcal D}(S)$ we have 
$za\in{\mathcal D}(S)$ and $S(za)=bS(z)$.  Then we have $a\in{\mathcal D}(\bar{S})$ 
and $S(a)=b$.
  \item Similarly, if for every $z\in{\mathcal D}(S)$ we have $az\in{\mathcal D}(S)$ 
and $S(az)=S(z)b$, then $a\in{\mathcal D}(\bar{S})$ and $S(a)=b$.
\end{enumerate}
\end{prop}

Also for $\bar{S}$, at the level of $M(A)$, we do have a similar characterization 
for its domain, $\overline{\mathcal D}_0$, like in Definition~\ref{D_0}.  Then 
$\bar{S}:x\mapsto\tilde{x}^*$ for $x\in\overline{\mathcal D}_0$.  For the proof, we can 
use Proposition~\ref{antipodestrictMA}.  As above, the maps $\bar{R}$, $\bar{\tau}$ are 
completely determined by $\bar{S}$.

We end this subsection with a result symmetric to Definition~\ref{D_0} 
and Proposition~\ref{D_0prop}:

\begin{prop}\label{D_0symmetric}
Let $y\in A$ be such that there exists an element $\tilde{y}\in A$ 
and an index set $I$ with elements $(p_i)_{i\in I},(q_i)_{i\in I}
\in MR_I(A)$, satisfying
$$
\sum_{i\in I}(p_i\otimes1)\Delta(q_i^*) {\text { converges strictly to }} 
(1\otimes y)E,
$$
$$
\sum_{i\in I}(q_i\otimes1)\Delta(p_i^*) {\text { converges strictly to }} 
(1\otimes\tilde{y})E.
$$
Then $y\in{\mathcal D}(S)$ and $S(y)=\tilde{y}^*$.
\end{prop}

\begin{proof}
Apply $\varsigma\circ(R\otimes R)$ to both sides of the two convergences above, 
where $\varsigma$ is the flip.  Using $\varsigma\circ(R\otimes R)(\Delta x)
=\Delta\bigl(R(x)\bigr)$ and $\varsigma\circ(R\otimes R)(E)=E$ 
(see Proposition~\ref{DeltaR} and Proposition~\ref{tausigmaR_corollaryE}), 
we will have:
$$
\sum_{i\in I}\Delta\bigl(R(q_i^*)\bigr)\bigl(1\otimes R(p_i)\bigr) 
{\text { converges strictly to }} E\bigl(R(y)\otimes1\bigr),
$$
$$
\sum_{i\in I}\Delta\bigl(R(p_i^*)\bigr)\bigl(1\otimes R(q_i)\bigr) 
{\text { converges strictly to }} E\bigl(R(\tilde{y})\otimes1\bigr).
$$
Note here that as $R$ is a ${}^*$-anti-automorphism,we have $\bigl(R(p_i)\bigr)_{i\in I}
\in MC_I(A)$ and $\bigl(R(q_i^*)\bigr)_{i\in I}\in MR_I(A)$.

Comparing these observations with Equations~\eqref{(D0eqn1)} and \eqref{(D0eqn2)}, 
and knowing Proposition~\ref{D_0prop}, we see that $R(y)\in{\mathcal D}(S)$ and 
$S\bigl(R(y)\bigr)=R(\tilde{y})^*$.  Since $SR=RS$ and since $R$ preserves 
the ${}^*$-operation, this becomes $R\bigl(S(y)\bigr)=R(\tilde{y}^*)$.  Since 
$R$ is injective, it follows that $y\in{\mathcal D}(S)$ and that $S(y)=\tilde{y}^*$.
\end{proof}

The result will hold true if we let $y,\tilde{y}\in M(A)$ and extend $S$ to $\bar{S}$. 
Similar also at the von Neumann algebra level.  In addition, all the comments 
earlier have analogous results, providing us with another characterization of 
the antipode map, which again does not explicitly involve the weights $\varphi$, 
$\psi$.

\subsection{Strong right invariance of $\psi$ and the operator $V$}

In this subsection, we will obtain results corresponding to 
Proposition~\ref{strongleftinvariance} and 
Proposition~\ref{antipodeidW}, but in terms of the weight $\psi$ 
and the operator $V$ (instead of $\varphi$ and $W$ earlier).

We begin with a result that is a consequence of the right invariance 
of $\psi$. 

\begin{prop}\label{Q_R'rightinvariance_Alt}
Let $c\in{\mathfrak N}_{\psi}$ and $w\in{\mathfrak N}_{\psi}$.  Then
$$
(\psi\otimes\operatorname{id}\otimes\operatorname{id})
\bigl((\Delta\otimes\operatorname{id})(\Delta(c^*)(w\otimes1))\bigr)
=E\,\bigl(1\otimes(\psi\otimes\operatorname{id})(\Delta(c^*)(w\otimes1))\bigr).
$$
\end{prop}

\begin{proof}
Since ${\mathfrak N}_{\varphi}^*{\mathfrak N}_{\psi}$ is dense in 
${\mathfrak N}_{\psi}$, we can assume without loss of generality that 
$c$ is of the form $c=r^*y$, where $r\in{\mathfrak N}_{\varphi}$, 
$y\in{\mathfrak N}_{\psi}$.  Also consider $x\in A$ arbitrary and 
let $s\in{\mathfrak N}_{\varphi}$ be arbitrary. 
$$
(\psi\otimes\operatorname{id}\otimes\varphi)
\bigl((1\otimes x\otimes s^*)(\Delta\otimes\operatorname{id})
(\Delta(c^*)(w\otimes1))\bigr)
=(\psi\otimes\operatorname{id})\bigl((1\otimes x)\Delta(qw)\bigr),
$$
where $q=(\operatorname{id}\otimes\varphi)\bigl((1\otimes s^*)\Delta(c^*)\bigr)
=(\operatorname{id}\otimes\varphi)\bigl((1\otimes s^*)\Delta(y^*r)\bigr)$. 

Here, by Proposition~\ref{Q_R'rightinvariance}, and by the property/definition 
of the $Q_{\rho}$ map given in Proposition~\ref{Q_R'}, the right side becomes:
\begin{align}
&=(\psi\otimes\operatorname{id})\bigl(Q_{\rho}(qw\otimes x)\bigr)
=(\psi\otimes\operatorname{id})\bigl(Q_{\rho}(q\otimes x)(w\otimes1)\bigr)
\notag \\
&=(\psi\otimes\operatorname{id})\bigl((\operatorname{id}\otimes\operatorname{id}
\otimes\varphi)((1\otimes x\otimes s^*)(1\otimes E)\Delta_{13}(y^*r))
(w\otimes1)\bigr)  \notag \\
&=(\psi\otimes\operatorname{id}\otimes\varphi)
\bigl((1\otimes x\otimes s^*)(1\otimes E)\Delta_{13}(c^*)(w\otimes1\otimes1)\bigr).
\notag
\end{align}

With $x\in A$, $s\in{\mathfrak N}_{\varphi}$ being arbitrary, and as $\varphi$ 
is faithful, it follows that
$$
(\psi\otimes\operatorname{id}\otimes\operatorname{id})
\bigl((\Delta\otimes\operatorname{id})(\Delta(c^*)(w\otimes1))\bigr)
=(\psi\otimes\operatorname{id}\otimes\operatorname{id})
\bigl((1\otimes E)\Delta_{13}(c^*)(w\otimes1\otimes1)\bigr),
$$
which is none other than 
$$
(\psi\otimes\operatorname{id}\otimes\operatorname{id})
\bigl((\Delta\otimes\operatorname{id})(\Delta(c^*)(w\otimes1))\bigr)
=E\,\bigl(1\otimes(\psi\otimes\operatorname{id})(\Delta(c^*)(w\otimes1))\bigr).
$$
\end{proof}

The following is the {\em strong right invariance\/} of the weight $\psi$, which is 
analogous to Proposition~\ref{strongleftinvariance} (``strong left invariance'') 
for $\varphi$.  The first part of the proof uses the approach analogous to the 
proof of Proposition~5.5 in \cite{KuVa}, and we also make use of the characterization 
of $S$ we obtained earlier in Proposition~\ref{D_0symmetric}.

\begin{prop}\label{strongrightinvariance}
For $a,b\in{\mathfrak N}_{\psi}$, we have: $(\psi\otimes\operatorname{id})
\bigl((a^*\otimes1)(\Delta b)\bigr)\in{\mathcal D}(S)$, with such elements forming 
a core for $S$, and
$$
S\bigl((\psi\otimes\operatorname{id})((a^*\otimes1)(\Delta b))\bigr)
=(\psi\otimes\operatorname{id})\bigl(\Delta(a^*)(b\otimes1)\bigr).
$$
\end{prop}

\begin{proof}
Let $({\mathcal H}_{\psi},\pi_{\psi},\Lambda_{\psi})$ be the GNS-representation 
for $\psi$, and let $(e_i)_{i\in I}$ be an orthonormal basis for ${\mathcal H}_{\psi}$. 
For $a,b\in{\mathfrak N}_{\psi}$, define:
$$
p_i:=(\theta_i\otimes1)(\Lambda_{\psi}\otimes\operatorname{id})(\Delta a),
\quad {\text { and }}\quad
q_i:=(\theta_i\otimes1)(\Lambda_{\psi}\otimes\operatorname{id})(\Delta b),
$$
where $\theta_i:=\langle\,\cdot\,,e_i\rangle$.  

Then for any finite subset $\Gamma\subset I$, we will have:
\begin{align}
&\sum_{i\in\Gamma}\Delta(q_i^*)(p_i\otimes1)  \notag \\
&=\sum_{i\in\Gamma}(\Lambda_{\psi}
\otimes\operatorname{id}\otimes\operatorname{id})\bigl((\operatorname{id}
\otimes\Delta)(\Delta(b^*))\bigr)(\theta_i^*\theta_i\otimes1\otimes1)(\Lambda_{\psi}
\otimes\operatorname{id}\otimes\operatorname{id})(\Delta a\otimes1)
\notag \\
&=(\Lambda_{\psi}\otimes\operatorname{id}\otimes\operatorname{id})
\bigl((\Delta\otimes\operatorname{id})(\Delta(b^*))\bigr)(P_{\Gamma}\otimes1\otimes1)
(\Lambda_{\psi}\otimes\operatorname{id}\otimes\operatorname{id})(\Delta a\otimes1),
\notag
\end{align}
where $P_{\Gamma}$ denotes the orthogonal projection onto the subspace generated by 
the $(e_i)_{i\in{\Gamma}}$.  From this we see that the net $\bigl\{\sum_{i\in\Gamma}
\Delta(q_i^*)(p_i\otimes1)\bigr\}_{i\in\Gamma}$ over the finite subsets $\Gamma\subset I$ 
is bounded and converges strictly to 
$$
(\psi\otimes\operatorname{id}\otimes\operatorname{id})\bigl((\Delta\otimes\operatorname{id})
(\Delta(b^*)(a\otimes1))\bigr),
$$
which is same as 
$E\,\bigl(1\otimes(\psi\otimes\operatorname{id})(\Delta(b^*)(a\otimes1))\bigr)$, by 
Proposition~\ref{Q_R'rightinvariance_Alt}.

In other words, we have established that 
$$
\sum_{i\in I}\Delta(q_i^*)(p_i\otimes1)
{\text { converges strictly to }}
E\,\bigl(1\otimes(\psi\otimes\operatorname{id})(\Delta(b^*)(a\otimes1))\bigr),
$$
which is equivalent to saying 
\begin{equation}\label{(strongrightinvarianceeqn1)}
\sum_{i\in I}(p_i^*\otimes1)\Delta(q_i)
{\text { converges strictly to }} 
\bigl(1\otimes(\psi\otimes\operatorname{id})((a^*\otimes1)(\Delta b))\bigr)\,E.
\end{equation}
By a similar argument, we can also show the following:
\begin{equation}\label{(strongrightinvarianceeqn2)}
\sum_{i\in I}(q_i^*\otimes1)\Delta(p_i)
{\text { converges strictly to }} 
\bigl(1\otimes(\psi\otimes\operatorname{id})((b^*\otimes1)(\Delta a))\bigr)\,E.
\end{equation}

Write $y=(\psi\otimes\operatorname{id})\bigl((a^*\otimes1)(\Delta b)\bigr)$ 
and $\tilde{y}=(\psi\otimes\operatorname{id})\bigl((b^*\otimes1)(\Delta a)\bigr)$. 
By Proposition~\ref{D_0symmetric}, the convergences given in 
Equations~\eqref{(strongrightinvarianceeqn1)} and \eqref{(strongrightinvarianceeqn2)} 
show that we have $y=(\psi\otimes\operatorname{id})\bigl((a^*\otimes1)(\Delta b)\bigr)
\in{\mathcal D}(S)$ and that $S(y)=\tilde{y}^*$.  Or, 
$$
S\bigl((\psi\otimes\operatorname{id})((a^*\otimes1)(\Delta b))\bigr)
=(\psi\otimes\operatorname{id})\bigl((b^*\otimes1)(\Delta a)\bigr)^*
=(\psi\otimes\operatorname{id})\bigl(\Delta(a^*)(b\otimes1)\bigr).
$$
\end{proof}

Now that we have shown the ``right'' analogue of Proposition~\ref{strongleftinvariance}, 
we will next show the analogue of Proposition~\ref{antipodeidW}.  We first observe the 
following proposition, which is itself providing analogues of Proposition~\ref{idomegaW} 
and Propositions~\ref{idomegaWclosureA}, \ref{DeltaW}.

\begin{prop}\label{omegaidV}
\begin{enumerate}
  \item For $a,b\in{\mathfrak N}_{\psi}$, we have:
$$
(\omega_{\Lambda_{\psi}(b),\Lambda_{\psi}(a)}\otimes\operatorname{id})(V)
=(\psi\otimes\operatorname{id})\bigl((a^*\otimes1)(\Delta b)\bigr),
$$
regarded as elements in $A=\pi(A)$.
  \item We have the following characterization of our $C^*$-algebra $A$:
$$
A=\overline{\bigl\{(\omega\otimes\operatorname{id})(V):\omega
\in{\mathcal B}({\mathcal H}_{\psi})_*\bigr\}}^{\|\ \|}.
$$
  \item We also have: $(\pi_{\psi}\otimes\pi)(\Delta x)
=V\bigl(\pi_{\psi}(x)\otimes1\bigr)V^*$, for all $x\in A$.
\end{enumerate}
\end{prop}

\begin{proof}
(1). The proof is done in essentially the same way as in Proposition~\ref{idomegaW}. 
Use the fact that $V\bigl(\Lambda_{\psi}(c)\otimes\Lambda(d)\bigr)
=(\Lambda_{\psi}\otimes\Lambda)\bigl((\Delta c)(1\otimes d)\bigr)$.

(2), (3). Again, imitate the proofs of Propositions~\ref{idomegaWclosureA} 
and \ref{DeltaW}.
\end{proof}

\begin{prop}\label{antipodeidV}
For any $\omega\in{\mathcal B}({\mathcal H}_{\psi})_*$, we have: 
$(\omega\otimes\operatorname{id})(V)\in{\mathcal D}(S)$, and 
$$
S\bigl((\omega\otimes\operatorname{id})(V)\bigr)=(\omega\otimes\operatorname{id})(V^*).
$$
\end{prop}

\begin{proof}
Without loss of generality, we may assume that $\omega=\omega_{\Lambda_{\psi}(b),
\Lambda_{\psi}(a)}$, for $a,b\in{\mathfrak N}_{\psi}$.  Any $\omega\in{\mathcal B}
({\mathcal H}_{\psi})_*$ can be approximated by the $\omega_{\Lambda_{\psi}(b),
\Lambda_{\psi}(a)}$.

By Proposition~\ref{omegaidV}\,(1) and the strong right invariance property of $\psi$ 
(Proposition~\ref{strongrightinvariance}), we have:
$$
S\bigl((\omega_{\Lambda_{\psi}(b),\Lambda_{\psi}(a)}\otimes\operatorname{id})(V)\bigr)
=S\bigl((\psi\otimes\operatorname{id})((a^*\otimes1)(\Delta b))\bigr)  \notag \\
=(\psi\otimes\operatorname{id})\bigl(\Delta(a^*)(b\otimes1)\bigr). 
$$
Meanwhile, we have:
$$
(\omega_{\Lambda_{\psi}(b),\Lambda_{\psi}(a)}\otimes\operatorname{id})(V^*)
=(\omega_{\Lambda_{\psi}(a),\Lambda_{\psi}(b)}\otimes\operatorname{id})(V)^*
=(\psi\otimes\operatorname{id})\bigl((b^*\otimes1)\Delta(a)\bigr)^*,
$$
which is same as $(\psi\otimes\operatorname{id})\bigl(\Delta(a^*)(b\otimes1)\bigr)$. 
This proves the result.
\end{proof}

Since the space $\bigl\{(\omega\otimes\operatorname{id})(V):\omega\in{\mathcal B}
({\mathcal H}_{\psi})_*\bigr\}$ is dense in $A$ (Proposition~\ref{omegaidV}), it would 
form a core for $S$.

\subsection{The operator $\nabla_{\psi}$ and the one-parameter group $(\sigma^{\psi}_t)$}

Using the right Haar weight $\psi$, we can construct the ``right'' analogues of $\nabla$ 
and $(\sigma_t)$, satisfying the properties that are similar to the results obtained earlier.

Here is a result that corresponds to Lemma~\ref{lemIJW1}.

\begin{lem}\label{lemIJV}
Let $\xi\in{\mathcal D}(L^{\frac12})$ and $\zeta\in{\mathcal D}(L^{-\frac12})$.  Then we have:
$$
(\operatorname{id}\otimes\omega_{\xi,\zeta})(V)^*
=(\operatorname{id}\otimes\omega_{IL^{\frac12}\xi,IL^{-\frac12}\zeta})(V).
$$
\end{lem}

\begin{proof}
Let $a,b\in{\mathfrak N}_{\psi}$ be arbitrary.  Then we have:
$$
\bigl\langle(\operatorname{id}\otimes\omega_{IL^{\frac12}\xi,IL^{-\frac12}\zeta})(V)
\Lambda_{\psi}(a),\Lambda_{\psi}(b)\bigr\rangle
=\bigl\langle(\omega_{\Lambda_{\psi}(a),\Lambda_{\psi}(b)}\otimes\operatorname{id})(V)
IL^{\frac12}\xi,IL^{-\frac12}\zeta\bigr\rangle.
$$
Since $IL^{\frac12}=L^{-\frac12}I$ and $IL^{-\frac12}=L^{\frac12}I$, the right side 
becomes
$$
=\bigl\langle(\omega_{\Lambda_{\psi}(a),\Lambda_{\psi}(b)}\otimes\operatorname{id})(V)
L^{-\frac12}I\xi,L^{\frac12}I\zeta\bigr\rangle
=\bigl\langle L^{\frac12}(\omega_{\Lambda_{\psi}(a),\Lambda_{\psi}(b)}\otimes\operatorname{id})(V)
L^{-\frac12}I\xi,I\zeta\bigr\rangle.
$$
But observe that
\begin{align}
&L^{\frac12}(\omega_{\Lambda_{\psi}(a),\Lambda_{\psi}(b)}\otimes\operatorname{id})(V)
L^{-\frac12}  \notag \\
&=\tau_{-\frac{i}{2}}\bigl((\omega_{\Lambda_{\psi}(a),\Lambda_{\psi}(b)}\otimes
\operatorname{id})(V)\bigr) 
=RS\bigl((\omega_{\Lambda_{\psi}(a),\Lambda_{\psi}(b)}\otimes
\operatorname{id})(V)\bigr)   \notag \\
&=R\bigl((\omega_{\Lambda_{\psi}(a),\Lambda_{\psi}(b)}\otimes
\operatorname{id})(V^*)\bigr) 
=I(\omega_{\Lambda_{\psi}(a),\Lambda_{\psi}(b)}\otimes\operatorname{id})(V^*)^*I,
\notag
\end{align}
where we used Proposition~\ref{antipodeidV} in the third equality.  

Putting all these together, knowing  that $I^2=1$, $I^*=I$, and $I$ is conjugate linear, 
we thus have: 
\begin{align}
&\bigl\langle(\operatorname{id}\otimes\omega_{IL^{\frac12}\xi,IL^{-\frac12}\zeta})(V)
\Lambda_{\psi}(a),\Lambda_{\psi}(b)\bigr\rangle  \notag \\
&=\bigl\langle I(\omega_{\Lambda_{\psi}(a),\Lambda_{\psi}(b)}\otimes\operatorname{id})(V^*)^*
\xi,I\zeta\bigr\rangle=\overline{\bigl\langle\xi,
(\omega_{\Lambda_{\psi}(a),\Lambda_{\psi}(b)}\otimes\operatorname{id})(V^*)\zeta\bigr\rangle}
\notag \\
&=\bigl\langle(\omega_{\Lambda_{\psi}(a),\Lambda_{\psi}(b)}\otimes\operatorname{id})(V^*)
\zeta,\xi\bigr\rangle
=\bigl\langle(\operatorname{id}\otimes\omega_{\zeta,\xi})(V^*)\Lambda_{\psi}(a),
\Lambda_{\psi}(b)\bigr\rangle.
\notag
\end{align}
Since $a,b\in{\mathfrak N}_{\psi}$ are arbitrary, this means that 
$$
(\operatorname{id}\otimes\omega_{IL^{\frac12}\xi,IL^{-\frac12}\zeta})(V)
=(\operatorname{id}\otimes\omega_{\zeta,\xi})(V^*)
=(\operatorname{id}\otimes\omega_{\xi,\zeta})(V)^*.
$$
\end{proof}

We may imitate the discussion given in \S\ref{subsec4.2}, and write $T_R$ to be the closed operator on 
${\mathcal H}_{\psi}$ such that $\Lambda_{\psi}({\mathfrak N}_{\psi}\cap{\mathfrak N}_{\psi}^*)$ 
is a core for $T_R$ and $T_R\Lambda_{\psi}(y)=\Lambda_{\psi}(y^*)$, for 
$y\in{\mathfrak N}_{\psi}\cap{\mathfrak N}_{\psi}^*$.  Its polar decomposition 
is $T_R=J_{\psi}\nabla_{\psi}^{\frac12}$, where $\nabla_{\psi}=T_R^*T_R$.  As before, 
we have $J_{\psi}=J_{\psi}^*$, $J_{\psi}^2=1$, and $J_{\psi}\nabla_{\psi}J_{\psi}=\nabla_{\psi}^{-1}$. 
Since $\psi$ is a KMS weight, its lift $\tilde{\psi}$ to the von Neumann algebra $\pi_{\psi}(A)''$ 
is an n.s.f. weight, with the same GNS-Hilbert space ${\mathcal H}_{\tilde{\psi}}={\mathcal H}_{\psi}$. 
We have the modular automorphism group $(\sigma^{\tilde{\psi}}_t)_{t\in\mathbb{R}}$ at the 
von Neumann algebra level, given by $\sigma^{\tilde{\psi}}_t:x\mapsto\nabla_{\psi}^{it}x
\nabla_{\psi}^{-it}$.  Note that $(\sigma^{\tilde{\psi}}_t)$ leaves the $C^*$-algebra $A$ invariant, 
and the restriction of $(\sigma^{\tilde{\psi}}_t)$ to $A$ is exactly the one-parameter group 
$(\sigma^{\psi}_t)$ associated with the KMS weight $\psi$, or  $\sigma^{\psi}_t
=\sigma^{\tilde{\psi}}_t\circ\pi_{\psi}$.  

With the notation above, we are able to prove the following, which correspond to 
Propositions~\ref{inequalityWLNabla}, \ref{inequalityWLNabla^z}, \ref{IJWequality}.

\begin{prop}
The following results hold:
\begin{enumerate}
  \item $V(\nabla_{\psi}\otimes L^{-1})\subseteq(\nabla_{\psi}\otimes L^{-1})V$
  \item $V(\nabla_{\psi}^z\otimes L^{-z})\subseteq(\nabla_{\psi}^z\otimes L^{-z})V$, for all 
$z\in\mathbb{C}$
  \item $(J_{\psi}\otimes I)V(J_{\psi}\otimes I)=V^*$
\end{enumerate}
\end{prop}

\begin{proof}
(1). Let $y\in{\mathfrak N}_{\psi}\cap{\mathfrak N}_{\psi}^*$.  For any 
$\omega\in{\mathcal B}({\mathcal H})_*$, we have:
$$
(\operatorname{id}\otimes\omega)(V)\Lambda_{\psi}(y)=\Lambda_{\psi}
\bigl((\operatorname{id}\otimes\omega)(\Delta y)\bigr),
$$
by definition of $V$ (Proposition~\ref{Vdefn}).  By the right invariance of $\psi$, we know 
$(\operatorname{id}\otimes\omega)(\Delta y)\in{\mathfrak N}_{\psi}\cap{\mathfrak N}_{\psi}^*$, 
and so we have $(\operatorname{id}\otimes\omega)(V)\Lambda_{\psi}(y)
\in{\mathcal D}(T_R)$.  Then: 
\begin{align}
T_R\bigl((\operatorname{id}\otimes\omega)(V)\Lambda_{\psi}(y)\bigr)
&=T_R\bigl(\Lambda_{\psi}((\operatorname{id}\otimes\omega)(\Delta y))\bigr) 
=\Lambda_{\psi}\bigl((\operatorname{id}\otimes\bar{\omega})(\Delta(y^*))\bigr)  \notag \\
&=(\operatorname{id}\otimes\bar{\omega})(V)\Lambda_{\psi}(y^*)
=(\operatorname{id}\otimes\bar{\omega})(V)T_R\Lambda_{\psi}(y).
\notag
\end{align}
Since $\Lambda_{\psi}({\mathfrak N}_{\psi}\cap{\mathfrak N}_{\psi}^*)$ is a core 
for $T_R$, we just showed that 
\begin{equation}\label{(idomegaVT_R)}
(\operatorname{id}\otimes\bar{\omega})(V)T_R\subseteq
T_R(\operatorname{id}\otimes\omega)(V).
\end{equation}
This is analogous to Proposition~\ref{omegaidW*T} earlier.  Taking adjoints, we also have: 
\begin{equation}\label{(idomegaVT_R*)}
(\operatorname{id}\otimes\bar{\omega})(V^*)T_R^*\subseteq
T_R^*(\operatorname{id}\otimes\omega)(V^*),\quad\forall\omega\in{\mathcal B}({\mathcal H})_*.
\end{equation}

Now let $\xi\in{\mathcal D}(L)$ and $\zeta\in{\mathcal D}(L^{-1})$.  By the inclusion \eqref{(idomegaVT_R*)}, 
we have:
$$
(\operatorname{id}\otimes\omega_{\xi,\zeta})(V^*)\nabla_{\psi}
=(\operatorname{id}\otimes\omega_{\xi,\zeta})(V^*)T_R^*T_R
\subseteq T_R^*(\operatorname{id}\otimes\omega_{\zeta,\xi})(V^*)T_R.
$$
Look at the right hand side.  By Lemma~\ref{lemIJV}, we have
$$
(RHS)=T_R^*(\operatorname{id}\otimes\omega_{\xi,\zeta})(V)^*T_R
=T_R^*(\operatorname{id}\otimes\omega_{IL^{\frac12}\xi,IL^{-\frac12}\zeta})(V)T_R.
$$
Using inclusion \eqref{(idomegaVT_R)}, and again using Lemma~\ref{lemIJV}, this becomes:
\begin{align}
(RHS)&\subseteq
T_R^*T_R(\operatorname{id}\otimes\omega_{IL^{-\frac12}\zeta,IL^{\frac12}\xi})(V)
=\nabla_{\psi}(\operatorname{id}\otimes\omega_{L^{\frac12}I\zeta,L^{-\frac12}I\xi})(V)
\notag \\
&=\nabla_{\psi}(\operatorname{id}\otimes\omega_{IL^{\frac12}L^{\frac12}I\zeta,
IL^{-\frac12}L^{-\frac12}I\xi})(V)^*   \notag \\
&=\nabla_{\psi}(\operatorname{id}\otimes\omega_{L^{-1}\zeta,L\xi})(V)^* 
=\nabla_{\psi}(\operatorname{id}\otimes\omega_{L\xi,L^{-1}\zeta})(V^*).
\notag
\end{align}

Combining the results, for any $\xi\in{\mathcal D}(L)$ and any $\zeta\in{\mathcal D}(L^{-1})$, 
we have:
\begin{equation}\label{(inequalityV*NablaLeqn)}
(\operatorname{id}\otimes\omega_{\xi,\zeta})(V^*)\nabla_{\psi}
\subseteq
\nabla_{\psi}(\operatorname{id}\otimes\omega_{L\xi,L^{-1}\zeta})(V^*).
\end{equation}
This means that if $p\in{\mathcal D}(\nabla_{\psi})$, we have 
$(\operatorname{id}\otimes\omega_{L\xi,L^{-1}\zeta})(V^*)p\in{\mathcal D}
(\nabla_{\psi})$ and that $\nabla_{\psi}(\operatorname{id}\otimes\omega_{L\xi,L^{-1}\zeta})(V^*)p
=(\operatorname{id}\otimes\omega_{\xi,\zeta})(V^*)\nabla_{\psi}p$.

Let $p,q\in{\mathcal D}(\nabla_{\psi})$ and $v,w\in{\mathcal D}(L^{-1})$.  Then, 
by Equation~\eqref{(inequalityV*NablaLeqn)}, we have: 
\begin{align}
\bigl\langle V^*(p\otimes v),\nabla_{\psi}q\otimes L^{-1}w\bigr\rangle
&=\bigl\langle(\operatorname{id}\otimes\omega_{v,L^{-1}w})(V^*)p,\nabla_{\psi}q\bigr\rangle
\notag \\
&=\bigl\langle\nabla_{\psi}(\operatorname{id}\otimes\omega_{LL^{-1}v,L^{-1}w})(V^*)p,q\bigr\rangle
\notag \\
&=\bigl\langle(\operatorname{id}\otimes\omega_{L^{-1}v,w})(V^*)\nabla_{\psi}p,q\bigr\rangle.
\notag
\end{align}

This result can be re-written as:
$$
\bigl\langle(\nabla_{\psi}\otimes L^{-1})V^*(p\otimes v),q\otimes w\bigr\rangle
=\bigl\langle V^*(\nabla_{\psi}\otimes L^{-1})(p\otimes v),q\otimes w\bigr\rangle,
$$
true for any $p,q\in{\mathcal D}(\nabla_{\psi})$ and $v,w\in{\mathcal D}(L^{-1})$.  We thus 
have the inclusion:
$V^*(\nabla_{\psi}\otimes L^{-1})\subseteq(\nabla_{\psi}\otimes L^{-1})V^*$.  Equivalently, 
$V(\nabla_{\psi}\otimes L^{-1})\subseteq(\nabla_{\psi}\otimes L^{-1})V$.

(2). Recall that $V$ is a partial isometry with the projections $VV^*=E$ and $V^*V=G_R$. 
Similar to the case involving $W$ (see Proposition~\ref{LNablaequationAB}, 
Proposition~\ref{LNablafunctcalc}), for all $z\in\mathbb{C}$ we can show the following:
$$
V(\nabla_{\psi}^z\otimes L^{-z})|_{\operatorname{Ran}(G_R)}
=(\nabla_{\psi}^z\otimes L^{-z})|_{\operatorname{Ran}(E)}V,
$$
as operators on $\operatorname{Ran}(G_R)$, and 
$$
V^*(\nabla_{\psi}^z\otimes L^{-z})|_{\operatorname{Ran}(E)}
=(\nabla_{\psi}^z\otimes L^{-z})|_{\operatorname{Ran}(G_R)}V^*,
$$
as operators on $\operatorname{Ran}(E)$.

As a consequence (similar to  the proof of Proposition~\ref{inequalityWLNabla^z}), 
we obtain the following result:
$$
V(\nabla_{\psi}^z\otimes L^{-z})\subseteq(\nabla_{\psi}^z\otimes L^{-z})V,\quad\forall z\in\mathbb{C}.
$$
This is in general not an equality, but if $z=it$ ($t\in\mathbb{R}$), the operators are bounded and we 
do have the equality:
\begin{equation}\label{(VNablaLequality)}
V(\nabla_{\psi}^{it}\otimes L^{-it})=(\nabla_{\psi}^{it}\otimes L^{-it})V,\qquad t\in\mathbb{R}.
\end{equation}

(3). Proof for $(J_{\psi}\otimes I)V(J_{\psi}\otimes I)=V^*$ is similar to the proof of Proposition~\ref{IJWequality}. 
We use $V(\nabla_{\psi}^{\frac12}\otimes L^{-\frac12})\subseteq
(\nabla_{\psi}^{\frac12}\otimes L^{-\frac12})V$, as well as Lemma~\ref{lemIJV}.
\end{proof}

Here is a result that corresponds to Proposition~\ref{Deltasigma_t}

\begin{prop}\label{Deltasigma'_t}
For all $x\in A$ and $t\in\mathbb{R}$, we have:
$$
\Delta\bigl(\sigma^{\psi}_t(x)\bigr)=(\sigma^{\psi}_t\otimes\tau_{-t})(\Delta x).
$$
This uniquely characterizes $(\tau_t)$.
\end{prop}

\begin{proof}
We have, for any $t\in\mathbb{R}$, 
\begin{align}
\Delta\bigl(\sigma^{\psi}_t(x)\bigr)&=\Delta\bigl(\nabla_{\psi}^{it}x\nabla^{-it})
=V(\nabla_{\psi}^{it}x\nabla_{\psi}^{-it}\otimes1)V^*  \notag \\
&=V(\nabla_{\psi}^{it}\otimes L^{-it})(x\otimes1)(\nabla_{\psi}^{-it}\otimes L^{it})V^*  \notag \\
&=(\nabla_{\psi}^{it}\otimes L^{-it})V(x\otimes1)V^*(\nabla_{\psi}^{-it}\otimes L^{it})
=(\sigma^{\psi}_t\otimes\tau_{-t})(\Delta x),
\notag
\end{align}
where we used Equation~\eqref{(VNablaLequality)}

As in Corollary of Proposition~\ref{Deltasigma_t}, this result provides an alternative
characterization of $(\tau_t)$.  Proof is similar.
\end{proof}

\subsection{Antipode map restricted to the base algebra}

We often indicated in Part~I that the $\gamma_B$, $\gamma_C$ maps at the level of 
the subalgebras $B$ and $C$, associated with the canonical idempotent $E$, are 
restrictions of the antipode map $S$.  Finally we will provide the proof.

We begin with a converse result of Equations~\eqref{(DeltaonB)} and \eqref{(DeltaonC)},
giving us useful characterizations of $M(B)$ and $M(C)$, the ``source algebra'' 
and the ``target algebra'' contained in $M(A)$.

\begin{prop}\label{DeltaonBandCconv}
Let $x,y\in M(A)$.  We have:
\begin{enumerate}
  \item $x\in M(C)$ if and only if $\Delta x=(x\otimes1)E=E(x\otimes1)$,
  \item $y\in M(B)$ if and only if $\Delta y=E(1\otimes y)=(1\otimes y)E$.
\end{enumerate}
\end{prop}

\begin{proof}
(1). One implication is already known from Equation~\eqref{(DeltaonC)}.  To prove 
the converse statement, suppose $x\in M(A)$ is such that $\Delta x=(x\otimes1)E
=E(x\otimes1)$.  By Proposition~\ref{Deltasigma_t}, we have:
$$
\Delta\bigl(\sigma_t(x)\bigr)=(\tau_t\otimes\sigma_t)(\Delta x)
=(\tau_t\otimes\sigma_t)\bigl((x\otimes1)E\bigr)=(\tau_t\otimes\sigma_t)\bigl(E(x\otimes1)\bigr).
$$
But we know that $(\tau_t\otimes\sigma_t)(E)=E$ (Proposition~\ref{tausigmaR_corollaryE}) 
and that $\tau_t$, $\sigma_t$ are automorphisms.  So from the above equation, we have:
\begin{equation}\label{(DeltaonBandCconveqn1)}
\Delta\bigl(\sigma_t(x)\bigr)=\bigl(\tau_t(x)\otimes1\bigr)E=E\bigl(\tau_t(x)\otimes1\bigr).
\end{equation}

For $n\in\mathbb{N}$, define $x_n,z_n\in M(A)$ as follows:
$$
x_n=\frac{n}{\sqrt{\pi}}\int\operatorname{exp}(-n^2t^2)\sigma_t(x)\,dt
\quad{\text { and }}\quad
z_n=\frac{n}{\sqrt{\pi}}\int\operatorname{exp}(-n^2t^2)\tau_t(x)\,dt.
$$
Note that $\lim_{n\to\infty}x_n=x$ and $\lim_{n\to\infty}z_n=x$.  By  
Equation~\eqref{(DeltaonBandCconveqn1)}, we have:
\begin{equation}\label{(DeltaonBandCconveqn2)}
\Delta(x_n)=(z_n\otimes1)E=E(z_n\otimes1).
\end{equation}

Since $x_n$ is analytic with respect to $(\sigma_t)$, we have $x_nk\in{\mathfrak M}_{\varphi}$ 
for any $k\in{\mathfrak M}_{\varphi}$ (see Lemma~1.2 in Part~I \cite{BJKVD_qgroupoid1}). 
So we have $\Delta(x_nk)\in\overline{\mathfrak M}_{\operatorname{id}\otimes\varphi}$, and 
we can apply $\operatorname{id}\otimes\varphi$.  Then we have:
$$
(\operatorname{id}\otimes\varphi)\bigl(\Delta(x_nk)\bigr)
=(\operatorname{id}\otimes\varphi)\bigl(\Delta(x_n)(\Delta k)\bigr)
=(\operatorname{id}\otimes\varphi)\bigl((z_n\otimes1)(\Delta k)\bigr),
$$
using Equation~\eqref{(DeltaonBandCconveqn2)}.  Let $n\to\infty$.  Then this becomes
\begin{equation}\label{(DeltaonBandCconveqn3)}
(\operatorname{id}\otimes\varphi)\bigl(\Delta(xk)\bigr)=x(\operatorname{id}\otimes\varphi)(\Delta k),
\quad\forall k\in{\mathfrak M}_{\varphi}.
\end{equation}
Similarly, again using Equation~\eqref{(DeltaonBandCconveqn2)}, we can also show that 
\begin{equation}\label{(DeltaonBandCconveqn4)}
(\operatorname{id}\otimes\varphi)\bigl(\Delta(kx)\bigr)=(\operatorname{id}\otimes\varphi)(\Delta k)x,
\quad\forall k\in{\mathfrak M}_{\varphi}.
\end{equation}

But we saw in Proposition~\ref{idphiBpsiidC} that the elements 
$(\operatorname{id}\otimes\varphi)(\Delta k)$, $k\in{\mathfrak M}_{\varphi}$, are dense in $C$. 
Therefore, Equations~\eqref{(DeltaonBandCconveqn3)} and \eqref{(DeltaonBandCconveqn4)} 
show that $xz\in C$ and $zx\in C$ for all $z\in C$.  It follows that $x\in M(C)$.

(2). Proof for the case of $M(B)$ is similarly done, this time using $\Delta\circ\sigma^{\psi}_t
=(\sigma^{\psi}_t\otimes\tau_{-t})\circ\Delta$ (from Proposition~\ref{Deltasigma'_t}).
\end{proof}

Next, we show how the scaling group $(\tau_t)$ and the anti-isomorphism $R$ act on 
the subalgebras $M(B)$ and $M(C)$.

\begin{prop}\label{tauRonBandC}
We have:
\begin{enumerate}
  \item The scaling group $(\tau_t)$ leaves both $M(B)$ and $M(C)$ invariant.
  \item $R\bigl(M(B)\bigr)=M(C)$ and $R\bigl(M(C)\bigr)=M(B)$.
\end{enumerate}
\end{prop}

\begin{proof}
(1). Let $x\in M(C)$.  Then $\Delta x=(x\otimes1)E=E(x\otimes1)$.  Apply $\tau_t\otimes\tau_t$. 
Since we know $(\tau_t\otimes\tau_t)(E)=E$ from Proposition~\ref{tausigmaR_corollaryE}, 
we obtain: 
$$
(\tau_t\otimes\tau_t)(\Delta x)=\bigl(\tau_t(x)\otimes1\bigr)E
=E\bigl(\tau_t(x)\otimes1\bigr).
$$
But note also that $(\tau_t\otimes\tau_t)(\Delta x)=\Delta\bigl(\tau_t(x)\bigr)$, by 
Proposition~\ref{Deltatau_t}.  In other words, we have:
$$
\Delta\bigl(\tau_t(x)\bigr)=\bigl(\tau_t(x)\otimes1\bigr)E=E\bigl(\tau_t(x)\otimes1\bigr).
$$
It follows from Proposition~\ref{DeltaonBandCconv} that $\tau_t(x)\in M(C)$.

Similarly, we can show that $y\in M(B)$ implies $\tau_t(y)\in M(B)$.

(2). Suppose $x\in M(C)$.  Then $\Delta x=(x\otimes1)E=E(x\otimes1)$.  Apply $R\otimes R$. 
Since $R$ is an anti-isomorphism, we then have:
$$
(R\otimes R)(\Delta x)=(R\otimes R)(E)\bigl(R(x)\otimes1\bigr)=\bigl(R(x)\otimes1\bigr)(R\otimes R)(E).
$$
By Propositions~\ref{DeltaR} and \ref{tausigmaR_corollaryE}, this can be rewritten as: 
$$
\Delta^{\operatorname{cop}}\bigl(R(x)\bigr)=(\varsigma E)\bigl(R(x)\otimes1\bigr)=\bigl(R(x)\otimes1\bigr)
(\varsigma E).
$$
Or, equivalently,
$$
\Delta\bigl(R(x)\bigr)=E\bigl(1\otimes R(x)\bigr)=\bigl(1\otimes R(x)\bigr)E.
$$
By Proposition~\ref{DeltaonBandCconv}, we see that $R(x)\in M(B)$.

Similarly, we can show that $y\in M(B)$ implies $R(y)\in M(C)$.
\end{proof}

Since $S=R\circ\tau_{-\frac{i}{2}}$, we can see immediately from Proposition~\ref{tauRonBandC} 
that $S|_{M(B)}:M(B)\to M(C)$ and $S|_{M(C)}:M(C)\to M(B)$, when restricted to the respective 
domains, of course.  Let us now compare these maps with the $\gamma_B$, $\gamma_C$ maps 
earlier.

Suppose $y\in B$ is such that $y\in{\mathcal D}(\gamma_B)$.  Let $a,b\in{\mathfrak N}_{\psi}$ be arbitrary, 
and consider $z=(\psi\otimes\operatorname{id})\bigl((a^*\otimes1)(\Delta b)\bigr)\in{\mathcal D}(S)$. 
Then, since $\Delta y=(1\otimes y)E$, we can show that 
\begin{align}\label{(yz)}
yz&=(\psi\otimes\operatorname{id})\bigl((a^*\otimes y)(\Delta b)\bigr)
=(\psi\otimes\operatorname{id})\bigl((a^*\otimes 1)(\Delta y)(\Delta b)\bigr)
\notag \\
&=(\psi\otimes\operatorname{id})\bigl((a^*\otimes 1)(\Delta(yb))\bigr),
\end{align}
which is also contained in ${\mathcal D}(S)$, because $yb\in{\mathfrak N}_{\psi}$, $a\in{\mathfrak N}_{\psi}$. 
We have:
\begin{align}
S(yz)&=S\bigl((\psi\otimes\operatorname{id})((a^*\otimes 1)(\Delta(yb)))\bigr)
=(\psi\otimes\operatorname{id})\bigl(\Delta(a^*)(yb\otimes 1)\bigr)    \notag \\
&=(\psi\otimes\operatorname{id})\bigl(\Delta(a^*)(b\otimes\gamma_B(y))\bigr)  \notag \\
&=(\psi\otimes\operatorname{id})\bigl(\Delta(a^*)(b\otimes1)\bigr)\gamma_B(y)
=S(z)\gamma_B(y). 
\notag
\end{align}
In the second and the last equalities, we used the strong right invariance property of $\psi$ 
(Proposition~\ref{strongrightinvariance}).  Third equality is using Lemma~\ref{gammagamma'}, 
saying that $E(y\otimes1)=E\bigl(1\otimes\gamma_B(y)\bigr)$.  The elements 
$(\psi\otimes\operatorname{id})\bigl((a^*\otimes1)(\Delta b)\bigr)$ form a core for $S$, 
so we can say from above that $yz\in{\mathcal D}(S)$ and $S(yz)=S(z)\gamma_B(y)$, 
for all $z\in{\mathcal D}(S)$.  Therefore, by Proposition~\ref{antipodestrictA}, we conclude 
that $y\in{\mathcal D}(S)$ and that $S(y)=\gamma_B(y)$. 

This observation shows that $\gamma_B\subseteq S|_B$.  But $\gamma_B$ is 
a densely-defined map, implying that $\gamma_B=S|_B$.  To explore this further, 
let $y\in B\cap{\mathcal D}(S)$.  Then for any $a,b\in{\mathfrak N}_{\psi}$, we have:
\begin{align}
(\psi\otimes\operatorname{id})\bigl(\Delta(a^*)(yb\otimes 1)\bigr)
&=S\bigl((\psi\otimes\operatorname{id})((a^*\otimes 1)(\Delta(yb)))\bigr)   \notag \\
&=S(yz)=S(z)S(y)
=(\psi\otimes\operatorname{id})\bigl(\Delta(a^*)(b\otimes1)\bigr)S(y)  \notag \\
&=(\psi\otimes\operatorname{id})\bigl(\Delta(a^*)(b\otimes S(y))\bigr),
\notag
\end{align}
where we wrote $z=(\psi\otimes\operatorname{id})\bigl((a^*\otimes1)(\Delta b)\bigr)$ 
and used Equation~\eqref{(yz)}.  We also used the strong right invariance property 
and the anti-multiplicativity of $S$.  Since $b$ is arbitrary and since $\psi$ is faithful, 
it follows that 
$$
\Delta(a^*)(y\otimes 1)=\Delta(a^*)\bigl(1\otimes S(y)\bigr),
$$
true for any $a\in{\mathfrak N}_{\psi}$, a dense subspace of $A$.  By the property of 
$E$ (see Proposition~3.3\,(2) in Part~I \cite{BJKVD_qgroupoid1}), it follows that 
\begin{equation}\label{(S|_B=gamma)}
E(y\otimes 1)=E\bigl(1\otimes S(y)\bigr).
\end{equation}
Equation~\eqref{(S|_B=gamma)} confirms that the restriction of our antipode map $S$ 
onto $B$ is exactly the $\gamma_B$ map earlier, whose image is contained in $C$.

In a similar way, we can show, using now the strong left invariance property of 
$\varphi$ (Proposition~\ref{strongleftinvariance}), that the restriction map $S|_C$ is 
exactly the $\gamma_C$ map earlier, whose image is contained in $B$.  We will further 
have:
\begin{equation}\label{(S|_C=gamma')}
(1\otimes x)E=(S(x)\otimes 1\bigr)E,\quad{\text { for $x\in{\mathcal D}(S)\cap C
={\mathcal D}(\gamma_C)$.}}
\end{equation}

We collect these results in the following proposition, together with some related results 
concerning the map $R$ and the automorphism group $(\tau_t)$.

\begin{prop}\label{SrestrictedtoBandC}
The antipode map $S=R\circ\tau_{-\frac{i}{2}}$, when restricted to the subalgebras 
$B$ and $C$, has the following properties: 
\begin{enumerate}
  \item $S|_B=\gamma_B$, with ${\mathcal D}(S)\cap B={\mathcal D}(\gamma_B)$;
  \item $S|_C=\gamma_C$, with ${\mathcal D}(S)\cap C={\mathcal D}(\gamma_B)$;
  \item For all $t\in\mathbb{R}$, we have: $\tau_t|_B=\sigma^{\nu}_{-t}$ and $\tau_t|_C
=\sigma^{\mu}_t$, where $(\sigma^{\nu}_t)$ and $(\sigma^{\mu}_t)$ are the automorphism 
groups associated with the KMS weights $\nu$ (on $B$) and $\mu$ (on $C$) considered earlier.
  \item $R|_B=R_{BC}:B\to C$ and $R|_C=R_{BC}^{-1}:C\to B$, where $R_{BC}$ is 
the ${}^*$-anti-isomorphism associated with the separability triple $(E,B,\nu)$.
\end{enumerate}
\end{prop}

\begin{proof}
(1), (2). These were shown in the above paragraphs.  Note in particular Equations~\eqref{(S|_B=gamma)} 
and \eqref{(S|_C=gamma')}.

(3). From Proposition~2.9 of Part~I \cite{BJKVD_qgroupoid1}, we know $\sigma^{\mu}_{-i}(c)
=(\gamma_B\circ\gamma_C)(c)$ for $c\in{\mathcal D}(\gamma_B\circ\gamma_C)$, and that 
$\sigma^{\nu}_{-i}(b)=(\gamma_B^{-1}\circ\gamma_C^{-1})(b)$ for $b\in{\mathcal D}
(\gamma_B^{-1}\circ\gamma_C^{-1})$, which is equivalent to $\sigma^{\nu}_{i}(b)
=(\gamma_C\circ\gamma_B)(b)$ for $b\in\operatorname{Ran}
(\gamma_B^{-1}\circ\gamma_C^{-1})={\mathcal D}(\gamma_C\circ\gamma_B)$.  

Meanwhile, we just observed that $\gamma_B$, $\gamma_C$ are restrictions of $S$ onto 
$B$ and $C$, respectively.  So, by Proposition~\ref{antipodeprop}\,(6), we have:
$$
\sigma^{\mu}_{-i}=\gamma_B\circ\gamma_C=S^2|_C=\tau_{-i}|_C
\quad{\text { and }}\quad
\sigma^{\nu}_{i}=\gamma_C\circ\gamma_B=S^2|_B=\tau_{-i}|_B.
$$
In other words, on $C$, we have two one-parameter groups $(\sigma^{\mu}_t)$ and $(\tau_t|_C)$ 
such that $\sigma^{\mu}_{-i}=\tau_{-i}|_C$.  But $\sigma^{\mu}_{-i}$ and $\tau_{-i}$ generate 
$(\sigma^{\mu}_t)$ and $(\tau_t)$, respectively.  This means that on $C$, the two one-parameter 
groups coincide, such that $\tau_t|_C=\sigma^{\mu}_t$ for all $t$. 

Similarly, on $B$, using $\sigma^{\nu}_{i}=S^2|_B=\tau_{-i}|_B$, we show that 
$\tau_t|_B=\sigma^{\nu}_{-t}$, $\forall t$.

(4). Recall that $\gamma_B=R_{BC}\circ\sigma^{\nu}_{\frac{i}{2}}$ (see \S1.1).  Compare this with 
the polar decomposition of the antipode $S=R\circ\tau_{-\frac{i}{2}}$ restricted to $B$, namely, 
$S|_B=R|_B\circ\tau_{-\frac{i}{2}}|_B$.  We now know that $S|_B=\gamma_B$ and 
$\tau_{-\frac{i}{2}}|_B=\sigma^{\nu}_{\frac{i}{2}}$.  It follows immediately that $R|_B=R_{BC}$. 

Proof for $R|_C=R_{BC}^{-1}$ is similar, using $S|_C=\gamma_C$ and $\tau_{-\frac{i}{2}}|_C
=\sigma^{\mu}_{-\frac{i}{2}}$.
\end{proof}

\begin{rem}
In the Remark following Definition~\ref{unitaryantipode}, we discussed our choice to denote 
by $R$ the anti-isomorphism $R_A$, the unitary antipode on $A$.  Here, we see that the restriction 
of $R_A$ to the subalgebra $B$ is $R_{BC}$, which was a part of the defining data for the 
separability triple $(E,B,\nu)$.  Proposition~\ref{SrestrictedtoBandC} assures us that using $R$ 
throughout is all right and not ambiguous.
\end{rem}

We have below a sharper result improving Proposition~\ref{tauRonBandC}, as a corollary to 
Proposition~\ref{SrestrictedtoBandC}.

\begin{cor}
We have:
\begin{enumerate}
  \item $(\tau_t)$ leaves $B$ invariant and $\nu\circ\tau_t|_B=\nu$, for any $t\in\mathbb{R}$. 
  \item $(\tau_t)$ leaves $C$ invariant and $\mu\circ\tau_t|_C=\mu$, for any $t\in\mathbb{R}$. 
  \item $R(B)=C$ and $R(C)=B$.
\end{enumerate}
\end{cor}

\begin{proof}
(1), (2). Since $\tau_t|_B=\sigma^{\nu}_{-t}$ and $\tau_t|_C=\sigma^{\mu}_t$, these are immediate.

(3). This follows from the observation that $R|_B=R_{BC}:B\to C$ and $R|_C=R_{BC}^{-1}:C\to B$.
\end{proof}

Next result is about the automorphism groups $(\sigma_t)$ and $(\sigma^{\psi}_t)$, associated 
with our Haar weights $\varphi$ and $\psi$, respectively, restricted to the base algebras 
$C$ and $B$.

\begin{prop}\label{sigmasigma'onCandB}
We have: $\sigma|_C=\sigma^{\mu}$ and $\sigma^{\psi}|_B=\sigma^{\nu}$.
\end{prop}

\begin{proof}
Consider $s=\sigma_{i}(a)b^*$, for $a,b\in{\mathcal T}_{\varphi}$, the Tomita subalgebra. 
Such elements are dense in ${\mathfrak M}_{\varphi}$, and for any $z\in M(A)$, we have 
$sz\in\overline{\mathfrak M}_{\varphi}$ with $\varphi(sz)=\varphi\bigl(\sigma_{i}(a)b^*z\bigr)
=\varphi\bigl(b^*za\bigr)$.  Similarly, we also have $zs\in\overline{\mathfrak M}_{\varphi}$.

Suppose $x\in{\mathfrak M}_{\mu}\,(\subseteq C)$ is analytic with respect to $\mu$, and 
consider $\sigma^{\mu}_{-i}(x)$, which is also an element in $C\,\bigl(\subseteq M(A)\bigr)$. 
As noted above, we will have $s\sigma^{\mu}_{-i}(x)\in\overline{\mathfrak M}_{\varphi}$. 
By Proposition~\ref{nuphipsi}, we have:
$$
\varphi\bigl(s\sigma^{\mu}_{-i}(x)\bigr)
=\mu\bigl((\operatorname{id}\otimes\varphi)(\Delta(s\sigma^{\mu}_{-i}(x)))\bigr)
=(\mu\otimes\varphi)\bigl(\Delta(s\sigma^{\mu}_{-i}(x))\bigr).
$$
Since $\Delta\bigl(\sigma^{\mu}_{-i}(x)\bigr)=E\bigl(\sigma^{\mu}_{-i}(x)\otimes1\bigr)$ 
and $\Delta x=(x\otimes1)E$, this becomes:
\begin{align}
\varphi\bigl(s\sigma^{\mu}_{-i}(x)\bigr)&=(\mu\otimes\varphi)
\bigl((\Delta s)(\sigma^{\mu}_{-i}(x)\otimes1))\bigr)
=(\mu\otimes\varphi)\bigl((x\otimes1)(\Delta s)\bigr)  \notag \\
&=(\mu\otimes\varphi)\bigl(\Delta(xs)\bigr)=\varphi(xs),
\notag
\end{align}
again by Proposition~\ref{nuphipsi}.

Compare this with the usual characterization of the automorphism group $(\sigma_t)$ for 
$\varphi$, given by $\varphi(xs)=\varphi\bigl(s\sigma_{-i}(x)\bigr)$.  As elements 
$s=\sigma_{i}(a)b^*$ are dense in ${\mathfrak M}_{\varphi}$ and $\varphi$ is faithful, 
we conclude that $x\in{\mathcal D}(\sigma_{-i})$ and that $\sigma_{-i}(x)=\sigma^{\mu}_{-i}(x)$, 
which is true on a dense subset of $C$.  Since $\sigma_{-i}$ and $\sigma^{\mu}_{-i}$ 
generate $(\sigma_t)$ and $(\sigma^{\mu}_t)$ respectively, we see that on $C$, 
we have $\sigma_t|_C=\sigma^{\mu}_t$ for all $t$.

Similarly, for $y\in{\mathfrak M}_{\nu}\,(\subseteq B)$, that is analytic with respect to $\nu$, 
we have, for any $s\in{\mathcal T}_{\psi}$, that 
\begin{align}
\psi\bigl(s\sigma^{\nu}_{-i}(y)\bigr)
&=(\psi\otimes\nu)\bigl((\Delta s)(1\otimes\sigma^{\nu}_{-i}(y))\bigr)
=(\psi\otimes\nu)\bigl((1\otimes y)(\Delta s)\bigr)  \notag \\
&=(\psi\otimes\nu)\bigl(\Delta(ys)\bigr)=\psi(ys).
\notag
\end{align}
Comparing with the characterization $\psi(ys)=\psi\bigl(s\sigma^{\psi}_{-i}(y)\bigr)$, 
we can conclude that on $B$, we have $\sigma^{\psi}_t|_B=\sigma^{\nu}_t$ for all $t$.
\end{proof}

\begin{cor}
As a consequence of the previous proposition, we have:
\begin{enumerate}
  \item $\sigma_t$ leaves $C$ invariant and $\mu\circ\sigma_t|_C=\mu$, 
for any $t\in\mathbb{R}$.
  \item $\sigma^{\psi}_t$ leaves $B$ invariant and $\nu\circ\sigma^{\psi}_t|_B=\nu$, 
for any $t\in\mathbb{R}$.
\end{enumerate}
\end{cor}

\begin{proof}
Since $\sigma_t|_C=\sigma^{\mu}$ and $\sigma'_t|_B=\sigma^{\nu}$, these results 
are immediate.
\end{proof}

Now that we considered the restrictions $\sigma|_C$ and $\sigma^{\psi}|_B$, which were 
respectively shown to coincide with $\sigma^{\mu}$ and $\sigma^{\nu}$, it is natural to 
consider the restriction $\sigma|_B$.  While it does not have to coincide with $\sigma^{\nu}$, 
we still have the following invariance result:  

\begin{prop}\label{sigma_tnuinvariance}
$\sigma_t$ leaves $B$ invariant and $\nu\circ\sigma_t|_B=\nu$, 
for any $t\in\mathbb{R}$.
\end{prop}

\begin{proof}
Suppose $y\in M(B)$.  Then by Proposition~\ref{Deltasigma_t}, we have: 
$$
\Delta\bigl(\sigma_t(y)\bigr)=(\tau_t\otimes\sigma_t)(\Delta y)
=(\tau_t\otimes\sigma_t)\bigl(E(1\otimes y)\bigr)
=E\bigl(1\otimes\sigma_t(y)\bigr),
$$
where we also used $(\tau_t\otimes\sigma_t)(E)=E$ (Proposition~\ref{tausigmaR_corollaryE}). 
But then, by Proposition~\ref{DeltaonBandCconv}, we see that $\sigma_t(y)\in M(B)$. 
It follows that $\sigma_t$ leaves $M(B)$ invariant.

At this stage, recall the last condition given in Definition~\ref{definitionlcqgroupoid},  
where we required the existence of a one-parameter group $(\theta_t)$ of $B$
such that $\theta_t(b)=\sigma_t(b)$, $\forall b\in B$, $\forall t\in\mathbb{R}$, 
and that $\nu\circ\theta_t=\nu$.  We can see immediately that $\sigma|_B\equiv\theta$, 
which in particular means that $(\sigma_t|_B)$ is a one-parameter group of $B$, and 
that $\nu\bigl(\sigma_t(b)\bigr)=\nu\bigl(\theta(b)\bigr)=\nu(b)$, for any $b\in B$ 
and $t\in\mathbb{R}$.
\end{proof}

The above proposition was the first occasion where the last condition of our main definition 
(Definition~\ref{definitionlcqgroupoid}) was used.  We will be using this result to gradually 
make our case that $\nu$ is ``quasi-invariant''.  First, we show the commutativity of $\sigma$ 
and $\tau$:

\begin{prop}\label{sigmataucommute}
We have:
\begin{enumerate}
  \item $\psi$ is invariant under $(\sigma_t\circ\tau_{-t})_{t\in\mathbb{R}}$.
  \item $\sigma$ and $\tau$ commute.  That is, $\sigma_t\circ\tau_s=\tau_s\circ\sigma_t$, $\forall t,s$.
\end{enumerate}
\end{prop}

\begin{proof}
(1). By Proposition~\ref{nuphipsi}, we have:
$$
\psi(\sigma_t\circ\tau_{-t})
=\nu\bigl((\psi\otimes\operatorname{id})(\Delta(\sigma_t\circ\tau_{-t}))\bigr).
$$
But by Propositions~\ref{Deltasigma_t} and \ref{Deltatau_t}, we know that 
\begin{equation}\label{(sigmasigma'commute_eqn1)}
\Delta\bigl(\sigma_t\circ\tau_{-t}\bigr)=(\tau_t\otimes\sigma_t)\bigl(\Delta(\tau_{-t})\bigr)
=(\operatorname{id}\otimes(\sigma_t\circ\tau_{-t}))\circ \Delta.
\end{equation}
Combining, we have:
$$
\psi(\sigma_t\circ\tau_{-t})
=\nu\bigl((\sigma_t\circ\tau_{-t})((\psi\otimes\operatorname{id})\Delta)\bigr)
=\nu\bigl((\psi\otimes\operatorname{id})\Delta\bigr)=\psi,
$$
where we used the fact that $(\psi\otimes\operatorname{id})(\Delta a)\in M(B)$, $\forall a\in{\mathfrak M}_{\psi}$, 
together with the knowledge that $\nu\circ\sigma_t|_B=\nu$ (Proposition~\ref{sigma_tnuinvariance}) 
and $\nu\circ\tau_{-t}|_B=\nu$ (Corollary of Proposition~\ref{SrestrictedtoBandC}).  The 
last equality is using Proposition~\ref{nuphipsi}.

(2). By (1), it follows from the general weight theory (see \cite{Tk2}) that $(\sigma^{\psi}_s)$ and 
$(\sigma_t\circ\tau_{-t})_{t\in\mathbb{R}}$ commute.  In other words, 
\begin{equation}\label{(sigmasigma'commute_eqn2)}
\sigma^{\psi}_s\circ(\sigma_t\circ\tau_{-t})=(\sigma_t\circ\tau_{-t})\circ\sigma^{\psi}_s,\quad
\forall s,\forall t.
\end{equation}
Then we have: 
\begin{align}
\bigl(\operatorname{id}\otimes(\sigma_t\circ\tau_{-t})\bigr)\circ\Delta
&=\Delta(\sigma_t\circ\tau_{-t})=\Delta\bigl(\sigma^{\psi}_{-s}\circ(\sigma_t\circ\tau_{-t})
\circ\sigma^{\psi}_s\bigr)
\notag \\
&=(\sigma^{\psi}_{-s}\otimes\tau_s)\bigl(\Delta((\sigma_t\circ\tau_{-t})\circ\sigma^{\psi}_s)\bigr) 
\notag \\
&=\bigl(\sigma^{\psi}_{-s}\otimes\tau_s\circ(\sigma_t\circ\tau_{-t})\bigr)\bigl(\Delta(\sigma^{\psi}_s)\bigr)
\notag \\
&=\bigl(\operatorname{id}\otimes\tau_s\circ(\sigma_t\circ\tau_{-t})\circ\tau_{-s}\bigr)\circ\Delta.
\notag 
\end{align}
The first and fourth equalities used Equation~\eqref{(sigmasigma'commute_eqn1)}; 
second equality used Equation~\eqref{(sigmasigma'commute_eqn2)}; third and fifth 
equalities used Proposition~\ref{Deltasigma'_t}.  

Since the elements of the form $(\omega\otimes\operatorname{id})(\Delta x)$, $x\in A$, 
$\omega\in A^*$, are dense in $A$, the above result implies that 
$$
\sigma_t\circ\tau_{-t}=\tau_s\circ(\sigma_t\circ\tau_{-t})\circ\tau_{-s}
=\tau_s\circ\sigma_t\circ\tau_{-s}\circ\tau_{-t},
$$
true for all $s$ and $t$.  From this, we have $\sigma_t=\tau_s\circ\sigma_t\circ\tau_{-s}$, 
or equivalently, $\sigma_t\tau_s=\tau_s\sigma_t$, $\forall s,t$.
\end{proof}

Results similar to Propositions~\ref{sigma_tnuinvariance} and \ref{sigmataucommute} 
do exist for $\sigma^{\psi}$ (the invariance of $\mu$ under $\sigma^{\psi}|_C$, 
and the commutativity of $\sigma^{\psi}$ and $\tau$).  But we will consider them later.

\subsection{The right Haar weight $\varphi\circ R$}

Consider our left Haar weight $\varphi$.  Using the involutive ${}^*$-anti-isomorphism 
$R:A\to A$, we can define a faithful weight $\varphi\circ R$.  It is not difficult to see that 
$\varphi\circ R$ is also KMS, with its one-parameter group of automorphisms 
$(\sigma^{\varphi\circ R}_t)_{t\in\mathbb{R}}$ given by 
$$
\sigma^{\varphi\circ R}_t=R\circ\sigma_{-t}\circ R.
$$
Let us further explore this KMS weight $\varphi\circ R$.

\begin{prop}\label{phiR_rightinvariant}
$\varphi\circ R$ is right invariant.
\end{prop}

\begin{proof}
Let $a\in{\mathfrak M}_{\varphi\circ R}$.  Then $R(a)\in{\mathfrak M}_{\varphi}$. 
By the left invariance of $\varphi$, we thus know that $\Delta\bigl(R(a)\bigr)
\in\overline{\mathfrak M}_{\operatorname{id}\otimes\varphi}$, and that 
$(\operatorname{id}\otimes\varphi)\bigl(\Delta(R(a))\bigr)\in M(C)$. 
With $\Delta^{\operatorname{cop}}=\varsigma\circ\Delta$, we have 
$(\operatorname{id}\otimes\varphi)\bigl(\Delta(R(a))\bigr)
=(\varphi\otimes\operatorname{id})\bigl(\Delta^{\operatorname{cop}}(R(a))\bigr)$. 
So by Proposition~\ref{DeltaR}, we have:
$$
(\operatorname{id}\otimes\varphi)\bigl(\Delta(R(a))\bigr)
=(\varphi\otimes\operatorname{id})\bigl((R\otimes R)(\Delta a)\bigr)
=R\bigl((\varphi\circ R\otimes\operatorname{id})(\Delta a)\bigr).
$$
From this, we can see that $\Delta a\in\overline{\mathfrak M}_{\varphi\circ R\otimes
\operatorname{id}}$, and that 
$$(\varphi\circ R\otimes\operatorname{id})(\Delta a)
=R\bigl((\operatorname{id}\otimes\varphi)[\Delta(R(a))]\bigr)\in M(B),
$$
because $R\bigl(M(C)\bigr)=M(B)$ by Proposition~\ref{tauRonBandC}.  It follows that 
$\varphi\circ R$ is right invariant.
\end{proof}

Unlike in the quantum group case, with our base algebra $B$ being nontrivial, 
we cannot expect any uniqueness result for the right Haar weight.  Specifically, 
while both $\psi$ and $\varphi\circ R$ are right invariant, it is not true in general 
that $\varphi\circ R$ is a scalar multiple of $\psi$.  Having said this, it turns out that 
for all practical purposes, replacing $\psi$ by $\varphi\circ R$ gives us essentially 
an equivalent structure for our quantum groupoid.  The remainder of this subsection 
discusses this point.

\begin{prop}\label{DeltasigmaphiR_t}
For all $x\in A$ and $t\in\mathbb{R}$, we have:
$$
\Delta\bigl(\sigma^{\varphi\circ R}_t(x)\bigr)=(\sigma^{\varphi\circ R}_t
\otimes\tau_{-t})(\Delta x).
$$
\end{prop}

\begin{proof}
Since $\sigma^{\varphi\circ R}_t=R\circ\sigma_{-t}\circ R$, we have:
\begin{align}
\Delta\bigl(\sigma^{\varphi\circ R}_t(x)\bigr)
&=\Delta\bigl((R\sigma_{-t}R)(x)\bigr)
=\varsigma\circ(R\otimes R)\bigl(\Delta(\sigma_{-t}(R(x)))\bigr) \notag \\
&=\varsigma\circ(R\tau_{-t}\otimes R\sigma_{-t})\bigl(\Delta(R(x))\bigr) 
=(R\sigma_{-t}\otimes R\tau_{-t})\bigl(\Delta^{\operatorname{cop}}(R(x))\bigr)
\notag \\
&=(R\sigma_{-t}R\otimes R\tau_{-t}R)(\Delta x)
=(\sigma^{\varphi\circ R}_t\otimes\tau_{-t})(\Delta x).
\notag
\end{align}
We used the results of Propositions~\ref{Deltasigma_t} and \ref{DeltaR}, as well as 
the commutativity of $\tau$ and $R$.
\end{proof}

Observe that this result is essentially the same as Proposition~\ref{Deltasigma'_t}. 
See the following proposition for our main claim:

\begin{theorem}\label{propphiR}
The data $(A,\Delta,E,B,\nu,\varphi,\varphi\circ R)$ determines a locally compact 
quantum groupoid of separable type, in the sense of Definition~\ref{definitionlcqgroupoid}.  
Moreover, $R$ and $(\tau_t)$ are still the unitary antipode and the scaling group 
for this new quantum groupoid.
\end{theorem}

\begin{proof}
Since $\varphi\circ R$ is shown to be right invariant, we can see immediately 
that $(A,\Delta,E,B,\nu,\varphi,\varphi\circ R)$ is a locally compact quantum 
groupoid separable type.  We also have the following result analogous to 
Proposition~\ref{nuphipsi}.  Namely, for $a\in{\mathfrak M}_{\varphi\circ R}$, we have: 
\begin{equation}\label{(nuphiphiR)}
\nu\bigl((\varphi\circ R\otimes\operatorname{id})(\Delta a)\bigr)
=(\nu\circ R)\bigl((\operatorname{id}\otimes\varphi)(\Delta(R(a)))\bigr)
=\varphi\bigl(R(a)\bigr).
\end{equation}
First equality is true because $(\varphi\circ R\otimes\operatorname{id})(\Delta a)
=R\bigl((\operatorname{id}\otimes\varphi)(\Delta(R(a)))\bigr)$, observed in the 
proof of Proposition~\ref{phiR_rightinvariant}, and the second equality is using 
$\nu\circ R=\mu$ and Proposition~\ref{nuphipsi}.

Therefore, all the same discussion and construction will go through, including 
the construction of the unitary antipode and the scaling group, with $\psi$ 
replaced by $\varphi\circ R$.

From Definition~\ref{unitaryantipode} and Proposition~\eqref{unitaryantipodeeqn}, 
we know that $R:A\to A$ is defined by
$$
R:(\operatorname{id}\otimes\omega_{\Lambda_{\tilde{\varphi}}(a),
\Lambda_{\tilde{\varphi}}(b)})(W)
\mapsto (\operatorname{id}\otimes\omega_{\Lambda_{\tilde{\varphi}}
(\sigma^{\tilde{\varphi}}_{-\frac{i}{2}}(b^*)),\Lambda_{\tilde{\varphi}}
(\sigma^{\tilde{\varphi}}_{-\frac{i}{2}}(a^*))})(W),
$$
for $a,b\in{\mathcal T}_{\tilde{\varphi}}$.  Since the elements $(\operatorname{id}
\otimes\omega_{\Lambda_{\tilde{\varphi}}(a),\Lambda_{\tilde{\varphi}}(b)})(W)$ 
generates all of $A$, this characterizes the map $R$.  Observe that the definition 
does not explicitly refer to the right Haar weight.

As for the scaling group, instead of Definition~\ref{scalinggroup}, which does 
depend on the right Haar weight, consider an alternative characterization 
noted in Proposition~\ref{Deltasigma'_t}. Comparing it with Proposition~\ref{DeltasigmaphiR_t}, 
we can see that the scaling group determined by having $\varphi\circ R$ in place of $\psi$ 
is exactly $\tau$.
\end{proof}

\begin{rem}
By Theorem~\ref{propphiR}, we see that the data 
$(A,\Delta,E,B,\nu,\varphi,\psi)$ and $(A,\Delta,E,B,\nu,\varphi,\varphi\circ R)$ 
represent essentially equivalent quantum groupoid structure, with same $R$ 
(unitary antipode), same $\tau$ (scaling group), so the same antipode 
map $S$.
\end{rem}

\subsection{Quasi-invariance of $\nu$}

From now on, building on the insight we gained from the previous subsection, 
we will take the right Haar weight for our quantum groupoid to be 
$\psi=\varphi\circ R$.  For convenience, write 
$\sigma'=\sigma^{\varphi\circ R}$.

Let us begin with an invariance result, concerning the restriction of $(\sigma'_t)$ 
to $C$ and the weight $\mu$ on $C$.

\begin{prop}\label{sigma'_tmuinvariance}
$\sigma'_t$ leaves $C$ invariant and $\mu\circ\sigma'_t|_C=\mu$, 
for any $t\in\mathbb{R}$.
\end{prop}

\begin{proof}
Let $c\in C$.  Then $R(c)\in B$.  By Proposition~\ref{sigma_tnuinvariance}, 
we thus have $\sigma_t\bigl(R(c)\bigr)\in B$, $\forall t$.  Since $\sigma'_t
=\sigma^{\varphi\circ R}=R\circ\sigma_{-t}\circ R$, we have, 
$$
\sigma'_t(c)=R\bigl(\sigma_{-t}(R(c))\bigr)\in C,
$$
showing that $\sigma'_t$ leaves $C$ invariant.

Meanwhile, we have (again by Proposition~\ref{sigma_tnuinvariance}), 
$$
\mu\circ\sigma'_t|_C=(\nu\circ R)\circ(R\circ\sigma_{-t}\circ R)|_C
=\nu\circ\sigma_{-t}|_B\circ R|_C=\nu\circ R|_C=\mu.
$$
\end{proof}

The next proposition shows the mutual commutativity of the automorphism groups 
$\sigma$, $\sigma'$, $\tau$.

\begin{prop}\label{sigmasigma'commute}
We have:
\begin{enumerate}
  \item $\sigma$ and $\tau$ commute.
  \item $\sigma'$ and $\tau$ commute.
  \item $\sigma$ and $\sigma'$ commute.
\end{enumerate}
\end{prop}

\begin{proof}
(1). This is none other than Proposition~\ref{sigmataucommute}\,(2).

(2). By modifying the proof of Proposition~\ref{sigmataucommute}\,(1), we can 
show that $\varphi$ is invariant under $(\sigma'_s\circ\tau_s)_{s\in\mathbb{R}}$.  
Then using the result, we can show that $\sigma'$ and $\tau$ commute.  Use 
Propositions~\ref{nuphipsi}, \ref{Deltasigma_t}, \ref{Deltasigma'_t}, \ref{Deltatau_t}, 
and imitate the proof of Proposition~\ref{sigmataucommute}\,(2). Alternatively, 
it can be obtained as a consequence of (1), by making use of the fact that 
$\sigma'=R\circ\sigma^{-1}\circ R$ and that $\tau$ and $R$ commute.

(3). By Propositions~\ref{Deltasigma_t} and \ref{DeltasigmaphiR_t}, and by using the 
results from (1), (2) above, we have:
\begin{align}
\Delta(\sigma_t\circ\sigma'_s)&=(\tau_t\otimes\sigma_t)\bigl(\Delta(\sigma'_s)\bigr)
=\bigl((\tau_t\circ\sigma'_s)\otimes(\sigma_t\circ\tau_{-s})\bigr)\circ\Delta \notag \\
&=\bigl((\sigma'_s\circ\tau_t)\otimes(\tau_{-s}\circ\sigma_t)\bigr)\circ\Delta
=(\sigma'_s\otimes\tau_{-s})\bigl(\Delta(\sigma_t)\bigr)
=\Delta(\sigma'_s\circ\sigma_t).
\notag
\end{align}
It follows that $\sigma_t\circ\sigma'_s=\sigma'_s\circ\sigma_t$, which holds true for all $t$, $s$.
\end{proof}

Commutativity of the automorphism groups $\sigma$ and $\sigma'$ corresponds to the 
{\em quasi-invariance condition\/} that is typically required of the ordinary locally compact 
groupoid theory \cite{Renbook}, \cite{Patbook}.  This is assumed also in the framework of 
measured quantum groupoids \cite{LesSMF}, \cite{EnSMF}. Some discussion on this was 
given in Introduction and section~4 of Part~I \cite{BJKVD_qgroupoid1}.  See also Remark 
given below.

\begin{rem}
From the commutativity of $\sigma$ and $\sigma'$, and by using a result by Vaes \cite{VaRN}, 
one can define the Radon--Nikodym derivative, and this in turn will allow us to define the 
{\em modular element\/} $\delta$.  It would be a positive unbounded operator affiliated to $A$ 
such that $\psi=\varphi\circ R=\varphi(\delta^{\frac12}\,\cdot\,\delta^{\frac12})$.  In addition to 
providing us with some useful results relating $\varphi$ and $\psi$, or $\sigma$ and $\sigma'$, 
the modular element allows a natural way of identifying the Hilbert spaces ${\mathcal H}_{\psi}$ 
and ${\mathcal H}_{\varphi}$.  This is similar to the quantum group case.  In this way, we will 
be able to work conveniently within the setting of one single Hilbert space ${\mathcal H}$. 
These discussions will be made more precise in Part~III \cite{BJKVD_qgroupoid3}, 
when we explore the duality issues.
\end{rem}

\bigskip



\providecommand{\bysame}{\leavevmode\hbox to3em{\hrulefill}\thinspace}
\providecommand{\MR}{\relax\ifhmode\unskip\space\fi MR }
\providecommand{\MRhref}[2]{%
  \href{http://www.ams.org/mathscinet-getitem?mr=#1}{#2}
}
\providecommand{\href}[2]{#2}

\end{document}